\documentclass[a4paper,11pt,reqno,twoside]{article}
\usepackage{amsmath,amsthm,amsfonts,amssymb,amscd,mathrsfs}
\usepackage[usenames,dvipsnames]{color} 
\usepackage[colorlinks,linkcolor=BlueViolet,citecolor=RedViolet,pdftex]{hyperref}
\usepackage{array} 
\usepackage{longtable} 
\usepackage{xfrac} 
\usepackage{verbatim} 
\usepackage{listings} 
\usepackage{bbm} 
\usepackage{tikz} 
\usepackage{graphicx}
\usepackage{MnSymbol} 
\usepackage{esvect} 
\usepackage{upgreek} 
\usepackage{yfonts} 
\usepackage{relsize} 
\usepackage[top=2cm, bottom=2cm, left=2cm, right=2cm]{geometry}
\usepackage{titlesec} 
\usepackage{caption} 
\usepackage{shadethm} 
\usepackage[tikz]{bclogo} 
\usepackage{fancyhdr} 

\def\vvs{\vskip 1cm}
\def\vs{\vskip 0.30cm}
\def\vss{\vskip 0.05cm}

\def\mhs{\hspace{-0.2cm}}
\def\mhss{\hspace{-0.05cm}}
\def\mvs{\vspace{-0.2cm}}

\def\medmu{\medmuskip=0mu}

\def\al{\upalpha}
\def\be{\upbeta}
\def\ga{\upgamma}
\def\De{\Updelta}

\def\ep{\upvarepsilon}
\def\eps{\upepsilon}
\def\la{\uplambda}
\def\si{\upsigma}

\def\longto{\longrightarrow}

\def\Sum{\mathlarger{\sum}}
\def\into{\hookrightarrow}
\def\onto{\twoheadrightarrow}
\newcommand{\dc}{\vartriangle}
\newcommand{\inv}{^{-1}}

\def\ie{that is }

\def\cf{\textit{cf.\ }}
\newcommand{\deff}[1]{\textbf{#1}}

\newcommand{\A}{\mathcal{A}}
\newcommand{\B}{\mathcal{B}}
\newcommand{\CC}{\mathbb{C}}
\newcommand{\dd}{\mathbbm{d}}
\newcommand{\F}{\mathcal{F}}
\newcommand{\Fm}{\mathcal{F}_{\textrm{mult}}}
\newcommand{\N}{\mathbb{N}}
\newcommand{\one}{\mathbbm{1}}
\newcommand{\PP}{\mathscr{P}}
\newcommand{\Q}{\mathbb{Q}}
\newcommand{\R}{\mathbb{R}}
\newcommand{\Z}{\mathbb{Z}}
\newcommand{\an}{A_n}
\newcommand{\sn}{S_n}
\newcommand{\snsn}{\sn \times \sn}
\newcommand{\slzz}{\mathrm{SL}_2(\Z)}
\newcommand{\glzz}{\mathrm{GL}_2(\Z)}
\newcommand{\gl}{\mathrm{GL}}

\newcommand{\mmod}{\textrm{ mod }}
\newcommand{\type}[1]{\mathsf{type}(#1)}
\newcommand{\flag}[1]{\mathsf{flag}(#1)}
\newcommand{\graph}[1]{\mathsf{Graph}(#1)}
\newcommand{\supp}{\mathrm{supp}}
\def\fix{\mathrm{fix}}
\def\id{\mathrm{id}}
\def\tr{\mathrm{tr}}

\newcommand{\group}[1]{\left\langle #1 \right\rangle}

\newcommand{\set}[1]{\left\{ #1\right\}}
\newcommand{\sett}[2]{\left\{ #1 \;\,\left|\;\; #2\right. \right\}}
\newcommand{\sign}[1]{\mathrm{sign}(#1)}
\newcommand{\drob}[2]{{\,\Large \sfrac{#1}{#2}}}

\theoremstyle{plain}
\definecolor{shadethmcolor}{cmyk}{.02,.02,0,0} 
\definecolor{shaderulecolor}{cmyk}{.75,.75,0,.5} 

\newshadetheorem{theo}{Theorem}
\newshadetheorem{theor}{Theorem}

\newshadetheorem{prop}{Proposition}
\newshadetheorem{lemm}{Lemma}
\newshadetheorem{corr}{Corollary}
\newshadetheorem{conj}{Conjecture}
\theoremstyle{remark}

\theoremstyle{definition}
\newtheorem*{defi}{Definition}

\def\onlyif{\noindent\tikz{\node[below,draw,rounded corners] at (0,0) {$\Longrightarrow$};} }
\def\ifonly{\noindent\tikz{\node[below,draw,rounded corners] at (0,0) {$\Longleftarrow$};} }

\newdimen\arrowsize 
\pgfarrowsdeclare{arcs}{arcs} 
{ 
\arrowsize=0.2pt 
\advance\arrowsize by .5\pgflinewidth 
\pgfarrowsleftextend{-4\arrowsize-.5\pgflinewidth} 
\pgfarrowsrightextend{.5\pgflinewidth} 
} 
{ 
\arrowsize=0.2pt 
\advance\arrowsize by .3\pgflinewidth 
\pgfsetdash{}{0pt} 
\pgfsetroundjoin 
\pgfsetroundcap 
\pgfpathmoveto{\pgfpoint{-17\arrowsize}{5\arrowsize}} 
\pgfpatharc{180}{270}{17\arrowsize and 5\arrowsize} 
\pgfusepathqstroke 
\pgfpathmoveto{\pgfpointorigin} 
\pgfpatharc{90}{180}{17\arrowsize and 5\arrowsize} 
\pgfusepathqstroke 
} 

\titleformat{\section}{\Large\bfseries\scshape}{\thesection}{0.5cm}{}
\titleformat{\subsection}{\Large\bfseries\color{Mahogany}}{\large\thesubsection}{0.4cm}{}
\titlespacing{\section}{0cm}{1.5cm}{0.5cm}
\titlespacing{\subsection}{0cm}{1cm}{0.5cm}

\begin{document}

\thispagestyle{empty}
\phantom{hi}
\vspace{4cm}
\renewcommand\bcStyleTitre[1]{\LARGE\textsc{#1}}
\begin{bclogo}[couleur=blue!20,arrondi=0.1,ombre=true,logo=$\phantom{123}$,barre=none]
{The probability of generating the symmetric \newline group with a commutator condition}
\end{bclogo}

\vspace{5cm}
\quad\textbf{\Large David Zmiaikou\footnote{Research partially supported by the Balzan Research Project of J. Palis}}
\vs
\quad Department of Mathematics, University Paris-Sud, 91400 Orsay, France 

\vs\vvs
\quad August 31, 2012

\newpage
\thispagestyle{empty}
\begin{abstract}
Let $\B(n)$ be the set of pairs of permutations from the symmetric group of degree $n$ with a $3$-cycle commutator, and let $\A(n)$ be the set of those pairs which generate the symmetric or the alternating group of degree $n$.   
     We find effective formulas for calculating the cardinalities of both sets. More precisely, we show that $\#\B(n)/n!$ is a discrete convolution of the partition function and a linear combination of divisor functions, while $\#\A(n)/n!$ is the product of a polynomial and Jordan's totient function. In particular, it follows that the probability that a pair of random permutations with a $3$-cycle commutator generates the symmetric or the alternating group of degree $n$ tends to zero as $n$ tends to infinity, which makes a contrast with Dixon's classical result.

Key elements of our proofs are Jordan's theorem from the 19th century, a formula by Ramanujan from the 20th century and a technique of square-tiled surfaces developed by French mathematicians Leli\`{e}vre and Royer \cite{lelievre-royer} in the beginning of the 21st century. This paper uses and highlights elegant connections between algebra, geometry, and number theory.
\end{abstract}

\tableofcontents

\newpage
\setcounter{page}{2}
\addcontentsline{toc}{section}{Notation}
\section*{Notation}
\vspace{-.9cm}
\begin{longtable}{lp{12cm}}
$\N=\{1,2,3,\hdots \}$ & set of natural numbers (positive integers)\\
$\N_{\infty} = \N \cup \set{\infty}$ & extended set of positive integers\\
$\Z,\; \Q$ & sets of integer and rational numbers\\
$\R,\; \R^2$ & real line and real plane\\
$\left[a, b\right[$\,,\; $\left]a,b\right]$ & semi-open intervals of real numbers from $a$ to $b$\\
$C_{n}^{k} = \frac{n!}{k!(n-k)!}$ & binomial coefficient\\
$a\land b$ & greatest common divisor of $a$ and $b$\\
$a\mmod b$ & remainder (residue) after the division of $a$ by $b$\\$\lsem a, b\rsem$ & set of integers in an interval $[a,b]$\\
$\liminf a_{n}$,\; $\limsup a_{n}$ & lower and upper limits of a sequence $(a_{n})$\\
$\tr(A)$,\; $\det(A)$ & trace and determinant of a matrix $A$\\
$\slzz$ & group of integer $2\times 2$  matrices with determinant 1\\
$\gl_{d}(\CC)$ & group of complex $d\times d$ matrices  with nonzero determinant\\
$\upchi_{\rho}$ & character of a representation $\rho$\\
$f\cdot g$ & product of two functions, $(f\cdot g)(n) = f(n) g(n)$\\
$f\dc g$ & discrete convolution of two functions, $(f\dc g)(n) = \sum_{k=1}^{n-1} f(k) g(n-k)$\\
$f\ast g$ & Dirichlet convolution of two functions, $(f\ast g)(n) = \sum_{d\mid n} f(d) g(n/d)$\\
$\F$ & set of all arithmetic functions\\ 
$\F^{\ast}$ & set of arithmetic functions different from zero at $1$\\
$\Fm$ & set of multiplicative arithmetic functions\\
$\ep(n)$ & trivial function, $\ep(1) = 1$ and $\ep(n) = 0$  for $n>1$\\
$\one(n)$ & constant function, $\one(n)\equiv 1$\\
$\id(n)$ & identical function, $\id(n)=n$\\
$\id_{k}(n)$ & power function of order $k$, $\id_{k}(n)=n^{k}$\\
$\mu(n)$ & M\"{o}bius function\\
$\tau(n)$ & number of divisors  of $n$\\
$\si(n)$ & sum of divisors of $n$\\
$\si_{k}(n)$ & sigma-function of order $k$ (sum of $k^{\textrm{th}}$ powers of the divisors of $n$)\\
$\phi(n)$ & Euler's totient function (number of integers from $1$ to $n$  coprime with \nolinebreak $n$)\\
$J_{k}(n)$ & Jordan's totient function of order $k>0$\\

$P(n)$ & number of partitions of a positive integer $n$ (partition function)\\
$\PP(n)$ & set of partitions of $n$\\
$\#M$ or $|M|$ & cardinality of a set $M$\\
$M\into L$ & injection from a set $M$ to a set $L$\\
$M\onto L$ & surjection from a set $M$ to a set $L$\\
$M\times L$ & Cartesian product of sets $M$ and $L$\\
$s\inv$ & inverse permutation to $s$\\
$s\circ t$ or $s t$ & composition of two permutations $s$ and $t$\\
$[s, t] = s t s\inv t\inv$ & commutator of two permutations $s$ and $t$\\
$Alt(M)$,\; $Sym(M)$ & alternating and symmetric groups of a set $M$\\
$A_n$, $S_n$ & alternating and symmetric groups of the set $\set{1,2,\ldots,n}$\\
$\group{s,t}$ & group generated by $s$ and $t$\\
$G_x$ & stabilizer of a point $x$ for an action of a group $G$\\
$\supp(s)$,\; $\fix(s)$ & support of a permutation $s$ and set of points fixed by $s$\\
$\graph{s_{1},\ldots,s_{g}}$ & graph of permutations $s_{1},\ldots,s_{g}$ \\
$\type{s}$ & cycle type of a permutation $s$\\
$\flag{s}$ & flag of a permutationи $s$\\
$\dd_{s}: \lsem 1, n\rsem^{2} \to \N_{\infty}$ & $s$-distance on the set $\lsem 1, n\rsem$\,, where $s\in\sn$\\

$O_{s,t}$ & square-tiled surface corresponding to a pair of permutations $(s, t)$\\
\end{longtable}

\vvs
\pagestyle{fancy}
\lhead[\footnotesize\scshape \leftmark]{}
\rhead[]{\slshape \rightmark}

\addcontentsline{toc}{section}{Introduction}
\section*{Introduction}
\noindent In 1892, Netto conjectured that almost all pairs of permutations will generate the symmetric or the alternating group. 76 years passed before the theorem was finally proved by Dixon. This paper investigates a refinement of Netto's conjecture. 

Denote by $\B(n)$ the set of pairs of permutations from $\sn$ with a $3$-cycle commutator, and let $\A(n)$ be the set of those pairs which generate the symmetric or the alternating group of degree $n$:
$$\B(n) = \sett{(s,t)\in\snsn}{[s,t] \textrm{ is a } 3\textrm{-cycle}},$$
$$\A(n) = \sett{(s,t)\in\B(n)}{\group{s,t}=\an \textrm{ or } \sn}.$$
We find the cardinalities of these sets, and show that the probability that a pair of permutations with a $3$-cycle commutator generates $\an$ or $\sn$ tends to zero as $n\to\infty$.

\begin{theor}\label{theor:A} 
\emph{(\cf Theorems \ref{th:B}, \ref{th:A} and Corollary \ref{corr:A:proba})} 

\noindent The following relations hold for any $n>2$:
\begin{equation*}
   \#\B(n) = \frac{3}{8} \bigg(\sum_{k=1}^{n}\si_{3}(k) P(n-k) - 2\sum_{k=1}^{n} k\si_{1}(k) P(n-k) + n P(n)\bigg) \, n!\,,
\end{equation*}
\begin{equation*}
   \#\A(n) = \frac{3}{8} (n - 2)J_{2}(n) \, n!\,,
\end{equation*}
where $\si_{\ell}(n)$ is the sum of the $\ell^{\textrm{th}}$ powers of the positive divisors of $n$,\, $J_{\ell}(n)$ is Jordan's totient function of order $\ell$, and $P(n)$ is the partition function.

In particular, the probability $\drob{\#\A(n)}{\#\B(n)}$ tends to zero as $n$ tends to infinity.
\end{theor}

\noindent  We also consider intermediate sets
$$\B_{1}(n) = \sett{(s,t)\in\snsn}{[s,t] = 3\textrm{-cycle},\; s = n\textrm{-cycle}},$$
$$\A_{1}(n) = \sett{(s,t)\in\B_{1}(n)}{\group{s,t}=\an \textrm{ or } \sn},$$
$$\B_{2}(n) = \sett{(s,t)\in\snsn}{[s,t] =3\textrm{-cycle},\; s = \textrm{arbitrary cycle}},$$
$$\A_{2}(n) = \sett{(s,t)\in\B_{2}(n)}{\group{s,t}=\an \textrm{ or } \sn}$$
and prove the following theorem:

\begin{theor}\label{theor:B}  
\emph{(\cf Propositions \ref{prop:B1}, \ref{prop:A1}, \ref{prop:B2}, \ref{prop:A2} and Corollaries \ref{corr:limp1}, \ref{corr:limp2})}

\noindent For any natural $n>2$, one has the formulas:
\renewcommand{\arraystretch}{2} 
\begin{equation*}
\begin{array}{rl} 
\#\B_{1}(n) = C_{n}^{3}\, n!\,, 
&\quad  \#\B_{2}(n) = \dfrac{1}{24}(n-1)(n-2)(n^{2}+5n+12)\, n!\,,\\
\#\A_{1}(n) = \dfrac{n}{6}\bigg( J_{2}(n) - 3 J_{1}(n) \bigg)\, n!\,,  
&\quad \#\A_{2}(n) = \dfrac{n}{6}\bigg( J_{2}(n) - 3 J_{1}(n) \bigg)\, n! \;+\; \dfrac{(n+1)(n-2)}{2}\, n!\,.
\end{array}
\end{equation*}
In particular, for the probabilities $p_{1}(n)=\drob{\#\A_{1}(n)}{\#\B_{1}(n)}$\; and\; $p_{2}(n)=\drob{\#\A_{2}(n)}{\#\B_{2}(n)}$ we obtain the lower and the upper limits:
\renewcommand{\arraystretch}{1.1} 
\begin{equation*}
\begin{array}{rclrcl} 
 \liminf \,p_{1}(n) &=& \drob{6}{\pi^{2}},  &\quad  \liminf \;n\cdot p_{2}(n) &=& \drob{24}{\pi^{2}},\\ 
 \limsup \,p_{1}(n) &=&1  &\quad  \limsup \;n\cdot p_{2}(n) &=& 4. 
\end{array}
\end{equation*}
\renewcommand{\arraystretch}{1} 
\end{theor}

\section{General remarks}
Denote by $\lsem a, b\rsem$ the set of integer from $a$ to $b$. Introduce the extended set of positive integers $\N_{\infty} = \N \cup \set{\infty}$.
Recall that any permutation setа $\lsem 1, n\rsem$ decomposes into a product of disjoint cycles (\cf Lemma \ref{prop:cycles}). 
\begin{defi}
Let $s\in \sn$ be a permutation. The \deff{$s$-distance} on the set $\lsem 1, n\rsem$ is defined to be the function $\dd_{s}: \lsem 1, n\rsem^{2} \to \N_{\infty}$ such that
\begin{equation}
   \dd_{s}(x,y) = 
\begin{cases}
\min\sett{d\in\N}{s^{d}(x)=y}, & \textrm{if } x \textrm{ and } y \textrm{ lie in the same cycle of } s\,;\\
\infty, & \textrm{otherwise. }  
\end{cases}
\end{equation}
In particular, the $s$-distance $\dd_{s}(x,x)$ is equal to the length of the cycle containing $x$.
\end{defi}
For instance, for the permutation $s = (1\;2\;3\;4)(7\;8\;9)$ we have 
$$\dd_{s}(1,4)=3, \quad \dd_{s}(1,5)=\dd_{s}(1,8)=\infty, \quad \dd_{s}(6,6)=1, \quad \dd_{s}(7,8)=1, \quad \dd_{s}(9,9)=3.$$
Note that the function $\dd_{s}$ is not symmetric. In the example above, $\dd_{s}(1,4)=3$ but $\dd_{s}(4,1)=1$. Evidently, two integers $x$ and $y$ lie in the same cycle of $s$ if and only if 
$$\dd_{s}(x,y) + \dd_{s}(y,x) = \dd_{s}(x,x).$$

\begin{lemm}\label{lemm:twotypes} 
Suppose that the commutator of two permutations $s$ and $t$ from $\sn$ is a $3$-cycle,\; $[s,t]=(z\;y\;x)$. Then one of the following holds:
\begin{itemize}
   \item[\emph{(a)}]  all three numbers $x, y, z$ lie in the same cycle of $s$ in the given order, \ie  $$\dd_{s}(x,y) + \dd_{s}(y,z) + \dd_{s}(z,x) = \dd_{s}(x,x);$$ 
   \item[\emph{(b)}]  two numbers \emph{(}say $x$ and $y$\emph{)} lie in the same cycle of $s$, and the third one \emph{(}$z$\emph{)} lies in another cycle such that \; $\dd_{s}(z,z)=\dd_{s}(y,x)$.
\end{itemize}
\end{lemm}
\begin{proof}
The equality $s t s\inv t\inv=(z\;y\;x)$ can be re-written as 
\begin{equation}\label{eq:twotypes:conjugate}
   (x\;y\;z) s = t s t\inv
\end{equation}
when one multiplies the former equality by the permutation $(t s\inv t\inv)\inv = t s t\inv$ from the right and by $(z\;y\;x)\inv = (x\;y\;z)$ from the left. From this relation follows that the permutations $(x\;y\;z) s$ and $s$ are conjugate -- they have the same cycle type (\cf the section \ref{sec:permutations} of Appendix). 

Suppose that the numbers $x=x_{1}$, $y=y_{1}$, $z=z_{1}$ lie in distinct cycles of $s$, \ie  that $s$ is the following product of permutation with disjoint supports: 
$$s=(x_{1}\;\ldots\;x_{a})(y_{1}\;\ldots\;y_{b})(z_{1}\;\ldots\;z_{c})\, u\, ,$$ 
Then the permutation 
$$(x_{1}\;y_{1}\;z_{1}) s = (x_{1}\;\ldots\;x_{a}\;y_{1}\;\ldots\;y_{b}\;z_{1}\;\ldots\;z_{c})\, u$$
is not conjugate to $s$, which contradicts the condition \eqref{eq:twotypes:conjugate}. Thus, either all three numbers $x, y, z$ lie in the same cycle of $s$, or two of them lie in one cycle of the permutation $s$ and the third number -- in another cycle.

\vs Consider the case (a). Up to a cyclic permutation, the set of three elements can be ordered in two ways: $x, y, z$ or $z, y, x$. If the numbers $x, y, z$ lie in some cycleе $s$ in the inverse order, \ie if
$$s = (z_{1}\;\ldots\;z_{c}\;y_{1}\;\ldots\;y_{b}\;x_{1}\;\ldots\;x_{a})\, u\, ,$$
then the permutation  
$$(x_{1}\;y_{1}\;z_{1}) s = (x_{1}\;\ldots\;x_{a})(y_{1}\;\ldots\;y_{b})(z_{1}\;\ldots\;z_{c})\, u$$
is not conjugate to $s$, which is false. Therefore, the numbers $x, y, z$ must be in the given order. 

\vs Consider the case (b), where $s$ can be presented as  the following product of permutations with disjoint supports: 
$$s = (z_{1}\;\ldots\;z_{c})(y_{1}\;\ldots\;y_{b}\;x_{1}\;\ldots\;x_{a})\, u\, .$$
Then, since the permutation 
$$(x_{1}\;y_{1}\;z_{1}) s = (y_{1}\;\ldots\;y_{b})(z_{1}\;\ldots\;z_{c}\;x_{1}\;\ldots\;x_{a})\, u$$
must be conjugate to $s$, we obtain $c = b$,  \ie \; $\dd_{s}(z,z)=\dd_{s}(y,x)$.
\end{proof}

Consider a pair of permutations $(s,t)\in\snsn$ such that $[s,t]=(z_{1}\;y_{1}\;x_{1})$. Suppose that the permutation group $G=\group{s,t}$ is transitive. Then by means of Lemma \ref{lemm:twotypes}, one can get a simple description of the \emph{connected} square-tiled surface $O(s,t)$ defined in the section \ref{sec:origamis} of Appendix. It was first obtained by Zorich (\cf the monograph \cite{zorich}). More precisely, there can be only two following situations:
\begin{itemize}
\item[(a)] The permutation $s$ is a product of $k$ disjoint cycles of length $\drob{n}{k}$, 
$$s=u_{1} u_{2}\cdots u_{k}\, ,$$
moreover, one of the cycles, say $u_{1}$, is presentable as $u_{1}=(x_{1}\;\ldots\;x_{a}\;y_{1}\;\ldots\;y_{b}\;z_{1}\;\ldots\;z_{c})$ so that
$$(x_{1}\;y_{1}\;z_{1}) s = (x_{1}\;\ldots\;x_{a}\;z_{1}\;\ldots\;z_{c}\;y_{1}\;\ldots\;y_{b})\, u_{2}\cdots u_{k}.$$
The permutation $t$ is such that, when conjugating $s$, it sends the cycle $u_{i}$ to the cycle $u_{i+1}$ for any $i\in\lsem 1, k-1\rsem$, and the cycle $u_{k}$ is sent to $(x_{1}\;\ldots\;x_{a}\;z_{1}\;\ldots\;z_{c}\;y_{1}\;\ldots\;y_{b})$.
\begin{figure}[htbp]
   \begin{center}
   \begin{tikzpicture}[scale=.6, >=latex,shorten >=1pt,shorten <=1pt,line width=1pt,bend angle=20, 
   	circ/.style={circle,inner sep=2pt,draw=black!40,rounded corners, text centered},
        tblue/.style={dashed,blue,rounded corners}]
      \draw[step=1] (0,0) grid +(11,4);
      \draw (0,0)+(.5,.5) node {$x_{1}$};
      \draw (1,0)+(.5,.5) node {$\ldots$};
      \draw (2,0)+(.5,.5) node {$x_{a}$};
      \draw (3,0)+(.5,.5) node {$y_{1}$};
      \draw (5,0)+(.5,.5) node {$\ldots$};
      \draw (7,0)+(.5,.5) node {$y_{b}$};
      \draw (8,0)+(.5,.5) node {$z_{1}$};
      \draw (9,0)+(.5,.5) node {$\ldots$};
      \draw (10,0)+(.5,.5) node {$z_{c}$};
      
      \draw[line width=2pt,blue] (3,0) -- ++(0,4) (8,0) -- ++(0,4);
      
      \draw[tblue] (3,4) -- ++(0,.5) -- ++(5,0) -- ++(0,-.5);
      \draw[tblue] (6,0) -- ++(0,-.5) -- ++(5,0) -- ++(0,.5);
      \draw[tblue] (8,4) -- ++(0,.5) -- ++(3,0) -- ++(0,-.5);
      \draw[tblue] (3,0) -- ++(0,-.5) -- ++(3,0) -- ++(0,.5);
      \draw[->,tblue] (5.5,4.5) -- ++(0,2) -- ++(8,0) -- ++(0,-8) -- ++(-5,0) -- ++(0,1); 
      \draw[->,tblue] (9.5,4.5) -- ++(0,1) -- ++(2.5,0) -- ++(0,-8) -- ++(-7.5,0) -- ++(0,2); 

      \begin{scope}[scale=0.6,xshift=-10cm,yshift=-3cm]
      \draw[->,green!60!black] (0,0) -- (3,0);
      \draw[->,dashed,blue] (0,0) -- (0,3);
      \draw (3,0) node[below] {$s$};
      \draw (0,3) node[left] {$t$};
      \end{scope}
      
      \draw (18,0) node{};
      
      \draw[thin] (0,0) -- ++(-.7,0) (0,4) -- ++(-.7,0);
      \draw[<->,thin,>=arcs] (-.5,0) -- node[left]{$k$} ++(0,4);
   \end{tikzpicture}
   \caption{One-cylinder origami with parameters $(k,a,b,c)$.}
   \label{fig:onecylinder}
   \end{center} \vspace{-0.3cm}
\end{figure}
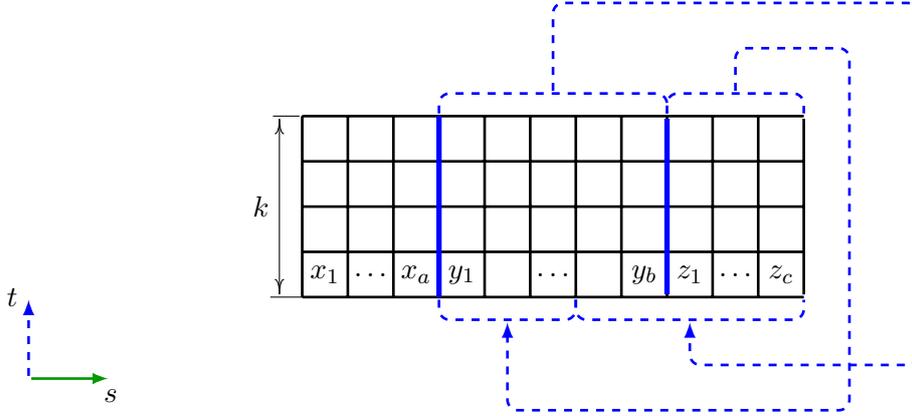

The square-tiled surface $O(s,t)$ is illustrated in Figure \ref{fig:onecylinder}, where any two unmarked opposite  sides are identified. It is called the \deff{one-cylinder origami with parameters $(k,a,b,c)$}.

\item[(b)] The permutation $s$ is a product of $a$ cycles of length $k$ and $b$ cycles of length $\ell$, 
$$s=u_{1}\cdots u_{a}\, v_{1}\cdots v_{b}\, ,\qquad \textrm{where }\; k < \ell\; \textrm{ and }\; a k + b \ell = n.$$
The cycle $u_{1}$ contains the point $z_{1}$, and the cycle $v_{1}$ contains the points $y_{1}$ and $x_{1}$. Moreover, the $s$-distance from $y_{1}$ to $x_{1}$ is equal to the length of the cycle $u_{1}$:
$$u_{1}=(z_{1}\;\ldots\;z_{k})\qquad\textrm{и}\qquad v_{1} = (y_{1}\;\ldots\;y_{k}\;x_{1}\;\ldots\;x_{k-\ell}),$$
so that 
$$(x_{1}\;y_{1}\;z_{1}) s = (y_{1}\;\ldots\;y_{k})\, u_{2}\cdots u_{a}\; (z_{1}\;\ldots\;z_{k}\; x_{1}\;\ldots\;x_{k-\ell})\, v_{2}\cdots v_{b}\, .$$
The permutation $t$ conjugates $s$ in the following way:
\renewcommand{\labelitemi}{$\bullet$}
\begin{itemize}
   \item[$\bullet$] the cycle $u_{i}$ is sent to the cycle $u_{i+1}$ for any $i\in\lsem 1, a-1\rsem$, 
   \item[$\bullet$] the cycle $v_{j}$ is sent to the cycle $v_{j+1}$ for any $j\in\lsem 1, b-1\rsem$, 
   \item[$\bullet$] the cycle $u_{a}$ is sent to the cycle $(y_{1}\;\ldots\;y_{k})$,
   \item[$\bullet$] the cycle $v_{b}$ is sent to the cycle $(z_{1}\;\ldots\;z_{k}\; x_{1}\;\ldots\;x_{k-\ell})$. 
\end{itemize}

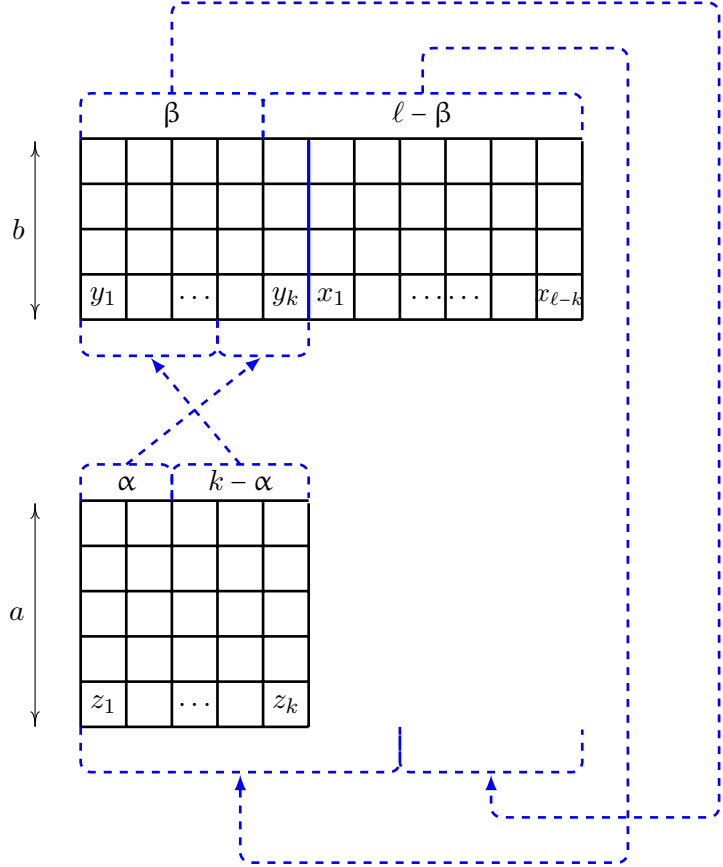
\begin{figure}[htbp]
   \begin{center}
   \begin{tikzpicture}[scale=.6, >=latex,shorten >=1pt,line width=1pt,bend angle=20,
        circ/.style={circle,inner sep=2pt,draw=black!40,rounded corners, text centered},
        tblue/.style={dashed,blue,rounded corners}]
      \draw[step=1] (3,0) grid +(11,4);
      \draw (3,0)+(.5,.5) node {$y_{1}$};
      \draw (5,0)+(.5,.5) node {$\ldots$};
      \draw (7,0)+(.5,.5) node {$y_{k}$};
      \draw (8,0)+(.5,.5) node {$x_{1}$};
      \draw (11,0)+(0,.5) node {$\ldots\ldots$};
      \draw (13,0)+(.5,.5) node {\small $x_{\ell-k}$};
      
      \draw[step=1] (3,-9) grid +(5,5);
      \draw (3,-9)+(.5,.5) node {$z_{1}$};
      \draw (5,-9)+(.5,.5) node {$\ldots$};
      \draw (7,-9)+(.5,.5) node {$z_{k}$};

      \draw[very thick,blue] (8,0) -- ++(0,4);
      
      \draw[tblue] (3,4) -- ++(0,1) -- ++(4,0) -- ++(0,-1);
      \draw[tblue] (10,-9) -- ++(0,-1) -- ++(4,0) -- ++(0,1);
      \draw[tblue] (7,4) -- ++(0,1) -- ++(7,0) -- ++(0,-1);
      \draw[tblue] (3,-9) -- ++(0,-1) -- ++(7,0) -- ++(0,1);
      \draw[->,tblue] (5,5) -- ++(0,2) -- ++(12,0) -- ++(0,-18) -- ++(-5,0) -- ++(0,1); 
      \draw[->,tblue] (10.5,5) -- ++(0,1) -- ++(4.5,0) -- ++(0,-18) -- ++(-8.5,0) -- ++(0,2); 

      \draw[tblue] (3,-4) -- ++(0,.8) -- ++(2,0) -- ++(0,-.8);
      \draw[tblue] (5,-4) -- ++(0,.8) -- ++(3,0) -- ++(0,-.8);
      \draw[tblue] (3,0) -- ++(0,-.8) -- ++(3,0) -- ++(0,+.8);
      \draw[tblue] (6,0) -- ++(0,-.8) -- ++(2,0) -- ++(0,+.8);
      \draw[->,tblue] (4,-4) ++(0,.8) -- (7,-.8);
      \draw[->,tblue] (6.5,-4) ++(0,.8) -- (4.5,-.8);

      \draw (4,-4) node[above] {$\al$};
      \draw (6.5,-4) node[above] {$k-\al$};
      \draw (5,4) node[above] {$\be$};
      \draw (10.5,4) node[above] {$\ell-\be$};
%
%
      \draw[<->,thin,>=arcs] (2,0) -- node[left]{$b$} ++(0,4);
      \draw[<->,thin,>=arcs] (2,-9) -- node[left]{$a$} ++(0,5);
   \end{tikzpicture}
   \caption{Two-cylinder origami with parameters $(a,b,k,\ell,\al,\be)$.}
   \label{fig:twocylinder}
   \end{center} \vspace{-0.3cm}
\end{figure}

The square-tiled surface $O(s,t)$ is illustrated in Figure \ref{fig:twocylinder}, where any two unmarked opposite  sides are identified. It has six parameters: 
$$
\deff{heights } a, b\in \N,\quad \deff{lengths } k<\ell\in \N,\quad \deff{twists } \al\in \lsem 0,k-1\rsem\; \textrm{ and }\; \be\in \lsem 0, \ell-1\rsem.
$$  
Such a surface is called \deff{two-cylinder with parameters $(a,b,k,\ell,\al,\be)$}.
\end{itemize}

\section{The case of a cycle of maximal length}
In this section, we will be interested in the following setа:
\setlength{\jot}{10pt} 
\begin{eqnarray*}
\A_{1}(n) &=& \sett{(s,t)\in\snsn}{[s,t] \textrm{ is a } 3\textrm{-cycle},\; s \textrm{ is an } n\textrm{-cycle},\; \group{s,t} = \an \textrm{ or } \sn},\\
\B_{1}(n) &=& \sett{(s,t)\in\snsn}{[s,t] \textrm{ is a } 3\textrm{-cycle},\; s \textrm{ is an } n\textrm{-cycle}}.
\end{eqnarray*}

\subsection{A formula for $\#\B_{1}(n)$}
\begin{prop}\label{prop:B1}  
The number of pairs of permutations $(s,t)\in \snsn$ such that the commutator $[s,t]$ is a $3$-cycle and the permutation $s$ is a cycle of length $n$, is equal to
\begin{equation}
   \#\B_{1}(n) = C_{n}^{3}\, n!
\end{equation}
\end{prop}
\begin{proof}
It follows from to Lemma \ref{lemm:twotypes}, that for an $n$-cycle $s$, there exists a permutation $t\in\sn$ with condition $[s,t]=(z\;y\;x)$ if and only if the points $x,y,z$ lie in the same cycle of $s$ in the given order. Such a permutation $t$ will be called an \deff{allowed permutation for $s$}, and such a triple $(x,y,z)\in\lsem 1,n\rsem^{3}$ will be said to be an \deff{allowed triple for $s$}. Herewith, we assume that the number $x$ is less than $y$ and $z$. It is clear that for any $n$-cycle $s$, the number of allowed triples is equal to the binomial coefficient $C_{n}^{3} = \frac{1}{6}n(n-1)(n-2)$.

Remark that in order to find the cardinality of the set $\B_{1}(n)$, it is sufficient to calculate the number of allowed permutations $t$ for some $n$-cycle $s$, say for $s_{0}=(1\;2\;\ldots\;n)$, and then multiply that number by the number of all $n$-cycles in $\sn$. 

Further, for a fixed $n$-cycle $s_{0}$ and a fixed allowed triple $(x,y,z)$, the number of all $t$ such that
$$[s_{0},t] = (z\;y\;x), \quad \textrm{\ie}\quad t s_{0} t\inv = (x\;y\;z) s_{0},$$
is equal to the order of the centralizer $C_{\sn}(s_{0})$ of the permutation $s_{0}$ in $\sn$. Indeed, the group $\sn$ acts on its element via conjugation. As the permutations $(x\;y\;z) s_{0}$ and $s_{0}$ lie in the same orbit (they are both $n$-cycles), then the number of distinct $t$ sending $s_{0}$ to $(x\;y\;z) s_{0}$ is equal to the number of distinct $t$ which stabilize the cycle $s_{0}$\,, \ie  $|C_{\sn}(s_{0})|$. 

We obtain
\setlength{\jot}{8pt} 
\begin{eqnarray*}
\#\B_{1}(n) &=& \#\sett{(s,t)\in\snsn}{[s,t] = 3\textrm{-cycle},\; s = n\textrm{-cycle}}\\
&=& \#\sett{s\in\sn}{ s = n\textrm{-cycle} }\;\times\; \#\sett{t\in\sn}{t \textrm{ allowed for } s_{0}=(1\;2\;\ldots\;n)} \\
&=& \#\sett{s\in\sn}{ s = n\textrm{-cycle} }\;\times\; |C_{\sn}(s_{0})| \\
& &\qquad\qquad\qquad\times\; \#\sett{ (x,y,z)\in \lsem 1,n\rsem^{3} }{\textrm{the triple } (x,y,z) \textrm{ is allowed for } s_{0}} \\
&=& n! \times C_{n}^{3}.
\end{eqnarray*}
Here we used the fact that the set $\sett{s\in\sn}{ s = n\textrm{-cycle} }$ is the orbit of the permutation $s_{0}$ for the action of $\sn$ via conjugation, and the centralizer $C_{\sn}(s_{0})$ is the stabilizer of the permutation $s_{0}$. In particular, we have\quad $\#\sett{s\in\sn}{ s = n\textrm{-cycle} }\times |C_{\sn}(s_{0})| = |\sn| = n!$.
\end{proof}

\subsection{A formula for $\#\A_{1}(n)$}
Let us now find the cardinality of the set $\A_{1}(n)$.
\begin{lemm}\label{lemm:onecylinderprim} 
A one-cylinder square-tiled surface with parameters $(k,a, b, c)$ is primitive if and only if\; $k=1$\; and\; $a\land b\land c = 1$.
\end{lemm}
\begin{proof}
\onlyif Let $O(s,t)$ be a \emph{primitive} one-cylinder origami with parameters $(k,a, b, c)$ as  in Figure \ref{fig:onecylinder}. By definition, its monodromy group $G=\group{s,t}\subseteq \sn$ is primitive. The permutation $s$ is a product of $k$ disjoint cycles of length $\drob{n}{k}$, 
$$s=u_{1} u_{2}\cdots u_{k}\, .$$
Consider the set of points $\De$ which are moved by the cycle $u_{1}$\,:
$$\De = \supp\; u_{1}\, .$$
Let us prove that it is a block for the group $G$. Indeed, any permutation $w\in G$ can be presented as a word in the alphabet $\set{s,t,s\inv,t\inv}$\,:
$$w = s^{\ep_{1}} t^{\ep_{2}}\ldots s^{\ep_{\ell-1}}t^{\ep_{\ell}},\qquad \textrm{where }\; \ep_{i}\in \lsem -1, 1\rsem,$$
since the group $G$ is generated by the pair of permutations $s$ and $t$. As 
$$s(\supp\; u_j) = \supp\; u_{j}\qquad\textrm{and}\qquad t(\supp\; u_{j}) = \supp\; u_{j+1\mmod k}$$
for each $j\in \lsem 1, k\rsem$, then the image of the set $\De$ under the action of $w$ is equal to the support of one of the permutations $u_{1}, \ldots, u_{k}$\,. In particular,
$$w(\De) = \De\quad \textrm{or}\quad w(\De)\cap \De=\emptyset\,,$$ 
from where $\De$ is a block for the group $G$. In view of the primitivity of $G$\,, we get that $\De = \lsem 1, n\rsem$, \ie $k=1$ and the permutation $s$ is an $n$-cycle (\cf Figure \ref{fig:onecylinder1}):
$$s = (x_{1}\;\ldots\;x_{a}\;y_{1}\;\ldots\;y_{b}\;z_{1}\;\ldots\;z_{c})\,.$$

Let us now show that the greatest common divisor  \;$d = a\land b\land c$\; equals $1$. Consider the set 
$$\De = \set{x_{d},x_{2d}, \ldots, x_{a},\; y_{d}, y_{2d},\ldots, y_{b},\; z_{d}, z_{2d},\ldots, z_{c}}\,,$$
which is subsequence of numbers in the cycle $s$ with step $d$. Denote by  
$$\De_{j} = \set{x_{d+j},x_{2d+j}, \ldots, x_{a+j}=x_{j},\;\; y_{d+j}, y_{2d+j},\ldots, y_{b+j}=y_{j},\;\; z_{d+j}, z_{2d+j},\dots, z_{c+j}=z_{j}}$$
the $j^{\textrm{th}}$ shift of the set $\De$\,, where $j\in \lsem 1, d\rsem$. Then 
$$s(\De_j) = \De_{j+1\mmod d}\quad\textrm{and}\quad t(\De_{j}) = \De_{j}\,.$$
Hence, the image of the set $\De$ for the action of any permutationи $w\in G$ (which can be expressed as a word on $s$ and $t$) coincides with one of the shifts of $\De$. In particular, $w(\De) = \De\; \textrm{or}\; w(\De)\cap \De=\emptyset\,,$ \ie  $\De$ is a block for the group $G$. Since $G$ is primitive and $\De$ contains more than one element, then $\De = \lsem 1, n\rsem$, from where $d=1$.

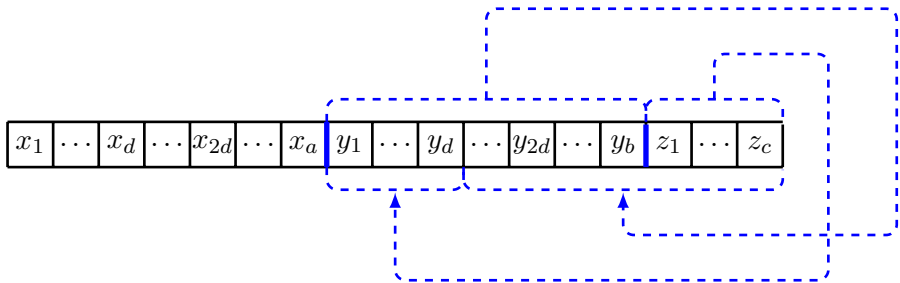
\begin{figure}[htbp]
   \begin{center}
   \begin{tikzpicture}[scale=.6, >=latex,shorten >=1pt,line width=1pt,bend angle=20,
        circ/.style={circle,inner sep=2pt,draw=black!40,rounded corners, text centered},
        tblue/.style={dashed,blue,rounded corners}]
      \draw[step=1] (-6,0) grid +(17,1);
      \draw (-6,0)+(.5,.5) node {$x_{1}$};
      \draw (-5,0)+(.5,.5) node {$\ldots$};
      \draw (-4,0)+(.5,.5) node {$x_{d}$};
      \draw (-3,0)+(.5,.5) node {$\ldots$};
      \draw (-2,0)+(.5,.5) node {$x_{2d}$};
      \draw (-1,0)+(.5,.5) node {$\ldots$};
      \draw (0,0)+(.5,.5) node {$x_{a}$};
      \draw (1,0)+(.5,.5) node {$y_{1}$};
      \draw (2,0)+(.5,.5) node {$\ldots$};
      \draw (3,0)+(.5,.5) node {$y_{d}$};
      \draw (4,0)+(.5,.5) node {$\ldots$};
      \draw (5,0)+(.5,.5) node {$y_{2d}$};
      \draw (6,0)+(.5,.5) node {$\ldots$};
      \draw (7,0)+(.5,.5) node {$y_{b}$};
      \draw (8,0)+(.5,.5) node {$z_{1}$};
      \draw (9,0)+(.5,.5) node {$\ldots$};
      \draw (10,0)+(.5,.5) node {$z_{c}$};
      
      \draw[line width=2pt,blue] (1,0) -- ++(0,1) (8,0) -- ++(0,1);
      
      \draw[tblue] (1,1) -- ++(0,.5) -- ++(7,0) -- ++(0,-.5);
      \draw[tblue] (4,0) -- ++(0,-.5) -- ++(7,0) -- ++(0,.5);
      \draw[tblue] (8,1) -- ++(0,.5) -- ++(3,0) -- ++(0,-.5);
      \draw[tblue] (1,0) -- ++(0,-.5) -- ++(3,0) -- ++(0,.5);
      \draw[->,tblue] (4.5,1.5) -- ++(0,2) -- ++(9,0) -- ++(0,-5) -- ++(-6,0) -- ++(0,1); 
      \draw[->,tblue] (9.5,1.5) -- ++(0,1) -- ++(2.5,0) -- ++(0,-5) -- ++(-9.5,0) -- ++(0,2); 
   \end{tikzpicture}
   \caption{One-cylinder origami with parameters $(1,a,b,c)$.}
   \label{fig:onecylinder1}
   \end{center} 
\end{figure}

\vs\ifonly Conversely, consider a one-cylinder origami $O(s,t)$  with parameters $(1,a, b, c)$, where
$$s = (x_{1}\;\ldots\;x_{a}\;y_{1}\;\ldots\;y_{b}\;z_{1}\;\ldots\;z_{c}), \quad [s,t] = (z_{1}\;y_{1}\;x_{1})\quad \textrm{and}\quad a\land b\land c=1.$$
Without loss of generality, we can assume that $s=(1\;2\;\ldots\;n)$, $x_{1}=1$, $y_{1}=a+1$ and $z_{1}=a+b+1$. Let us show that the permutation group $G=\group{s,t}$ is primitive. For this, consider an arbitrary block $\De\subseteq\lsem 1,n\rsem$ of the group $G$ containing the point $1$ and another point. 
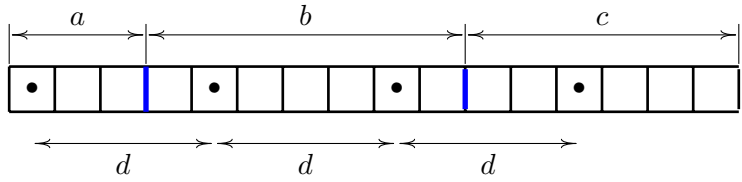
\begin{figure}[htbp]
   \begin{center}
   \begin{tikzpicture}[scale=.6, >=latex,shorten >=1pt,shorten <=1pt,line width=1pt,bend angle=20, circ/.style={circle,inner sep=2pt,draw=black!40,rounded corners, text centered}]
        
      	\draw[step=1] (0,0) grid +(16,1);
      	\foreach \x in {0,4,...,12}
      		\draw (\x,0)+(.5,.5) node {$\bullet$}; 
	
	\path[<->,thin,>=arcs]  (0,1.7) edge node[above] {$a$} ++ (3,0)
						(3,1.7)  edge node[above] {$b$} ++ (7,0)  
						(10,1.7) edge node[above] {$c$} ++ (6,0);
	\foreach \x in {0,3,10,16}
		\draw[thin] (\x,1) -- ++(0,1);
      	\draw[line width=2pt,blue] (3,0) -- ++(0,1) (10,0) -- ++(0,1);
	
	\foreach \x in {0,4,8}
		\path [<->,thin,>=arcs]  (\x,0)++(.5,-.7) edge node[below] {$d$} ++(4,0);
   \end{tikzpicture}
   \caption{The marked points belong to the block $\De$.}
   \label{fig:periodicblock}
   \end{center} 
\end{figure}
Denote by $d$ the least $s$-distance between the pairs of points from the block $\De$. Choose a pair of points $i$ and $i+d$ (mod $n$) from $\De$, for which such distance is reached. Since $i+d \in \De\cap s^{d}(\De)$, then by definition of a block we have $\De = s^{d}(\De)$, and so the point $i+2d$\, also lies in $\De$. By induction, we conclude that the block $\De$ contains the number $i+ \ell d$ (mod $n$) for any integer $\ell$. As $d$ is the least $s$-distance between points of the block and $1\in\De$, then $d$ divides $n$ and 
$$\De = \sett{1 + \ell d\mmod n}{\ell\in\Z}.$$

Let us also prove that $d$ divides the numbers $a$, $b$ and $c$. Indeed, divide $a$ and $a+b$ by $d$\,: 
$$a = \ell_{1}d+r_{1} \quad\textrm{and\quad} a+b = \ell_{2}d+r_{2}$$ 
with remainders $r_{1}, r_{2}\in \lsem 1, d\rsem$. Investigate the image $t(\De)$, \cf Figure \ref{fig:periodicblock1}. Since $1 \in \De\cap t(\De)$, then from the definition of a block follows that $t(\De) = \De$. 
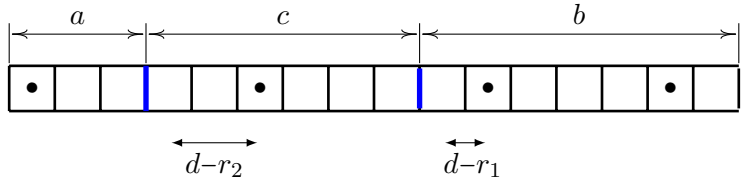
\begin{figure}[htbp]
   \begin{center}
   \begin{tikzpicture}[scale=.6, >=latex,shorten >=1pt,shorten <=1pt,line width=1pt,bend angle=20, circ/.style={circle,inner sep=2pt,draw=black!40,rounded corners, text centered}]
        
      	\draw[step=1] (0,0) grid +(16,1);
      	\foreach \x in {0,5,10,14}
      		\draw (\x,0)+(.5,.5) node {$\bullet$}; 
	
	\path[<->,thin,>=arcs]  (0,1.7) edge node[above] {$a$} ++ (3,0)
						(3,1.7)  edge node[above] {$c$} ++ (6,0)  
						(9,1.7) edge node[above] {$b$} ++ (7,0);
	\foreach \x in {0,3,9,16}
		\draw[thin] (\x,1) -- ++(0,1);
      	\draw[line width=2pt,blue] (3,0) -- ++(0,1) (9,0) -- ++(0,1);

	\path [<->,thin]  (3,0)++(.5,-.7) edge node[below] {$\medmu d-r_{2}$} ++(2,0);
	\path [<->,thin]  (9,0)++(.5,-.7) edge node[below,xshift=.1cm] {$\medmu d-r_{1}$} ++(1,0);
   \end{tikzpicture}
   \caption{The marked points belong to the block $t(\De)$.}
   \label{fig:periodicblock1}
   \end{center} \vspace{-0.3cm}
\end{figure}
Taking into account how the permutation $t$ acts on the set $\lsem 1,n\rsem$ and that $d$ divides $a+b+c$, we obtain 
\begin{align*}
d - r_{1} &= 0, \\
d - r_{2} &= d - r_{1}\,. 
\end{align*}
Hence, $r_{1}=r_{2}=d$, from where $d$ divides number $a$, $a+b$, $a+b+c$, and so the 
integers $a$, $b$, $c$ as well. The last three numbers are coprime, implying that $d=1$.
\end{proof}

\vs
\begin{lemm}\label{lemm:onecylinderprim1} 
Let $d$ be a natural divisor of $n$. The number of triples $(x, y, z)$ in the set $\lsem 1, n\rsem^{3}$ such that
$$x< y< z,\quad d\mid y-x\quad\textrm{and}\quad d\mid z-y\,,$$
is equal to\; $d\cdot C_{n/d}^{3}$\,.
\end{lemm}
\begin{proof} 
Since the differences $y-x$ and $z-y$ are divisible by $d$, the numbers $x, y, z$ are equal modulo $d$\,:
$$x= \al d + r, \quad y = \be d+r, \quad z = \ga d + r,$$
where $\al < \be < \ga\in \lsem 0, \drob{n}{d}-1\rsem$\; and\; $r\in \lsem 1,d\rsem$. There are exactly $C_{n/d}^{3}$ ways to choose a triple of coefficients $(\al, \be, \ga)$ and exactly $d$ ways to choose a remainder $r$, which proves the statement.
\end{proof}

\vs
\begin{lemm}\label{lemm:onecylinderprim2} 
For any positive integer $n>2$, the number of triples $x, y, z$ in the set $\lsem 1, n\rsem$ such that
$$x< y< z\quad\textrm{ and }\quad (y-x)\land (z-y)\land n=1\,,$$
equals\; $\frac{n}{6}J_{2}(n) - \frac{n}{2} J_{1}(n)$\,, where $J_{k}(n)$ is Jordan's totient function of order $k$\,.
\end{lemm}
\begin{proof} 
We shall proceed by the inclusion-exclusion principle. There are $C_{n}^{3}$ triples $x< y< z$ in the interval $\lsem 1, n\rsem$. From this number one must subtract the number of triples such that $y-x$ and $z-y$ are simultaneously divisible by some prime divisor of $n$. After that, one must add the number of those triples, for which $y-x$ and $z-y$ are simultaneously divisible by some two \emph{distinct} prime divisors of $n$, and so on. In the end, we will get
\vs
$\#\sett{(x,y,z)\in\lsem 1, n\rsem^{3}}{x<y<z,\;  (y-x)\land (z-y)\land n=1}$
\setlength{\jot}{8pt} 
\begin{eqnarray*}
 &=& C_{n}^{3} \hss - \sum_{\substack{p | n\\ p\;\textrm{prime}}}\mhs \sett{(x,y,z)\in\lsem 1, n\rsem^{3}}{ x<y<z,\; p \textrm{ divides } y-x\textrm{ and } z-y}\\
 & & \hspace{.55cm} +\mhs \sum_{\substack{p_{1} p_{2} | n\\ p_{1}\neq p_{2}\;\textrm{prime}}}\mhs\mhs \sett{(x,y,z)\in\lsem 1, n\rsem^{3}}{ x<y<z,\; p_{1} p_{2} \textrm{ divides } y-x\textrm{ and } z-y}\\
\end{eqnarray*}
\begin{eqnarray*}
 & & \hspace{.55cm} -\mhs\mhs \sum_{\substack{p_{1} p_{2} p_{3} | n\\ p_{1}\neq p_{2}\neq p_{3}\;\textrm{prime}}}\mhs\mhs\mhs \sett{(x,y,z)\in\lsem 1, n\rsem^{3}}{ x<y<z,\; p_{1} p_{2} p_{3} \textrm{ divides } y-x\textrm{ and } z-y}\\
 & & \hspace{.55cm} +\hspace{.5cm} \ldots\\
 &=& \sum_{d | n} \mu(d) \sett{(x,y,z)\in\lsem 1, n\rsem^{3}}{ x<y<z,\;\; d \textrm{ divides } y-x\textrm{ and } z-y}.
\end{eqnarray*}
Due to Lemma \ref{lemm:onecylinderprim1}, the last sum is equal to
$$\sum_{d | n} \mu(d) d\,C_{\drob{n}{d}}^{3} \;=\; \frac{1}{6} \sum_{d | n} \mu(d)\, n\left(\frac{n}{d} - 1\right) \left( \frac{n}{d} - 2\right) \;=\; \frac{\,n^{3}}{6} \sum_{d | n} \frac{\mu(d)}{d^{2}} 
\;-\; \frac{\,n^{2}}{2} \sum_{d | n} \frac{\mu(d)}{d} \;+\; \frac{n}{3} \sum_{d | n} \mu(d)\,,$$
and using the equality \eqref{eq:mud} for $k=0, 1, 2$, we obtain the expression
$$\frac{\,n^{3}}{6}\mhs \prod_{\substack{p | n\\ p\;\textrm{prime}}} \mhs\left(1 - \frac{1}{p^{2}}\right) 
\;-\; \frac{\,n^{2}}{2}\mhs \prod_{\substack{p | n\\ p\;\textrm{prime}}} \mhs\left(1 - \frac{1}{p}\right)
\;+\; \frac{n}{3}\times 0 \;=\; \frac{n}{6}J_{2}(n) - \frac{n}{2} J_{1}(n)\,,
$$
which completes the proof.
\end{proof}

\begin{prop}\label{prop:A1} 
The number of pairs of permutations $(s,t)\in \snsn$ with a $3$-cycle commutator, the permutation $s$ being an $n$-cycle and which generate $A_{n}$ or $S_{n}$\,, is equal to
\begin{equation}
   \#\A_{1}(n) = \frac{n}{6}\bigg( J_{2}(n) - 3 J_{1}(n) \bigg)\, n!
\end{equation}
\end{prop}
\begin{proof} 
We shall apply Jordan's theorem (Proposition \ref{prop:jordan}), Lemma \ref{lemm:onecylinderprim}, Lemma \nolinebreak \ref{lemm:twotypes} and Lemma \ref{lemm:onecylinderprim2}. 

By Jordan's theorem, a pair of permutations $(s,t)\in\snsn$ with a $3$-cycle commutator 
\begin{equation}\label{eq:stzyx}
  [s,t] = (z\;y\;x), \quad \textrm{\ie}\quad t s t\inv = (x\;y\;z) s,
\end{equation}
generates the alternating or the symmetric group of degree $n$ if and only if the square-tiled surface $O(s,t)$ is primitive.  
If the permutation $s$ is an $n$-cycle, then the surface $O(s,t)$ is one-cylinder.  We conclude that according Lemma \ref{lemm:onecylinderprim}, the pair  $(s,t)$ belongs to the set $\A_{1}(n)$ if and only if the origami $O(s,t)$ is one-cylinder with parameters $(1,a, b, c)$, where $a\land b\land c = 1$.

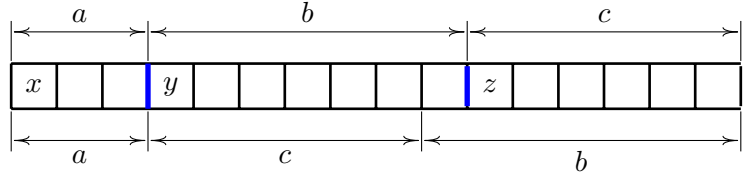
\begin{figure}[htbp]
   \begin{center}
   \begin{tikzpicture}[scale=.6, >=latex,shorten >=1pt,shorten <=1pt,line width=1pt,bend angle=20, circ/.style={circle,inner sep=2pt,draw=black!40,rounded corners, text centered}]
        
      	\draw[step=1] (0,0) grid +(16,1);
	\draw (0,0)+(.5,.5) node {$x$}; 
	\draw (3,0)+(.5,.5) node {$y$}; 
	\draw (10,0)+(.5,.5) node {$z$}; 
	
	\path[<->,thin,>=arcs]  (0,1.7) edge node[above] {$a$} ++ (3,0)
						(3,1.7)  edge node[above] {$b$} ++ (7,0)  
						(10,1.7) edge node[above] {$c$} ++ (6,0);
	\foreach \x in {0,3,10,16}
		\draw[thin] (\x,1) -- ++(0,1);
      	\draw[line width=2pt,blue] (3,0) -- ++(0,1) (10,0) -- ++(0,1);

	\path[<->,thin,>=arcs]  (0,-.7) edge node[below] {$a$} ++ (3,0)
						(3,-.7)  edge node[below] {$c$} ++ (6,0)  
						(9,-.7) edge node[below] {$b$} ++ (7,0);
	\foreach \x in {0,3,9,16}
		\draw[thin] (\x,0) -- ++(0,-1);
   \end{tikzpicture}
   \caption{A primitive one-cylinder origami with parameters $(1,a,b,c)$.}
   \label{fig:onecilinderprim}
   \end{center} \vspace{-0.3cm}
\end{figure}

A permutation $t\in\sn$ will be called \deff{allowed for $s$} if $(s,t)\in\A_{1}(n)$.
A triple  of integers $(x,y,z)$ from the set $\lsem1,n\rsem^{3}$ will be said to be \deff{allowed for $s$} if the relation \eqref{eq:stzyx} holds for some $t\in\sn$ and the pair $(s, t)$ generates $\an$ or $\sn$. Herewith, we assume the the number $x$ is less than $y$ and $z$. From what was explained above and from Lemma \ref{lemm:twotypes} follows that a triple $(x,y,z)$ is allowed for the permutation $s_{0}=(1\;2\;\ldots\;n)$ if and only if 
\begin{equation}\label{eq:1xyzn}
1\le x<y<z\le n\quad\textrm{ and }\quad (y-x)\land (z-y)\land (n-z+x)=1,
\end{equation}
since $\set{y-x, z-y, n-z+x}=\set{a,b,c}$. In Figure \ref{fig:onecilinderprim} we illustrated the case that $y-x=a$, $z-y=b$ and $n-z+x=c$. Since $n-z+x= n - (z-y) - (y-x)$, then any common divisor of  the numbers $y-x$, $z-y$, $n-z+x$ divides also $n$, and conversely -- any common divisor of $y-x$, $z-y$, $n$ divides $n-z+x$. Therefore, the condition \eqref{eq:1xyzn} can be re-written as
\begin{equation*}
1\le x<y<z\le n\quad\textrm{ and }\quad (y-x)\land (z-y)\land n=1.
\end{equation*}
The number of such triples has been found in Lemma \ref{lemm:onecylinderprim2}.

Analogically to the proof of Proposition \ref{prop:B1}, we obtain
\setlength{\jot}{8pt} 
\begin{eqnarray*}
\#\A_{1}(n) &=& \#\sett{(s,t)\in\snsn}{[s,t] = 3\textrm{-cycle},\; s = n\textrm{-cycle},\; \group{s,t}=\an \textrm{ or } \sn}\\
&=& \#\sett{s\in\sn}{ s = n\textrm{-cycle} }\;\times\; \#\sett{t\in\sn}{t \textrm{ is allowed for } s_{0}=(1\;2\;\ldots\;n)} \\
&=& \#\sett{s\in\sn}{ s = n\textrm{-cycle} }\;\times\; |C_{\sn}(s_{0})| \\
&&\qquad\qquad\qquad\times\; \#\sett{ (x,y,z)\in \lsem1,n\rsem^{3} }{\textrm{the triple } (x,y,z) \textrm{ is allowed for } s_{0}} \\
&=& n! \times \bigg( \frac{n}{6}J_{2}(n) - \frac{n}{2} J_{1}(n) \bigg),
\end{eqnarray*}
as required.
\end{proof}

\begin{corr}\label{corr:A1:gen} 
Let $p$ be a prime number. For any $p$-cycle $s\in S_{p}$ and any permutation $t\in S_{p}$ such that the commutator $[s,t]$ is $3$-cycle, the pair $(s,t)$ must generate $A_{p}$ or $S_{p}$. 
\end{corr}
\begin{proof}
Indeed, according to Propositionм \ref{prop:B1} and \ref{prop:A1}
\setlength{\jot}{8pt} 
\begin{eqnarray*}
\frac{\#\B(p)}{p!} &=& \frac{p(p-1)(p-2)}{6}\,,\\
\frac{\#\A(p)}{p!} &=& \frac{p}{6} \left( J_{2}(p) - 3J_{1}(p) \right) 
	\;=\; \frac{p}{6} \left( p^{2} - 1 - 3(p-1) \right) 
	\;=\; \frac{p(p-1)(p-2)}{6}\,,
\end{eqnarray*}
\ie the cardinalities of the sets $\A(p)$ and $\B(p)$ are equal. Since $\A(p)\subseteq \B(p)$, these sets coincide, as required.
\end{proof}

\begin{corr}\label{corr:limp1} 
Consider the pairs of permutations from $\sn$ with a $3$-cycle commutator and the first permutation being a cycle of length $n$. Denote by $p_{1}(n)$ the probability that such a pair generates $\an$ or $\sn$\,, \ie  $p_{1}(n) = \#\A_{1}(n) / \#\B_{1}(n)$. Then the lower and the upper limits of this probability as $n\to\infty$ are equal to
\begin{equation}\label{eq:limp1}
   \liminf \,p_{1}(n) = \frac{6}{\pi^{2}}\qquad\textrm{and}\qquad \limsup \,p_{1}(n) =1.
\end{equation}
\end{corr}
\begin{proof}
Due to Corollary \ref{corr:A1:gen}, we know that $\drob{\#\A_{1}(p)}{\#\B_{1}(p)} = 1$ for any prime $p$. Hence, $\limsup p_{1}(n) =1$.

From Propositions \ref{prop:B1} and \ref{prop:A1} follows that
$$
\frac{\#\A_{1}(n)}{\#\B_{1}(n)} = \frac{ \dfrac{n}{6}\big( J_{2}(n) - 3 J_{1}(n) \big)\, n! }{ C_{n}^{3}\, n! } = \frac{ J_{2}(n) - 3J_{1}(n) }{ (n-1)(n-2) } \; \underset{n\to\infty}{\mathlarger{\mathlarger{\sim}}} \; 
\frac{J_{2}(n)}{n^{2}} - \frac{3J_{1}(n)}{n^{2}}.
$$
Jordan's totient function $J_{1}(n)=\phi(n)$ is defined as the number of positive integers not greater than $n$ and coprime with $n$. Thus, the inequality $1\le J_{1}(n)\le n$ holds, from where
$$\lim_{n\to\infty}  \frac{ 3J_{1}(n) }{ n^{2} } = 0 \qquad\textrm{and}\qquad
\liminf p_{1}(n) = \liminf \frac{ J_{2}(n) }{ n^{2} }.$$
According to the formula \eqref{eq:Jk}
$$\frac{J_{2}(n)}{n^{2}} = \prod_{\substack{p | n\\ p\;\textrm{prime}}} \mhs\left(1-\frac{1}{p^{2}}\right),\quad \textrm{ from where }\quad  
\frac{J_{2}(n)}{n^{2}} \ge \prod_{\substack{\textrm{over all}\\ \textrm{prime }p}} \left(1-\frac{1}{p^{2}}\right).$$
We conclude that 
$$ \liminf\frac{J_{2}(n)}{n^{2}} \ge \prod_{\substack{\textrm{over all}\\ \textrm{prime }p}} \left(1-\frac{1}{p^{2}}\right). $$
Let us now show that the equality is achieved for some subsequence of the sequence $1,2,3,\ldots$ of positive integers. Let $p_{k}$ denote the $k^{\textrm{th}}$ prime number ($p_{1}=2$, $p_{2}=3$, $p_{3}=5$ and so on) and take $a_{k} = p_{1}p_{2}\cdots p_{k}$. Then
$$\lim_{k\to\infty}  \frac{ J_{2}(a_{k}) }{ a_{k}^{2} } \;=\; \lim_{k\to\infty} \prod_{i=1}^{k} \left(1-\frac{1}{p_{i}^{2}}\right) \;= \prod_{i=1}^{\infty} \left(1-\frac{1}{p_{i}^{2}}\right).$$
The infinite product that we got here is equal to $\drob{6}{\pi^{2}}$ by the formula \eqref{eq:6pi2}. Therefore, 
$$\liminf p_{1}(n) = \liminf \frac{ J_{2}(n) }{ n^{2} } =  \frac{6}{\pi^{2}}\,,$$
as required.
\end{proof}

\section{The case of an arbitrary cycle}
Introduce the notation:
\setlength{\jot}{10pt} 
\begin{eqnarray*}
\A_{2}(n) &=& \sett{(s,t)\in\snsn}{[s,t] \textrm{ is a } 3\textrm{-cycle},\; s \textrm{ is an arbitrary cycle},\; \group{s,t} = \an \textrm{ or } \sn},\\
\B_{2}(n) &=& \sett{(s,t)\in\snsn}{[s,t] \textrm{ is a } 3\textrm{-cycle},\; s \textrm{ is an arbitrary cycle}}.
\end{eqnarray*}

\subsection{A formula for $\#\B_{2}(n)$}
\begin{prop}\label{prop:B2}  
The number of pairs of permutations $(s,t)\in \snsn$ with a $3$-cycle commutator and  $s$ being an arbitrary cycle, is equal to
\begin{equation}
   \#\B_{2}(n) = \frac{1}{24}(n-1)(n-2)(n^{2}+5n+12)\, n!
\end{equation}
for any natural $n>2$.
\end{prop}
\begin{proof} 
We will need Lemma \nolinebreak \ref{lemm:twotypes} and the formulas \eqref{eq:knk} and \eqref{eq:kk1}. 

Consider an arbitrary $k$-cycle $s\in\sn$\,, where $k\in\lsem 2, n\rsem$. A permutation $t\in\sn$ will be called \deff{allowed for $s$} if 
\begin{equation}\label{eq:stzyx1}
  [s,t] = (z\;y\;x), \quad \textrm{\ie}\quad t s t\inv = (x\;y\;z) s.
\end{equation}
A triple $(x,y,z)$ of integers from the set $\lsem1,n\rsem^{3}$ will be called \deff{allowed for $s$} if the relation \eqref{eq:stzyx1} holds for some $t\in\sn$. Herewith, we assume that the number $z$ is \emph{greater} than $y$ and $x$. From Lemma \ref{lemm:twotypes} follows that a triple $(x,y,z)$ is allowed for the permutation $s_{0}=(1\;2\;\ldots\;k)$ if and only if one of the next  situations occurs:
\begin{itemize}
\item[(a)] The three numbers $x, y, z$ lie in the cycle $(1\;2\;\ldots\;k)$ in the given order, \ie $1\le x<y<z\le k$ (\cf Figure \ref{fig:onecilinderB2}). There are exactly $C_{k}^{3}$ such triples.
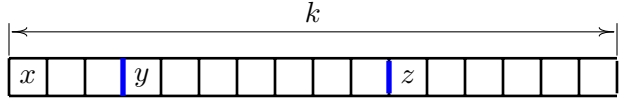
\begin{figure}[htbp]
   \begin{center}
   \begin{tikzpicture}[scale=.5, >=latex,shorten >=1pt,shorten <=1pt,line width=1pt,bend angle=20, circ/.style={circle,inner sep=2pt,draw=black!40,rounded corners, text centered}]
        
      	\draw[step=1] (0,0) grid +(16,1);
	\draw (0,0)+(.5,.5) node {$x$}; 
	\draw (3,0)+(.5,.5) node {$y$}; 
	\draw (10,0)+(.5,.5) node {$z$}; 
	
      	\draw[line width=2pt,blue] (3,0) -- ++(0,1) (10,0) -- ++(0,1);

	\path[<->,thin,>=arcs]  (0,+1.7) edge node[above] {$k$} ++ (16,0);
	\foreach \x in {0,16}
		\draw[thin] (\x,1) -- ++(0,1);
   \end{tikzpicture}
   \caption{An allowed triple of integers $1\le x<y<z\le k$.}
   \label{fig:onecilinderB2}
   \end{center} \vspace{-0.3cm}
\end{figure}

\item[(b)] Two numbers $x, y$ lie in the cycle $(1\;2\;\ldots\;k)$ and $z\in \lsem k+1,n\rsem$. Moreover, $x = y+1 \mmod k$ (\cf Figure \ref{fig:twocilinderB2}). There are exactly $k(n-k)$ such triples, since $y$  can be chosen in $k$ ways and $z$ can be chosen in $n-k$ ways.
\begin{figure}[htbp]
   \begin{center}
   \begin{tikzpicture}[scale=.5, >=latex,shorten >=1pt,shorten <=1pt,line width=1pt,bend angle=20, circ/.style={circle,inner sep=2pt,draw=black!40,rounded corners, text centered}]
        
      	\draw[step=1] (0,0) grid +(10,1);
      	\draw[step=1] (0,-6) grid +(1,6);
	\draw (0,0)+(.5,.4) node {$y$}; 
	\draw (1,0)+(.5,.5) node {$x$}; 
	\draw (0,-6)+(.5,.5) node {$z$}; 
	
      	\draw[line width=2pt,blue] (1,0) -- ++(0,1);

	\path[<->,thin,>=arcs]  (0,1.7) edge node[above] {$k$} ++ (10,0)
						(-.7,-6) edge node[left] {$\le n-k$} ++ (0,6);
	\foreach \x in {0,10}
		\draw[thin] (\x,1) -- ++(0,1);
	\foreach \y in {0,-6}
		\draw[thin] (0,\y) -- ++(-1,0);
   \end{tikzpicture}
   \caption{An allowed triple of integers $x, y\in \lsem 1, k\rsem$ and $z\in \lsem k+1,n\rsem$, where $x = y+1 \mmod k$.}
   \label{fig:twocilinderB2}
   \end{center} \vspace{-0.3cm}
\end{figure}
\end{itemize}

We conclude that there are $C_{k}^{3} + k(n-k)$ allowed triples for the $k$-cycle $s_{0}=(1\;2\;\ldots\;k)$.

By analogy with the proof of Proposition \ref{prop:B1}, we get
\setlength{\jot}{8pt} 
\begin{eqnarray*}
\#\B_{2}(n) &=& \sum_{k=2}^{n} \#\sett{(s,t)\in\snsn}{[s,t] = 3\textrm{-cycle},\; s = k\textrm{-cycle}}\\
&=& \sum_{k=2}^{n} \#\sett{s\in\sn}{ s = k\textrm{-cycle} }\;\times\; \#\sett{t\in\sn}{t \textrm{ is allowed for } s_{0}=(1\;2\;\ldots\;k)} \\
\end{eqnarray*}
\begin{eqnarray*}
\phantom{\#\B_{2}} &=& \sum_{k=2}^{n} \#\sett{s\in\sn}{ s = k\textrm{-cycle} }\;\times\; |C_{\sn}(s_{0})| \\
& &\qquad\qquad\qquad\times\; \#\sett{ (x,y,z)\in \lsem 1,n\rsem^{3} }{\textrm{the triple } (x,y,z) \textrm{ is allowed for } s_{0}} \\
&=& \sum_{k=2}^{n} n! \times \big(C_{k}^{3} + k(n-k)\big) = \dfrac{n!}{6} \sum_{k=2}^{n} k(k-1)(k-2) \;+\; n! \sum_{k=2}^{n} k(n-k).
\end{eqnarray*}
Finally, we apply the formula \eqref{eq:kk1} with $r=2$ and the formula \eqref{eq:knk}:
\begin{eqnarray*}
\#\B_{2}(n) &=&  \frac{n!}{6}\cdot \dfrac{(n+1)n(n-1)(n-2)}{4} \;+\; n!\cdot \dfrac{(n+3)(n-1)(n-2)}{6} \\
&=& \frac{n!}{24}(n-1)(n-2)\big( (n+1)n \;+\; 4(n+3)\big) \\
&=& \frac{1}{24}(n-1)(n-2)(n^{2}+5n+12)\, n!\,,
\end{eqnarray*}
as required.
\end{proof}

\subsection{A formula for $\#\A_{2}(n)$}
\begin{prop}\label{prop:A2}  
The number of pairs of permutations $(s,t)\in \snsn$ with a $3$-cycle commutator, the permutation $s$ being an arbitrary cycle and which generate $A_{n}$ or $S_{n}$\,, is equal to
\begin{equation}
   \#\A_{2}(n) =  \frac{n}{6}\bigg( J_{2}(n) - 3 J_{1}(n) \bigg)\, n! \;+\; \frac{(n+1)(n-2)}{2}\, n! 
\end{equation}
for any natural $n>2$.
\end{prop}
\begin{proof} 
We will use Proposition \ref{prop:A1}, Jordan's theorem (Proposition \ref{prop:jordan}), Lemma \nolinebreak \ref{lemm:twotypes} and Proposition \ref{prop:isaacs}.

From Proposition \ref{prop:A1} we know that, in the case where $s$ is an $n$-cycle, the number of required pairs is equal to 
$$\#\A_{1}(n) = \frac{n}{6}\bigg( J_{2}(n) - 3 J_{1}(n) \bigg)\, n!\,.$$

Consider now an arbitrary $\ell$-cycle $s\in\sn$ with $2\le \ell<n$. A permutation $t\in\sn$ will be called \deff{allowed for $s$} if the pair $(s,t)$ generates the group $\an$ or $\sn$ and the commutator $[s,t]$ is a $3$-cycle.

A triple of integers $(x,y,z)$ from the set $\lsem1,n\rsem^{3}$ will be called \deff{allowed for $s$} if there exists an allowed permutation $t\in\sn$ for which $[s,t] = (z\;y\;x)$.  Herewith, we assume that the number $z$ is \emph{greater} than $y$ and $x$. 

From Lemma \ref{lemm:twotypes} follows that a triple $(x,y,z)$ is allowed for the permutation $s_{0}=(1\;2\;\ldots\;\ell)$ if and only if  
 $x$ and $y$ lie in the cycle $(1\;2\;\ldots\;\ell)$\; and\; $z\in \lsem \ell+1,n\rsem$ with condition that $x = y+1 \mmod \ell$ (\cf Figure \ref{fig:twocilinderA2}). There are exactly $\ell(n-\ell)$ such triples, since $y$ can be chosen in $\ell$ ways  and $z$ can be chosen in $n-\ell$ ways.
\begin{figure}[htbp]
   \begin{center}
   \begin{tikzpicture}[scale=.5, >=latex,shorten >=1pt,shorten <=1pt,line width=1pt,bend angle=20, 
   		circ/.style={circle,inner sep=2pt,draw=black!40,rounded corners, text centered},
   		tblue/.style={dashed,blue,rounded corners}]
        
      	\draw[step=1] (0,0) grid +(10,1);
      	\draw[step=1] (0,-6) grid +(1,6);
	\draw (0,0)+(.5,.4) node {$y$}; 
	\draw (1,0)+(.5,.5) node {$x$}; 
	\draw (0,-6)+(.5,.5) node {$z$}; 
	
      	\draw[line width=2pt,blue] (1,0) -- ++(0,1);

	\path[<->,thin,>=arcs]  (0,1.7) edge node[above] {$\ell$} ++ (10,0)
						(-.7,-6) edge node[left] {$n-\ell$} ++ (0,6);
	\foreach \x in {0,10}
		\draw[thin] (\x,1) -- ++(0,1);
	\foreach \y in {0,-6}
		\draw[thin] (0,\y) -- ++(-1,0);

	\begin{scope}[xshift=13cm]        
      	\draw[step=1] (0,0) grid +(10,1);
      	\draw[step=1] (0,-6) grid +(1,6);
	\draw (0,0)+(.5,.4) node {$y$}; 
	\draw (1,0)+(.5,.5) node {$x$}; 
	\draw (0,-6)+(.5,.5) node {$z$}; 
	
      	\draw[line width=2pt,blue] (1,0) -- ++(0,1);

	\path[<->,thin,>=arcs,blue]  (0,1.7) edge node[above] {$\be$} ++ (4,0)
						(4,1.7) edge node[above] {$\ell-\be$} ++ (6,0)
						(0,-6.7) edge node[below] {$\ell-\be$} ++ (6,0)
						(6,-6.7) edge node[below] {$\be$} ++ (4,0);
	\foreach \x in {0,4,10}
		\draw[thin,blue] (\x,1) -- ++(0,1);
	\foreach \x in {0,6,10}
		\draw[thin,blue] (\x,-6) -- ++(0,-1);
   	\end{scope}
   \end{tikzpicture}

   \caption{An allowed triple of integers $x, y\in \lsem 1, \ell\rsem$ and $z\in \lsem \ell+1,n\rsem$, where $x = y+1$ mod $\ell$.}
   \label{fig:twocilinderA2}
   \end{center} \vspace{-0.3cm}
\end{figure}

For a fixed $\ell$-cycle $s_{0}\in\sn$ and a fixed triple $(x,y,z)\in\lsem 1,n \rsem^{3}$, let us find the number of allowed permutations $t\in\sn$ such that 
\begin{equation}\label{eq:stzyxA2}
  [s_{0},t] = (z\;y\;x), \quad \textrm{\ie}\quad t s_{0} t\inv = (x\;y\;z) s_{0}\,.
\end{equation}
Any such permutation $t$ can be expressed as
\begin{equation}\label{eq:tyz}
   t=(y\;z) \cdot u\cdot s_{0}^{\be},
\end{equation}
where $u$ is some $\medmu (n-\ell)$-cycle containing all integers from $\ell+1$ to $n$, and the parameter $\be\in\lsem 0,\ell-1\rsem$ is a twist (\cf also Figure.\ \ref{fig:twocylinder}). Indeed, since the group $\group{s,t}$ is transitive, then the two-cylinder origami $O(s,t)$ given in Figure \ref{fig:twocilinderA2} consists of $n$ squares, and so $t=(y\;z) \,u\, s_{0}^{\be}$. Conversely: if \,$t=(y\;z) \,u\, s_{0}^{\be}$\, then the pair $(s_{0},t)$ generates the group $\an$ or $\sn$\,. This follows from Proposition \ref{prop:isaacs}, since
$$s_{0} = (y+1\;\ldots\;\ell\;\;1\;\ldots\;y)\quad\textrm{ and }\quad  s_{0} t s_{0}^{-\be} = s_{0} (y\;z\;z_{2}\;\ldots\;z_{n-\ell}) = (y+1\;\ldots\;\ell\;\;1\;\ldots\;y\;z\;z_{2}\;\ldots\;z_{n-\ell}),$$
where $u = (z\;z_{2}\;\ldots\;z_{n-\ell})$. In other words, after renumbering the points from $1$ to $n$, we will get $s_{0}=(1\;2\;\dots\;\ell)$ and $t=(1\;2\;\dots\;n)$.

The number of permutations $t$ of  form \eqref{eq:tyz} is equal to 
$$\frac{(n-\ell)!}{n-\ell} \cdot \ell \,,$$
since there are $\frac{1}{n-\ell} \medmu (n-\ell)!$ ways to choose an $\medmu (n-\ell)$-cycle $u$ and $\ell$ ways to choose $\be$.

\vs
We conclude that
\setlength{\jot}{8pt} 
\begin{eqnarray*}
\#\A_{2}(n) &=& \#\A_{1}(n) + \sum_{\ell=2}^{n-1}\#\sett{(s,t)\in\snsn}{[s,t] = 3\textrm{-cycle},\; s = \ell\textrm{-cycle},\; \group{s,t}=\an \textrm{ or } \sn}\\
\phantom{\#\A_{2}(n)} &=&  \#\A_{1}(n) + \sum_{\ell=2}^{n-1}\#\sett{s\in\sn}{ s = \ell\textrm{-cycle} }\;\times\; \#\sett{t\in\sn}{t \textrm{ is allowed for } s_{0}=(1\;2\;\ldots\;\ell)} \\
\end{eqnarray*}
\vspace{-1cm}
\begin{eqnarray*}
\phantom{\#\A_{2}} &=& \#\A_{1}(n) + \sum_{\ell=2}^{n-1}\#\sett{s\in\sn}{ s = \ell\textrm{-cycle} }\;\times \\
&&\qquad\qquad\times\; \frac{(n-\ell)!}{n-\ell} \cdot \ell \;\times\; \#\sett{ (x,y,z)\in \lsem1,n\rsem^{3} }{\textrm{the triple } (x,y,z) \textrm{ is allowed for } s_{0}} \\
&=& \#\A_{1}(n) + \sum_{\ell=2}^{n-1} \frac{1}{\ell}\cdot\frac{n!}{(n-\ell)!} \;\times\; \frac{(n-\ell)!}{n-\ell} \cdot \ell \;\times\; \ell (n-\ell) \;=\;  \#\A_{1}(n) + \sum_{\ell=2}^{n-1} \ell n!\\
&=& \frac{n}{6}\bigg( J_{2}(n) - 3 J_{1}(n) \bigg)\, n! \;+\; \frac{(n+1)(n-2)}{2}\, n!\,,
\end{eqnarray*}
as required.
\end{proof}
\begin{corr}\label{corr:limp2} 
Consider the pairs of permutations from $\sn$ with a $3$-cycle commutator and the first permutation being an arbitrary cycle. Denote by $p_{2}(n)$ the probability that such a pair generates $\an$ or $\sn$\,, \ie  $p_{2}(n) = \#\A_{2}(n) / \#\B_{2}(n)$. Then one has the following lower and upper limits as $n\to\infty$
\begin{equation}\label{eq:limp1}
   \liminf \;n\cdot p_{2}(n) = \frac{24}{\pi^{2}}\qquad\textrm{and}\qquad \limsup \;n\cdot p_{2}(n) =4.
\end{equation}
\end{corr}
\begin{proof}
From Propositions \ref{prop:B2} and \ref{prop:A2} follows that
\setlength{\jot}{10pt} 
\begin{eqnarray*}
n\cdot \frac{\#\A_{2}(n)}{\#\B_{2}(n)} &=& n\cdot \frac{ \dfrac{n}{6}\bigg( J_{2}(n) - 3 J_{1}(n) \bigg)\, n! \;+\; \dfrac{(n+1)(n-2)}{2}\, n! }{ \dfrac{1}{24}(n-1)(n-2)(n^{2}+5n+12)\, n! } \\
&=&  \frac{ 4n^{2} J_{2}(n) - 12n^{2}J_{1}(n) + 12n(n+1)(n-2) }{ (n-1)(n-2)(n^{2}+5n+12) }
\end{eqnarray*}
\setlength{\jot}{8pt} 
$$\underset{n\to\infty}{\mathlarger{\mathlarger{\sim}}} \; \frac{4 J_{2}(n)}{n^{2}} - \frac{12 J_{1}(n)}{n^{2}} + \frac{12}{n}. \hspace{1.5cm}$$
Since $\dfrac{12 J_{1}(n)}{n^{2}} \to 0$ \;and\; $\dfrac{12}{n} \to 0$ \;as $n\to\infty$, then similarly to the proof of Corollary \ref{corr:limp1} we obtain
$$\liminf \;n\cdot p_{2}(n) = \liminf \frac{ 4J_{2}(n) }{ n^{2} } =  \frac{24}{\pi^{2}}\qquad\textrm{and}\qquad \limsup \;n\cdot p_{2}(n) =\limsup \frac{ 4 J_{2}(n) }{ n^{2} } =4,$$
as required.
\end{proof}

\section{The case of an arbitrary permutation}
Introduce the following notation:
\setlength{\jot}{10pt} 
\begin{eqnarray*}
\A(n) &=& \sett{(s,t)\in\snsn}{[s,t] \textrm{ is a } 3\textrm{-cycle},\; \group{s,t} = \an \textrm{ or } \sn},\\
\B(n) &=& \sett{(s,t)\in\snsn}{[s,t] \textrm{ is a } 3\textrm{-cycle}}.
\end{eqnarray*}

\subsection{The number of all pairs: a formula for $\#\B(n)$}
\begin{theo}\label{th:B}  
The number of pairs of permutations from $S_{n}$ with a $3$-cycle commutator, is equal to
\begin{equation}\label{eq:B:formula}
   \#\B(n) = \frac{3}{8} \bigg(\sum_{k=1}^{n}\si_{3}(k) P(n-k) - 2\sum_{k=1}^{n} k\si_{1}(k)P(n-k) + n P(n)\bigg) \cdot n!
\end{equation}
for any natural $n>2$.
\end{theo}
\begin{proof} 
We are going to apply Lemma \ref{lemm:twotypes}, Ramanujan's formula (Proposition \ref{prop:ramanujan}) and Proposition \ref{prop:sidcP}.

Denote by $\PP(n)$ the set of partitions of a positive integer $n$:
\begin{equation}
   \PP(n) = \sett{(a_{1},\ldots,a_{r})\in \N^{r}}{r\in\N,\quad a_{1}\le\ldots\le a_{r}\in\N,\quad a_{1}+\dots + a_{r}=n}.
\end{equation}
Consider an arbitrary permutation $s\in\sn$ with flag $a_{1}\le \ldots\le a_{r}$\,, where
\begin{equation}\label{eq:partitionB}
   a_{1} +  \dots + a_{r} = n
\end{equation}
(\cf the section \ref{sec:permutations} of Appendix). 
A permutation $t\in\sn$ will be called \deff{allowed for $s$} if 
\begin{equation}\label{eq:stzyxB}
  [s,t] = (z\;y\;x), \quad \textrm{\ie}\quad t s t\inv = (x\;y\;z) s.
\end{equation}
A triple $(x,y,z)$ from the set $\lsem1,n\rsem^{3}$ will be said to be \deff{allowed for $s$}, if the relation \eqref{eq:stzyxB} holds for some allowed $t\in\sn$. Herewith, we assume that the integer $z$ is \emph{greater} than $y$ and $x$. 
According to Lemma \ref{lemm:twotypes}, one of the two following situations occurs: 
\begin{itemize}
\item[(a)] All three integers $x, y, z$ lie in the same cycle of the permutation $s$ in the given order. There are exactly $C_{a_{1}}^{3} + \dots + C_{a_{r}}^{3}$\, such triples, where $C_{1}^{3}=C_{2}^{3}=0$. In this case, the triple $(x,y,z)\in\lsem1,n\rsem^{3}$ and the permutation $t\in\sn$ with condition \eqref{eq:stzyxB} will be called \deff{allowed for $s$ of first kind}. Let us find the number of pairs $(s,t)\in\snsn$ when $t$ is allowed of first kind:
\setlength{\jot}{8pt} 
\begin{eqnarray*}
f(n) &:=& \hspace{-.6cm} \sum_{(a_{1},\ldots,a_{r})\in \PP(n)}\hspace{-.6cm} \#\sett{(s,t)\in\snsn}{ \flag{s} = a_{1},\ldots,a_{r}\,, \;\; t \textrm{ is allowed for } s \textrm{ of first kind}}\\
&=& \hspace{-.6cm} \sum_{(a_{1},\ldots,a_{r})\in \PP(n)}\hspace{-.6cm} \#\sett{s\in\sn}{ \flag{s} = a_{1},\ldots,a_{r} }\\
&& \hspace{4cm}\times\; \#\sett{t\in\sn}{ t \textrm{ is allowed for } s_{0} \textrm{ of first kind}},
\end{eqnarray*}
where $\medmu s_{0}=(1\;2\;\ldots\;a_{1})(a_{1}+1\;\;a_{1}+2\;\;\ldots\;\;a_{1}+a_{2})\cdots (a_{1}+\cdots +a_{r-1}+1\;\;a_{1}+\cdots +a_{r-1}+2\;\;\ldots\;\;a_{1}+\cdots+a_{r})$, and  $\flag{s}$ denotes the flag of the permutation (\cf Appendix \ref{sec:permutations}).
\begin{eqnarray*}
f(n) &=& \hspace{-.6cm} \sum_{(a_{1},\ldots,a_{r})\in \PP(n)}\hspace{-.6cm} \#\sett{s\in\sn}{ \flag{s} = a_{1},\ldots,a_{r} }\;\times\; |C_{\sn}(s_{0})| \\
& &\qquad\qquad\times\; \#\sett{ (x,y,z)\in \lsem 1,n\rsem^{3} }{\textrm{the triple } (x,y,z) \textrm{ is allowed for } s_{0} \textrm{ of first kind}} \\
&=& \hspace{-.6cm} \sum_{(a_{1},\ldots,a_{r})\in \PP(n)}\hspace{-.6cm} n! \;\times\; \big( C_{a_{1}}^{3} + \dots + C_{a_{r}}^{3} \big) 
\;=\; n! \hspace{-.6cm} \sum_{(a_{1},\ldots,a_{r})\in \PP(n)}\hspace{-.6cm} \big( C_{a_{1}}^{3} + \dots + C_{a_{r}}^{3} \big).
\end{eqnarray*}
In order to simplify the last sum, we will proceed as follows. If we run over all partitions \eqref{eq:partitionB} of the positive integer $n$ and count the number of ones, then we will get
$$P(n-1) + P(n-2) + \dots + P(1) + P(0),$$
where $P(n)$ is the partition function (\cf the section \ref{sec:partitions}). Indeed, there are $P(n-1)$ partitions with at least one `1' (\ie $a_{1}=1$), exactly $P(n-2)$ partitions with at least two `1's (\ie $a_{1}=1$ and $a_{2}=1$), and so on, exactly $P(0)=1$ partitions with $n$ ones. Analogically, the general statement is true: the number of parts of length $d$ in the partitions of the integer $n$, is equal to
$$P(n-d) + P(n-2d) + \dots + P(n - m d), \quad \textrm{ where } m = \left[\frac{n}{d}\right].$$
Thus,
\begin{eqnarray*}
f(n) &=& n! \hspace{-.6cm} \sum_{(a_{1},\ldots,a_{r})\in \PP(n)}\hspace{-.6cm} \big( C_{a_{1}}^{3} + \dots + C_{a_{r}}^{3} \big) \;=\; n! \sum_{d=1}^{n} \bigg( P(n-d) + P(n-2d) + \dots + P(n - [\drob{n}{d}] d) \bigg)\cdot C_{d}^{3}\,.
\end{eqnarray*}
Let us regroup the summands in the last sum, by gathering the coefficients of $P(n-k)$ for $k\in\lsem 1,n\rsem$. For instance, the term $P(n-6)$ occurs $4$ times in the sum: as a multiplier of  $C_{1}^{3}$, $C_{2}^{3}$, $C_{3}^{3}$ and $C_{6}^{3}$. We obtain that
\renewcommand{\arraystretch}{2}
\begin{equation}\label{eq:fnB}
\begin{array}{rl}
\dfrac{f(n)}{n!} \;=& \,\Sum\limits_{k=1}^{n} \bigg(\Sum\limits_{d\mid k} C_{d}^{3}\bigg)\, P(n-k) 
\;=\; \dfrac{1}{6} \,\Sum\limits_{k=1}^{n} \bigg(\Sum\limits_{d\mid k} d(d-1)(d-2) \bigg)\, P(n-k)\\ 
=& \dfrac{1}{6} \,\Sum\limits_{k=1}^{n} \bigg(\Sum\limits_{d\mid k} d^{3} - 3d^{2} + 2d \bigg)\, P(n-k) \\
=& \Sum\limits_{k=1}^{n} \bigg( \dfrac{1}{6}\si_{3}(k) - \dfrac{1}{2}\si_{2}(k) + \dfrac{1}{3}\si_{1}(k) \bigg)\, P(n-k)\,, 
\end{array}
\end{equation}
where $\si_{i}(n)$ is the sum of $i^{\textrm{th}}$ powers of the divisors  of $n$, \cf the formula \eqref{eq:sik}.
\item[(b)] Two integers $x$ and $y$ lie in the same cycle of the permutation $s$, and the number $z$ is in another cycle. Moreover, the length of the latter cycle equals the $s$-distance from $y$ to $x$. In this case, the triple $(x,y,z)\in\lsem1,n\rsem^{3}$ and the corresponding permutations $t\in\sn$ with condition \eqref{eq:stzyxB} will be called \deff{allowed for $s$ of second kind}. It is clear that if $x$ and $y$ lie in a cycle of length $a_{j}$ and if $z$ lies in a cycle of length $a_{i}$ with $a_{i}<a_{j}$\,, then such a triple $(x,y,z)$ can be chosen in $a_{i}\cdot a_{j}$ ways. Hence, the number of all allowed triples for $s$ of second kind is equal to 
$$\sum\limits_{ \substack{1\le i<j\le r \\ a_{i} < a_{j}} } a_{i} a_{j}.$$  
Let us find the number of pairs $(s,t)\in\snsn$ when $t$ is allowed of second kind:
\setlength{\jot}{8pt} 
\begin{eqnarray*}
g(n) &:=& \hspace{-.6cm} \sum_{(a_{1},\ldots,a_{r})\in \PP(n)}\hspace{-.6cm} \#\sett{(s,t)\in\snsn}{ \flag{s} = a_{1},\ldots,a_{r}\,, \;\; t \textrm{ is allowed for } s \textrm{ of second kind}}\\
&=& \hspace{-.6cm} \sum_{(a_{1},\ldots,a_{r})\in \PP(n)}\hspace{-.6cm} \#\sett{s\in\sn}{ \flag{s} = a_{1},\ldots,a_{r} }\\
&& \hspace{4cm}\times\; \#\sett{t\in\sn}{ t \textrm{ is allowed for } s_{0} \textrm{ of second kind}},
\end{eqnarray*}
where $s_{0}$ is a fixed permutation with flag $(a_{1},\ldots,a_{r})$,
%
\begin{eqnarray*}
g(n) &=& \hspace{-.6cm} \sum_{(a_{1},\ldots,a_{r})\in \PP(n)}\hspace{-.6cm} \#\sett{s\in\sn}{ \flag{s} = a_{1},\ldots,a_{r} }\;\times\; |C_{\sn}(s_{0})| \\
& &\qquad\qquad\times\; \#\sett{ (x,y,z)\in \lsem 1,n\rsem^{3} }{\textrm{the triple } (x,y,z) \textrm{ is allowed for } s_{0} \textrm{ of second kind}} \\
&=& \hspace{-.6cm} \sum_{(a_{1},\ldots,a_{r})\in \PP(n)}\hspace{-.6cm} n! \;\times\; \bigg( \sum\limits_{ \substack{1\le i<j\le r \\ a_{i} < a_{j}} } a_{i} a_{j} \bigg) 
\;=\; n! \hspace{-.6cm} \sum_{ \substack{(a_{1},\ldots,a_{r})\in \PP(n) \\ a_{i} < a_{j}} }\hspace{-.6cm} a_{i} a_{j}\,.
\end{eqnarray*}
Let us simplify the last sum. If there are exactly $a$ parts of length $k=a_{i}$ and exactly $b$ parts of length  $\ell=a_{j}$ in a partition $(a_{1},\ldots,a_{r})$, then gather them in the sum as $a b\cdot k \ell$. Running over all distinct partitions, such a term will occur
\begin{equation}\label{eq:Pnak}
   \medmu P\big(n - a k -b \ell\big) \;-\; P\big(n - (a+1)k - b \ell\big) \;-\; P\big(n - a k - (b+1) \ell\big) \;+\; P\big(n - (a+1)k - (b+1) \ell\big)
\end{equation}
times, according to the inclusion-exclusion principle. The expression \eqref{eq:Pnak} counts the number of partitions of $n$, which contain exactly $a$ parts of length $k$ and exactly $b$ parts of length $\ell$. Therefore,
\renewcommand{\arraystretch}{2}
$$
\begin{array}{rl}
\dfrac{g(n)}{n!} =& \medmu  \hspace{-.2cm} \Sum\limits_{ \substack{a,b,k,\ell\in\N \\ k<\ell \\ a k + b \ell \le n} } 
\hspace{-.2cm} a b k \ell \bigg(  P\big(n - a k -b \ell\big) \;-\; P\big(n - (a+1)k - b \ell\big) \;-\; P\big(n - a k - (b+1) \ell\big) \;+\; P\big(n - (a+1)k - (b+1) \ell\big) \bigg)\,,
\end{array}
$$
from where, gathering the coefficients of $P(n-c)$ for $c\in\lsem 1,n\rsem$, we obtain 
$$
\begin{array}{rl}
\dfrac{g(n)}{n!} =& \Sum\limits_{c=1}^{n} \hspace{.2cm}  \Sum\limits_{ \substack{a,b,k,\ell\in\N \\ k<\ell \\ a k + b \ell = c} } \bigg( a b k \ell - (a-1)b k \ell - a(b-1)k \ell + (a-1)(b-1)k \ell \bigg)\;P(n-c)\\
=& \Sum\limits_{c=1}^{n} \hspace{.2cm}  \Sum\limits_{ \substack{a,b,k,\ell\in\N \\ k<\ell \\ a k + b \ell = c} } k \ell\, P(n-c)\;=\; \Sum\limits_{c=1}^{n} \hspace{.2cm}  \Bigg(\Sum\limits_{ \substack{a,b,k,\ell\in\N \\ k<\ell \\ a k + b \ell = c} } \mhs k \ell\Bigg)\, P(n-c) \,.
\end{array}
$$
Now, let us evaluate the following sum
\begin{equation}\label{eq:kl}
 \sum\limits_{ \substack{a,b,k,\ell\in\N \\ k<\ell \\ a k + b \ell = c} } \mhs k \ell 
\;=\; \frac{1}{2} \Bigg( \sum\limits_{ \substack{a,b,k,\ell\in\N \\ a k + b \ell = c} } \mhs k \ell 
\quad- \sum\limits_{ \substack{a,b,k\in\N \\ (a + b)k = c} } \mhs k^{2}\Bigg)
\end{equation}
From the equality \,$(a + b)k = c$\, follows that $k$ divides $c$. Moreover, for fixed $k$ and $c$, exactly $c/k-1$ pairs of positive integers $(a,b)$ satisfy this equality  (since $a$ can take values from $1$ to $c/k-1$). We have
\begin{equation}\label{eq:kl1}
\sum\limits_{ \substack{a,b,k\in\N \\ (a + b)k = c} } \mhs k^{2} \;=\; \sum_{k\mid c} \left(\frac{c}{k}-1\right)\, k^{2} \;=\; c \sum_{k\mid c} k -  \sum_{k\mid c}k^{2} \;=\; c\, \si_{1}(c) - \si_{2}(c)\,.
\end{equation}
Further, with notation $\al = a k$ and $\be = b \ell$ one has
\begin{equation*}
   \sum_{\substack{a,b,k,\ell\in\N \\ a k + b \ell = c} } \mhs k \ell \;=\; \sum_{\substack{\al, \be\in\N \\ \al + \be = c} }  \bigg(\sum_{k\mid \al}k\bigg) \bigg(\sum_{\ell\mid \be}\ell\bigg) \;=\; \sum_{\substack{\al, \be\in\N \\ \al + \be = c} } \si_{1}(\al) \si_{1}(\be)\,.
\end{equation*}
The last sum appears in one of Ramanujan's formula (\cf Proposition \ref{prop:ramanujan}), due to which
\begin{equation}\label{eq:kl2}
\sum_{\substack{a,b,k,\ell\in\N \\ a k + b \ell = c} } \mhs k \ell \;=\; \sum_{\substack{\al, \be\in\N \\ \al + \be = c} } \si_{1}(\al) \si_{1}(\be) \;=\; \frac{5}{12}\si_{3}(c) + \frac{1}{12}\si_{1}(c) - \frac{1}{2}c\, \si_{1}(c)\,.
\end{equation}

The relations \eqref{eq:kl}, \eqref{eq:kl1} and \eqref{eq:kl2} give
\begin{equation}\label{eq:sumklB}
\sum\limits_{ \substack{a,b,k,\ell\in\N \\ k<\ell \\ a k + b \ell = c} } \mhs k \ell 
\;=\; \frac{1}{2} \Bigg(  \frac{5}{12}\si_{3}(c) + \frac{1}{12}\si_{1}(c) - \frac{1}{2}c\, \si_{1}(c)
 - c\, \si_{1}(c) + \si_{2}(c) \Bigg)
 \;=\; \frac{5}{24}\si_{3}(c) + \frac{1}{2}\si_{2}(c) + \bigg(\frac{1}{24} - \frac{3c}{4}\bigg) \si_{1}(c)    
\end{equation}
and so
\begin{equation}\label{eq:gnB}
   \dfrac{g(n)}{n!} \;=\; \Sum\limits_{c=1}^{n} \hspace{.2cm}  \Bigg(\Sum\limits_{ \substack{a,b,k,\ell\in\N \\ k<\ell \\ a k + b \ell = c} } \mhs k \ell\Bigg)\, P(n-c)
   \;=\; \Sum\limits_{c=1}^{n} \hspace{.2cm}  \Bigg( \frac{5}{24}\si_{3}(c) + \frac{1}{2}\si_{2}(c) + \bigg(\frac{1}{24} - \frac{3c}{4}\bigg) \si_{1}(c) \Bigg)\, P(n-c)\,.
\end{equation}
\end{itemize}
Finally, summing the formulas \eqref{eq:fnB} and \eqref{eq:gnB}, we obtain
\begin{eqnarray*}
   \#\B(n) &=& \hspace{-.6cm} \sum_{(a_{1},\ldots,a_{r})\in \PP(n)}\hspace{-.6cm} \#\sett{(s,t)\in\snsn}{ \flag{s} = a_{1},\ldots,a_{r}\,, \;\; t \textrm{ allowed for } s} \;=\;  f(n) + g(n) \\
   &=& n!\; \sum\limits_{k=1}^{n} \Bigg( \dfrac{1}{6}\si_{3}(k) - \dfrac{1}{2}\si_{2}(k) + \dfrac{1}{3}\si_{1}(k) + \frac{5}{24}\si_{3}(k) + \frac{1}{2}\si_{2}(k) + \bigg(\frac{1}{24} - \frac{3k}{4}\bigg) \si_{1}(k) \Bigg)\, P(n-k)\\
   &=& n!\; \sum\limits_{k=1}^{n} \Bigg( \dfrac{3}{8}\si_{3}(k) - \frac{3k}{4} \si_{1}(k) + \dfrac{3}{8}\si_{1}(k) \Bigg)\, P(n-k)\,.
\end{eqnarray*}
According to Proposition \ref{prop:sidcP}, the equality $\sum\limits_{k=1}^{n} \si_{1}(k) \, P(n-k) = n P(n)$ holds for $n\in\N$\,, from where 
$$\#\B(n) \;=\; \dfrac{3}{8}n!\, \sum\limits_{k=1}^{n} \Big( \si_{3}(k) - 2 \si_{1}(k) \Big)\, P(n-k) \;+\; \dfrac{3n}{8} P(n) n!\,,$$
as required.
\end{proof}


\vs
In order to estimate the cardinality of the set $\B(n)$, let us show the next lemma.

\begin{lemm}\label{lemm:B:si3} 
For any natural $n>1$, the following inequality holds
$$\si_{3}(n) < n^{2} \si_{1}(n).$$
For any $\delta > 0$, there exists a positive integer $K=K(\delta)$ such that\footnote{The notation $K=K(\delta)$ means that a number $K$ depends only on $\delta$.} 
$$\si_{3}(n) > n^{2 - \delta} \si_{1}(n)$$
for all natural $n \ge K$.
\end{lemm}
\begin{proof}
The first inequality follows from the definition of $\si_{3}(n)$ and $\si_{1}(n)$:
$$\si_{3}(n) \;=\; \sum_{d\mid n} d^{3} \;<\; \sum_{d\mid n} n^{2} d \;=\; n^{2}\si_{1}(n).$$
Consider now the arithmetic function
$$g(n) = \frac{\si_{3}(n)}{n \si_{1}(n)}.$$
Notice that it is multiplicative as a ratio of two multiplicative functions.
For a power of a prime $p^{\al}$, according to the formula \eqref{eq:sik} we have
\setlength{\jot}{8pt} 
\begin{eqnarray*}
g(p^{\al}) &=& \frac{\si_{3}(p^{\al})}{p^{\al} \si_{1}(p^{\al})} 
	\;=\; \frac{ \drob{ (p^{3(\al+1)} - 1) }{ (p^{3} - 1) } }
		{ p^{\al} \drob{ (p^{\al+1} - 1) }{(p-1)} } 
	\;=\; \frac{ (p^{\al+1}-1) (p^{2(\al+1)} + p^{\al+1} + 1) }
		{ p^{\al} (p^{\al+1}-1) (p^{2} + p + 1) } \\
&=& \frac{ p^{2\al+2}\big(1 + \drob{1}{p^{\al+1}} + \drob{1}{p^{2\al+2}}\big) }
		{ p^{\al+2}\big(1 + \drob{1}{p} + \drob{1}{p^{2}}\big) }
	\;=\; p^{\al}\cdot \frac{ 1 + \drob{1}{p^{\al+1}} + \drob{1}{p^{2\al+2}} }{ 1 + \drob{1}{p} + \drob{1}{p^{2}} } \\
&>& p^{\al}\cdot \frac{ 1 }{ 1 + \frac{1}{p} + \frac{1}{p^{2}} + \frac{1}{p^{3}} + \cdots} 
	\;=\; p^{\al}\cdot \frac{ 1 }{ \drob{1}{\big(1-\frac{1}{p}\big)} } 
	\;=\; p^{\al} \left(1 - \frac{1}{p}\right).
\end{eqnarray*}
Thus, for any positive integer $n=p_{1}^{\al_{1}}\cdots\, p_{r}^{\al_{r}}>1$, the following inequality holds
$$g(n) \;=\; g(p_{1}^{\al_{1}}) \cdots\, g(p_{r}^{\al_{r}}) 
	\;>\; p_{1}^{\al_{1}} \left(1 - \frac{1}{p_{1}}\right)\cdots\, p_{r}^{\al_{r}} \left(1 - \frac{1}{p_{r}}\right)
	\;=\; n\mhs \prod_{\substack{p | n\\ p\;\textrm{prime}}} \mhs \left(1-\frac{1}{p}\right)
	\;=\; \phi(n),$$ 
\begin{equation}\label{eq:B:si3:si3}
	\textrm{\ie}\qquad \si_{3}(n) > n\phi(n) \si_{1}(n),
\end{equation}
where $\phi(n)$ is Euler's totient function, \cf \eqref{eq:phi}.

Let us show that for any $\delta > 0$, there exists a natural $K=K(\delta)$ such that 
\begin{equation}\label{eq:B:si3:phi}
	\phi(n) > n^{1 - \delta}
\end{equation}
for all natural $n \ge K$. Consider the function
$$f(n) = \frac{n^{1-\delta}}{\phi(n)}.$$
It is multiplicative and when $n=p^{\al}$ it takes the value
$$f(p^{\al}) \;=\; \frac{ p^{\al(1-\delta)} }{ p^{\al} \left(1 - \drob{1}{p}\right)  }
	\;=\; \frac{1}{p^{\al\delta}}\cdot \frac{p}{p-1}\,,$$
which tends to zero as $p^{\al}\to +\infty$. According to Proposition \ref{prop:appA:epsilon}, we conclude that $f(n)\to 0$ as $n\to +\infty$. In particular, there exists $K\in\N$ such that
$$f(n) < 1 \quad \textrm{for all } n\ge K.$$
This proves the inequality \eqref{eq:B:si3:phi}, from where taking into account the inequality \eqref{eq:B:si3:si3} we obtain that
$$\si_{3}(n) \;>\; n\phi(n) \si_{1}(n) \;>\; n^{2 - \delta} \si_{1}(n) \quad \textrm{for all } n\ge K,$$
as required.
\end{proof}

\vs
For a real number $a$, consider the following arithmetic function:
\begin{equation}\label{eq:B:psi}
	\psi_{a}(n) = \sum_{k=1}^{n} k^{a} \si_{1}(k)P(n-k).
\end{equation}
In Proposition \ref{prop:sidcP} it is shown that $\psi_{0}(n) = n P(n)$. If $a\ge b$ then for any  $n\in\N$, the inequality $\psi_{a}(n) \ge \psi_{b}(n)$ holds. We have the following simple bounds of the function $\psi_{a}(n)$ when $a\ge 0$:
\begin{equation}\label{eq:B:psibounds}
	n P(n) =\sum_{k=1}^{n} \si_{1}(k)P(n-k) \;\le\; \psi_{a}(n) \;\le\; n^{a} \sum_{k=1}^{n}  \si_{1}(k)P(n-k) = n^{a+1} P(n).
\end{equation}
In particular, $n P(n)\le \psi_{2}(n) \le n^{3}P(n)$. 

Let us now estimate the cardinality of the set $\B(n)$.

\begin{corr}\label{corr:B:bounds} 
For any $\eps > 0$, there exists a positive integer $N=N(\eps)$ such that 
$$\psi_{2 - \eps}(n)\cdot n! \;<\; \#\B(n) \;<\; \frac{3}{8} \psi_{2}(n)\cdot n!$$
for all natural \,$n \ge N$.
\end{corr}

\begin{proof}
From Theorem \ref{th:B}, the relation $\psi_{0}(n) = n P(n)$, the inequality $\psi_{1}(n)\ge \psi_{0}(n)$ and  the inequality $\si_{3}(k) \le k^{2}\si(k)$ in Lemma \ref{lemm:B:si3} follows that
\begin{eqnarray*}
   \frac{8}{3}\cdot \frac{\#\B(n)}{n!} &=& \sum_{k=1}^{n}\si_{3}(k) P(n-k) \;-\; 2 \psi_{1}(n) \;+\; \psi_{0}(n)\\
   &\le&  \sum_{k=1}^{n} k^{2} \si_{1}(k) P(n-k) \;-\; 2 \psi_{0}(n) \;+\; \psi_{0}(n) \;=\; \psi_{2}(n) - \psi_{0}(n)\\
   &<& \psi_{2}(n),
\end{eqnarray*}
which proves the upper bound of the cardinality of $\B(n)$ for all $n\in\N$. 

Let $\eps$ be an arbitrary positive real number. Choose any number $\delta > 0$ with conditions $\delta<1$ and $\delta < \eps$. 
According to Lemma \ref{lemm:B:si3}, there exists a positive integer $K$ such that $\si_{3}(k) > k^{2 - \delta} \si_{1}(k)$ for all natural $k \ge K$. Then
\begin{eqnarray*}
 \sum_{k=1}^{n}\si_{3}(k) P(n-k) &=& \sum_{k=K}^{n}\si_{3}(k) P(n-k) + \sum_{k=1}^{K-1}\si_{3}(k) P(n-k)\\
 	&>& \sum_{k=K}^{n} k^{2-\delta} \si_{1}(k) P(n-k) + \sum_{k=1}^{K-1}\si_{3}(k) P(n-k)\\ 
	&>& \sum_{k=K}^{n} k^{2-\delta} \si_{1}(k) P(n-k).
\end{eqnarray*}
In the right side of this inequality, let us add and subtract the expression $\sum\limits_{k=1}^{K-1} k^{2-\delta}\si_{1}(k) P(n-k)$, getting
\begin{eqnarray}
	\sum_{k=1}^{n}\si_{3}(k) P(n-k) 
	&>& \sum_{k=1}^{n} k^{2-\delta} \si_{1}(k) P(n-k) \;-\; \sum_{k=1}^{K-1} k^{2-\delta} \si_{1}(k) P(n-k) \nonumber\\ 
	&>&  \psi_{2-\delta}(n) \;-\;  K^{5} P(n)\,, \label{eq:B:si3-1}
\end{eqnarray}
since $k^{2-\delta} < K^{2}$, \;$\si_{1}(k)\le k^{2} < K^{2}$  \;and\;  $P(n-k) < P(n)$ for $1\le k < K$.

From the conditions $\delta < 1$ and $\delta < \eps$ follows that     
\begin{equation*}
	\frac{2k + \frac{8}{3} k^{2-\eps}}{k^{2-\delta}} \;=\; \frac{2}{k^{1-\delta}} + \frac{8/3}{k^{\eps-\delta}} 
	\;\;\underset{k\to+\infty}{\longto}\;\; 0.
\end{equation*}
Hence, there is a natural $L=L(\eps)$ such that 
\begin{equation*}
	\frac{2k + \frac{8}{3} k^{2-\eps}}{k^{2-\delta}} < 1,\quad\textrm{\ie}\quad 
	k^{2-\delta} - 2k > \frac{8}{3} k^{2-\eps} \quad\textrm{for } k\ge L\,.
\end{equation*}
When $1\le k < L$, one has the inequalities $k^{2-\delta} - 2k - \frac{8}{3} k^{2-\eps} > k - 2k - 3k^{2} \ge -L - 3L^{2} \ge - 4L^{2}$, \ie $k^{2-\delta} - 2k >  \frac{8}{3} k^{2-\eps} - 4L^{2}$.
Therefore, we obtain
\begin{eqnarray*}
\psi_{2-\delta}(n) - 2\psi_{1}(n) 
&=& \sum_{k=1}^{n} \left(k^{2-\delta} - 2k\right)\cdot \si_{1}(k) P(n-k)\\
&=& \sum_{k=L}^{n} \left(k^{2-\delta} - 2k\right)\cdot \si_{1}(k) P(n-k) \;+\; \sum_{k=1}^{L-1} \left(k^{2-\delta} - 2k\right)\cdot \si_{1}(k) P(n-k)\\
&>& \sum_{k=L}^{n}  \frac{8}{3} k^{2-\eps} \si_{1}(k) P(n-k) \;+\; \sum_{k=1}^{L-1} \left(\frac{8}{3} k^{2-\eps} - 4L^{2}\right)\cdot \si_{1}(k) P(n-k) \\
&>& \frac{8}{3} \sum_{k=1}^{n}  k^{2-\eps} \si_{1}(k) P(n-k) \;-\; L\cdot 4L^{2} \cdot L^{2} P(n) \;=\; \frac{8}{3} \psi_{2-\eps}(n) - 4L^{5} P(n),
\end{eqnarray*}
since $\si_{1}(k)\le k^{2} < L^{2}$ \;and\; $P(n-k) < P(n)$, when $1\le k<L$. 

By applying this inequality and the inequality \eqref{eq:B:si3-1}, we get the following bound
\setlength{\jot}{8pt} 
\begin{eqnarray*}
   \frac{8}{3}\cdot \frac{\#\B(n)}{n!} 
   &=& \sum_{k=1}^{n}\si_{3}(k) P(n-k)  \;-\; 2\psi_{1}(n) \;+\; n P(n)\\
   &>& \psi_{2-\delta}(n) \;-\; 2\psi_{1}(n) \;-\;  K^{5} P(n) \;+\; n P(n)\\
   &>& \frac{8}{3} \psi_{2-\eps}(n) \;+\;  \left(n - K^{5} - 4L^{5} \right)\cdot P(n) \\
   &>& \frac{8}{3} \psi_{2-\eps}(n) \qquad \textrm{for } n\ge N,
\end{eqnarray*}
where a constant $N = K^{5} + 4L^{5}$ depends on $\delta$ and $\eps$. In its turn, the choice of $\delta$ was based on the number $\eps$. Therefore, $N$ depends only on $\eps$.

We conclude that for any $\eps > 0$, there exists a natural $N=N(\eps)$ such that $\#\B(n) > \psi_{2-\eps} (n)\cdot n!$ for all $n\ge N$. This completes the proof of Corollary.
\end{proof}

\vs
The next statement is an application of Theorem \ref{th:B} to the theory of representations of the symmetric group (several main definitions are given in Appendix \ref{sec:representations}).

\begin{corr}\label{corr:B:character} 
Let $c\in\sn$ be any $3$-cycle. 
For all natural $n>2$ the following relation holds
\begin{equation*}
	\frac{8}{9} n(n-1)(n-2) \sum_{\rho} \frac{\upchi_{\rho}(c)}{\dim \rho}
	\;=\; n P(n) + \sum_{k=1}^{n} \bigg( \si_{3}(k) - 2k\, \si_{1}(k) \bigg) P(n-k)\,,
\end{equation*}
where the sum is over all $($pairwise nonequivalent$)$ irreducible representations $\rho$ of $\sn$\,, and $\upchi_{\rho}$ denotes the character of $\rho$.

In particular, one has the bounds
\begin{equation*}
	\frac{3}{n^{2}} P(n) \;<\;
	\sum_{\rho} \frac{\upchi_{\rho}(c)}{\dim \rho} \;<\; 
	\frac{5}{4} P(n) 
\end{equation*}
for any sufficiently large $n$.
\end{corr}
\begin{proof}
According to Frobenius's formula for $u=c$  (\cf Proposition \ref{prop:frobenius}), we have
\begin{equation*}
	\#\B(n) \;=\; n!\cdot\frac{n(n-1)(n-2)}{3}\cdot \sum_{\rho} \frac{\upchi_{\rho}(c)}{\dim \rho}\,. 
\end{equation*}
Comparing with Theorem \ref{th:B}, we get the required relation.

We know from Corollary \ref{corr:B:bounds} with $\eps=2$ that, for sufficiently large $n$, the following inequalities hold
$$
\psi_{0}(n) \;<\; \frac{\#\B(n)}{n!} \;<\; \frac{3}{8}\psi_{2}(n)\,,
\qquad\textrm{and so}\qquad
n P(n)\;<\; \frac{\#\B(n)}{n!} \;<\; \frac{3}{8} n^{3} P(n)\,,
$$
since $\psi_{0}(n) = n P(n)$ and $\psi_{2}(n)\le n^{3} P(n)$.\, Therefore,
\begin{equation*}
	\frac{3}{(n-1)(n-2)} P(n) \;<\;
	\sum_{\rho} \frac{\upchi_{\rho}(c)}{\dim \rho} \;<\;
	\frac{9}{8}\cdot \frac{n^{2}}{(n-1)(n-2)} P(n)\,.
\end{equation*}
But $\frac{3}{n^{2}}<\frac{3}{(n-1)(n-2)}$ and for sufficiently large $n$ one has 
$$\frac{9}{8}\cdot \frac{n^{2}}{(n-1)(n-2)} = \frac{9}{8}\cdot \frac{1}{(1-\frac{1}{n})(1-\frac{2}{n})} < \frac{10}{8} =\frac{5}{4}\,.$$
This proves the required bounds.
\end{proof}

\subsection{The number of primitive pairs: a formula for $\#\A(n)$}

\begin{lemm}\label{lemm:twocylinderprim0} 
Vectors $\vv{(\al,a)}$, $\vv{(\be,b)}$, $\vv{(k,0)}$ and $\vv{(\ell,0)}$ generate the lattice $\Z^{2}$ if and only if the following two conditions are stisfied:
$$a\land b = 1 \qquad\textrm{and}\qquad k\land \ell\land (a\be - b\al) = 1.$$
\end{lemm}
\begin{proof}
\onlyif The condition $a\land b = 1$ is necessary, since the ordinate of any linear combination with integer coefficients of the four given vectors is obviously divisible by the greatest common divisor of $a$ and $b$. Suppose that the vectors $\vv{(\al,a)}$, $\vv{(\be,b)}$, $\vv{(k,0)}$, $\vv{(\ell,0)}$ generate $\Z^{2}$ and that $a\land b = 1$. Let us find all integers $p$, for which the system of equations
\begin{equation*}
	\begin{cases}
		\al X + \be Y = p, \\
		a X + b Y = 0
	\end{cases}
\end{equation*}
has a solution in integers $X$ and $Y$. We get\; $a X = -b Y$, which is equivalent to equalities $X = -m b$ and $Y = m a$ for some $m\in\Z$, since $a$ and $b$ are coprime. From this we obtain 
$$p = \al X + \be Y = m(a\be - b\al),$$
which is a multiple of $a\be - b\al$. 

By assumption, the vector $\vv{(1,0)}$ is a linear combination 
$$X \vv{(\al,a)} + Y \vv{(\be,b)} + Z \vv{(k,0)} + W \vv{(\ell,0)}$$ 
with integer coefficients $X$, $Y$, $Z$ and $W$. This is possible only if 
\begin{equation}\label{eq:twocylinderlinear}
	m(a\be - b\al) + Z k + W \ell=1\quad \textrm{for some integers } m, Z \textrm{ and } W,
\end{equation}
\ie if the numbers $a\be - b\al$, $k$ and $\ell$ are coprime.

\vs
\ifonly Conversely, suppose that the conditions $a\land b = 1$ and $k\land \ell\land (a\be - b\al) = 1$ are satisfied. From the second condition follows that the equality \eqref{eq:twocylinderlinear} holds, and so
$$-m b \vv{(\al,a)} + m a \vv{(\be,b)} + Z \vv{(k,0)} + W \vv{(\ell,0)} = \vv{(m(a\be - b\al) + Z k + W \ell,\; -m b a + m a b)} = \vv{(1,0)}.$$
The first condition implies the existence of integers $X$ and $Y$ such that $X a + Y b = 1$, from where 
$$X \vv{(\al,a)} + Y \vv{(\be,b)} = \vv{(X\al+Y\be,\; 1)}.$$
Since linear combinations of the vector $\vv{(1,0)}$ and $\vv{(X\al+Y\be,\, 1)}$ with integer coefficients cover all $\Z^{2}$, then the given four vectors generate $\Z^{2}$.
\end{proof}


\vs
\begin{lemm}\label{lemm:twocylinderprim} 
A two-cylinder origami with parameters $(a, b, k, \ell, \al, \be)$ is primitive if and only if the following two conditions are satisfied:
$$a\land b = 1 \qquad\textrm{and}\qquad k\land \ell\land (a\be - b\al) = 1.$$
\end{lemm}
\begin{proof} 
A two-cylinder origami with parameters $(a, b, k, \ell, \al, \be)$ consists of $n = a k + b \ell$ squares. Number its squares by the integers from $1$ to $n$. Then we will get two permutations $s$ and $t$ from $\sn$ that indicate how the squares are glued in the horizontal and the vertical directions respectively. Denote by $G$ the monodromy group origami, \ie the permutation group generatеd by the pair $(s,t)$.

\vs
\onlyif Suppose that the conditions $a\land b = 1$ and $k\land \ell\land (a\be - b\al) = 1$ are not satisfied. By Lemma \ref{lemm:twocylinderprim0}, this means that the vectors $\vv{(\al,a)}$, $\vv{(\be,b)}$, $\vv{(k,0)}$ and $\vv{(\ell,0)}$ generate a lattice $L$ not coinciding with the entire $\Z^{2}$.   Let us show then that the square-tiled surface is not primitive, \ie the permutation group  $G = \group{s,t}$ has a nontrivial block $\Delta\subset \lsem 1,n\rsem$.

\begin{figure}[htbp]
   \begin{center}
   \begin{tikzpicture}[scale=.4, >=latex,shorten >=1pt,shorten <=1pt,line width=1pt,bend angle=20, 
   		circ/.style={circle,inner sep=2pt,draw=black!40,rounded corners, text centered},
   		tblue/.style={dashed,blue,rounded corners}]
        
      	\draw[step=1] (0,0) grid +(11,3);
      	\draw[step=1] (0,-8) grid +(5,4);
	
	\path[<->,thin,>=arcs] 	(-1.7,-8) edge node[left] {\small $a$} ++ (0,4)
						(-1.7,0) edge node[left] {\small $b$} ++ (0,3);
	\foreach \x in {0,10}
		\draw[thin] (\x,1) -- ++(0,1);

      	\draw[line width=2pt,blue] (5,0) -- ++(0,3);

	\path[<->,thin,>=arcs,blue]  (0,3.7) edge node[above] {\small $\be$} ++ (4,0)
						(4,3.7) edge node[above] {\small $\ell-\be$} ++ (7,0)
						(0,-8.7) edge node[below] {\small $\ell-\be$} ++ (7,0)
						(7,-8.7) edge node[below] {\small $\be$} ++ (4,0);
						
	\draw[dotted] 	(0,0) -- ++(4,3)
				(11,0) -- ++(4,3) -- ++(-4,0);

	\draw[dotted] 	(0,-8) -- ++(2,4)
				(5,-8) -- ++(2,4) -- ++(-2,0);

	\path[<->,thin,>=arcs,blue]  (0,-3.5) edge node[above] {\small $\al$} ++ (2,0)
						(2,-3.5) edge node[above] {\small $k-\al$} ++ (3,0)
						(0,-.5) edge node[below] {\small $k-\al$} ++ (3,0)
						(3,-.5) edge node[yshift=-.08cm,below] {\small $\al$} ++ (2,0);
	\foreach \x in {0,4,11}
		\draw[thin,blue] (\x,3) -- ++(0,1);
	\foreach \x in {0,7,11}
		\draw[thin,blue] (\x,-8) -- ++(0,-1);
	
	\foreach \x in {0,2,5}
		\draw[thin,blue] (\x,-4) -- ++(0,.8);
	\foreach \x in {0,3,5}
		\draw[thin,blue] (\x,0) -- ++(0,-.8);
		
	\begin{scope}[xshift=17cm,yshift=-7cm]
	\node at (-4,4) {$=$};
	\path[<->,thin,>=arcs,green!40!black]  	
						(2,7.8) edge node[above] {\small $\be$} ++ (4,0)
						(6,7.8) edge node[above] {\small $k$} ++ (5,0)
						(11,7.8) edge node[above] {\small $l-k$} ++ (6,0)
						(0,-.8) edge node[below] {\small $k$} ++ (5,0)
						(5,-.8) edge node[below] {\small $\al$} ++ (2,0)
						(7,3.2) edge node[below] {\small $l-k$} ++ (6,0)
						(-.8,0) edge node[left] {\small $a$} ++ (0,4)
						(1.2,4) edge node[left] {\small $b$} ++ (0,3);
	\clip (0,0) -- ++(2,4) -- ++(4,3) -- ++(11,0) -- ++(-4,-3) -- ++(-6,0) -- ++(-2,-4) -- ++(-5,0);
      	\draw[step=1] (0,0) grid +(17,7);
	\draw 	(0,0) -- ++(2,4) -- ++(4,3)
			(5,0) -- ++(2,4)
			(13,4) -- ++(4,3);
	\end{scope}
   \end{tikzpicture}
   \caption{An unfolding of the two-cylinder square-tiled surface.}
   \label{fig:twocylinderprimunfold}
   \end{center} \vspace{-0.3cm}
\end{figure}
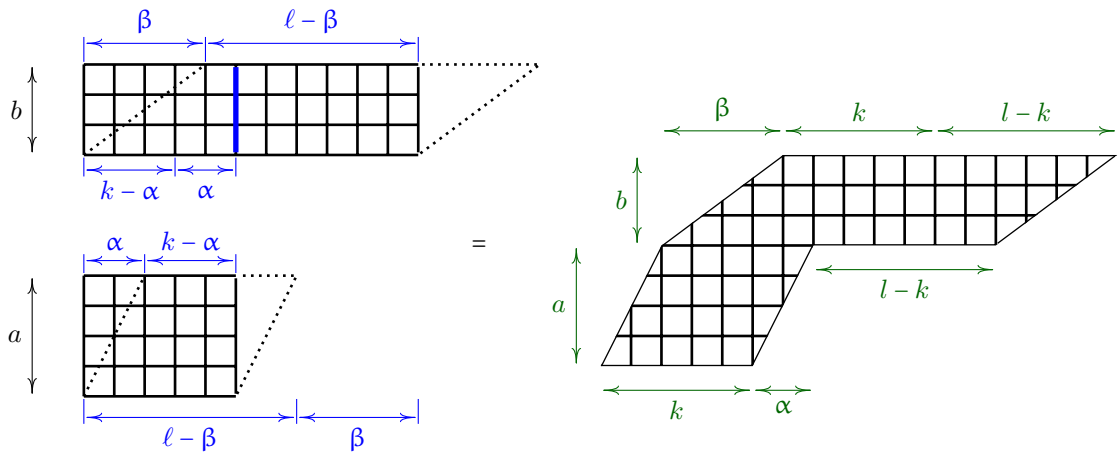

In the real plane, mark the points of the lattice $L$. A two-cylinder square-tiled surface can be unfolded in the plane as two glued parallelograms: the first one is constructed on the vectors $\vv{(\al,a)}$ and $\vv{(k,0)}$, and the second one is constructed on the vectors $\vv{(\be,b)}$ and $\vv{(\ell,0)}$, see Figure \ref{fig:twocylinderprimunfold}. Place such an unfolding as shown in Figure \ref{fig:twocylinderlattice}. Color those squares of the surface, the lower left vertex of which is a point of the lattice $L$. Let us prove that the set $\Delta$ of numbers of the colored squares is a block for the group $G$.

Indeed, consider an arbitrary square of the surface, and let $(p,q)$ be the coordinates of its  lower left vertex. Then, when acting by the permutations $s$ or $t$, this square will be sent to the square of the surface with the following coordinates of its  lower left vertex: 
$$(p+1,q) + \ep_{1}\vv{(k,0)} + \ep_{2}\vv{(\ell,0)} \qquad\textrm{or}\qquad
(p,q+1) - \ep_{3}\vv{(\al+\be,a+b)} - \ep_{4}\vv{(\be,b)}$$
respectively, where $\ep_{1}, \ep_{2}, \ep_{3}, \ep_{4}\in \set{0,1}$. Hence, under the action of any permutation $g\in G$ (which is of course a word on the letters $s$ and $t$), the coordinates of the lower left vertices of the colored squares of the surface will be translated by the same vector $\vv{v}\in \Z^{2}$ up to linear combinations of type $\ep_{1}\vv{(k,0)} + \ep_{2}\vv{(\ell,0)}$ and $-\ep_{3}\vv{(\al+\be,a+b)} - \ep_{4}\vv{(\be,b)}$. By definition, the lattice $L$  consists of all linear combinations
$$m_{1}\cdot \vv{(\al,a)} + m_{2}\cdot \vv{(\be,b)} + m_{3}\cdot \vv{(k,0)} + m_{4}\cdot \vv{(\ell,0)}\,.$$ 
Therefore, under the action of a permutation $g$, either the set of all marked squares will not change, or it will be sent to a set of squares not intersecting with it, \ie
$$g(\Delta) = \Delta \qquad\textrm{or}\qquad g(\Delta)\cap\Delta = \emptyset$$ 
depending on whether the vector $\vv{v}$ belongs to the lattice $L$ or not.
We conclude that $\Delta$ is a block for the group $G$.
Since the lattice $L$ is distinct from $\Z^{2}$, the block $\Delta$ is nontrivial and the group $G$  is not primitive.
\begin{figure}[htbp]
   \begin{center}
   \begin{tikzpicture}[scale=.4, >=latex,shorten >=0pt,shorten <=0pt,line width=1pt,bend angle=20, 
   		circ/.style={circle,inner sep=2pt,draw=black!40,rounded corners, text centered},
   		tblue/.style={dashed,blue,rounded corners}]
	
      	\draw[help lines] (-15,-8) grid (15,6);
      	\node[left] at (-11.8,-6.3){\scriptsize $(0,0)$};
      
	\begin{scope}
      	\clip (-12,-6) -- ++(3,6) -- ++(4,4) -- ++(15,0) -- ++(-4,-4) -- ++(-6,0) -- ++(-3,-6);
	\foreach \x in {-8,...,8}
		\foreach \y in {-8,...,8}
			\fill[green!60!white] (\x*3 + \y*2,\y*2) rectangle +(1,1);
	\end{scope}

      	\clip (-15.5,-8.5) rectangle (15.5,6.5);
	\foreach \x in {-8,...,8}
		\foreach \y in {-8,...,8}
      			\fill[blue] (\x*3 + \y*2,\y*2) circle (5pt); 

      	\path[->] 	(-12,-6) 	edge node[below,yshift=1pt]{\scriptsize $\vv{(k,0)}$} +(9,0)
      					edge node[above,xshift=-12pt,yshift=-6pt]{\scriptsize $\vv{(\al,a)}$} +(3,6)
			(-9,0) 	edge node[above,xshift=-10pt]{\scriptsize $\vv{(\be,b)}$} +(4,4)
			(-5,4) 	edge node[above,xshift=0pt,yshift=-1pt]{\scriptsize $\vv{(\ell,0)}$} +(15,0);
			
	\draw (10,4) -- ++(-4,-4) -- ++(-6,0) -- ++(-3,-6);
   \end{tikzpicture}
   \captionwidth=14cm 
   \caption{The marked points belong to the lattice $L$ generated by the vectors with coordinates $(\al,a)$, $(\be,b)$, $(k,0)$ and $(\ell,0)$. The colored squares of the surface form a nontrivial block for the group $G$.}
   \label{fig:twocylinderlattice}
   \end{center} 
\end{figure}
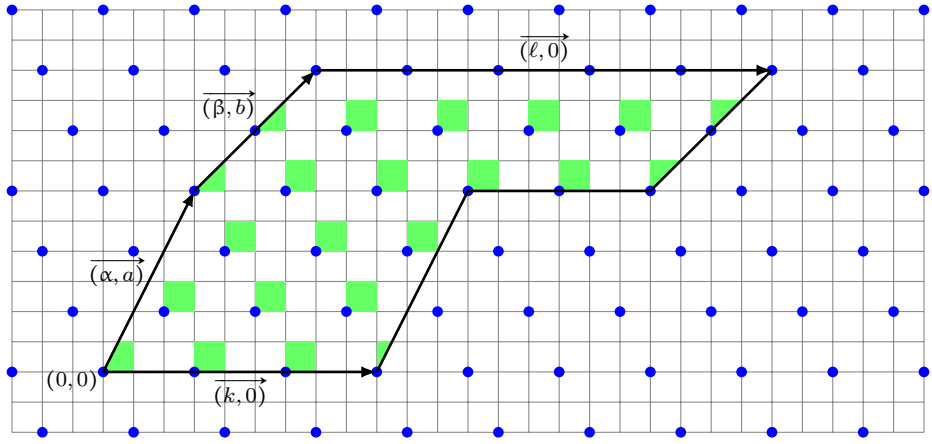

\vs  
\ifonly Suppose now that the conditions $a\land b = 1$ and $k\land \ell\land (a\be - b\al) = 1$ are satisfied, and let us prove that the square-tiled surface is primitive. Unwrap the surface in the real plane as shown in Figure \ref{fig:twocylprimeP}. Consider the set $P$ of points of the plane with integer coordinates, which are situated strictly inside the unfolded surface of on its left and lower borders, excluding the points with coordinates $(k,0)$, $(\al+\ell,a)$ and $(\al+\be,a+b)$. This set consists of  $n$  vertices of squares of the surface. Certainly, the action of the permutation group  $G = \group{s,t}$ on the set $\lsem 1,n\rsem$, on the set of squares of the surface and on the set of points $P$ are equivalent. 

\begin{figure}[htbp]
   \begin{center}
   \begin{tikzpicture}[scale=.6, >=latex,shorten >=0pt,shorten <=0pt,line width=1pt,bend angle=20, 
   		circ/.style={circle,inner sep=2pt,draw=black!40,rounded corners, text centered},
   		tblue/.style={dashed,blue,rounded corners}]
	
      	\draw[help lines] (-3,-2) grid (13,7);
      	\node[left] at (0,-.3){\scriptsize $(0,0)$};

	\foreach \x/\y/\n in 
		{0/0/1,1/0/2,2/0/3,3/0/4,
		1/1/5,2/1/6,3/1/7,4/1/8,
		1/2/9,2/2/10,3/2/11,4/2/12,5/2/13,6/2/14,7/2/15,
		2/3/16,3/3/17,4/3/18,5/3/19,6/3/20,7/3/21,8/3/22,
		3/4/23,4/4/24,5/4/25,6/4/26,7/4/27,8/4/28,9/4/29}
	{
		\fill[blue] (\x,\y) circle (3pt);
		\node[right,xshift=1pt,yshift=5pt] at (\x,\y) {\footnotesize \n};
	}
	      
      	\path[->,>=stealth] 	(0,0) 		edge node[below]{\scriptsize $\vv{(k,0)}$} +(4,0)
      					edge node[above,xshift=-13pt,yshift=-5pt]{\scriptsize $\vv{(\al,a)}$} +(1,2)
			(1,2) 		edge node[above,xshift=-10pt]{\scriptsize $\vv{(\be,b)}$} +(3,3)
			(4,5) 		edge node[above,xshift=0pt]{\scriptsize $\vv{(\ell,0)}$} +(7,0);
			
	\draw (11,5) -- ++(-3,-3) -- ++(-3,0) -- ++(-1,-2);
   \end{tikzpicture}
   \captionwidth=14cm 
   \caption{An unfolding of the square-tiled surface and the set $P$ of marked points.}
   \label{fig:twocylprimeP}
   \end{center} 
\end{figure}
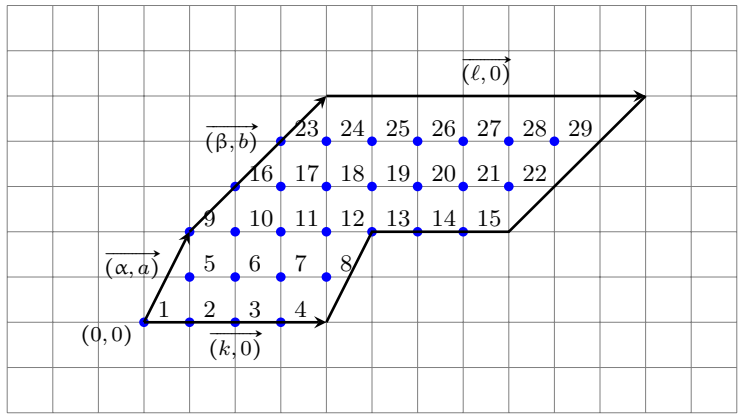

   Let $\Delta_{1}$ be an arbitrary block for $G$ containing at least two points. Since $G$ acts on the set $P$ transitively (the square-tiled surface is connected), then $P$ decomposes into several disjoint blocks of the same length:
\begin{equation*}
	P = \Delta_{1} \sqcup \ldots \sqcup \Delta_{m}\,,\quad \textrm{ where } m\in\N\, .
\end{equation*}    
Let us tile the plane with copies of the unfolded surface as shown in Figure \ref{fig:twocylprimePP}. An infinite number of equal parallelograms, however, will stay uncovered, we will call them \textbf{teleports}. In each copy of the surface we shall mark the same integer points as in the initial one. 
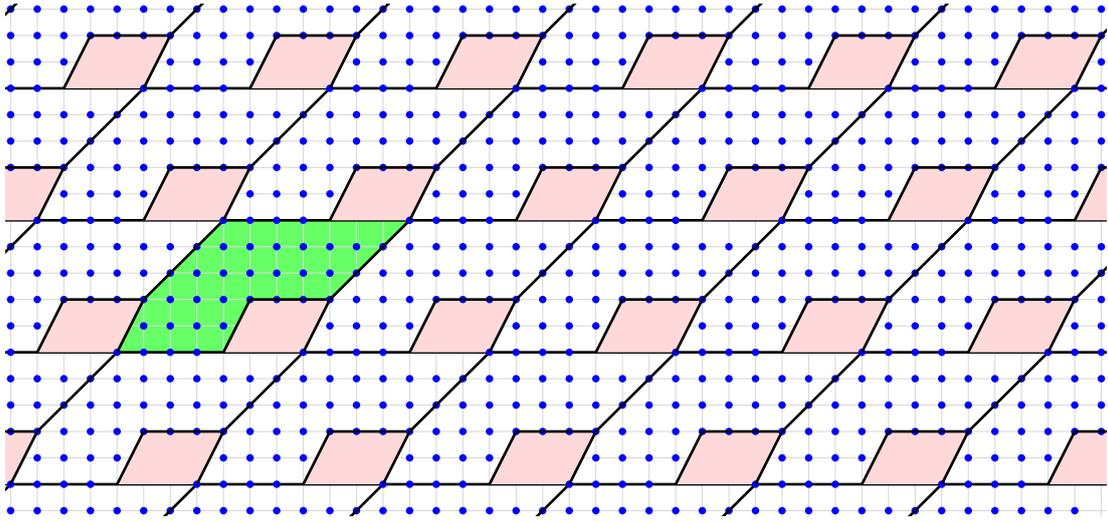
\begin{figure}[t]
   \begin{center}
   \begin{tikzpicture}[scale=.35, >=latex,shorten >=0pt,shorten <=0pt,line width=1pt,bend angle=20, 
   		circ/.style={circle,inner sep=2pt,draw=black!40,rounded corners, text centered},
   		tblue/.style={dashed,blue,rounded corners}]
	
	\clip (-4.2,-6.2) rectangle (37.2,13.2);

	\fill[green!60!white] (0,0) -- ++(1,2) -- ++(3,3) -- ++(7,0) -- ++(-3,-3) -- ++(-3,0) -- ++(-1,-2) -- ++(-4,0);
      	\draw[line width=.4pt,black!15!white] (-10,-10) grid (40,20);

	\foreach \A in {-21,-14,-7,0,7,14,21,28,35}
	\foreach \B/\shift in {-10/-8,-5/-4, 0/0, 5/4, 10/8, 15/12}
	{
		\begin{scope}[xshift=\A cm +\shift cm,yshift=\B cm]
		\fill[red!15!white] (4,0) -- ++(1,2) -- ++(3,0) -- ++(-1,-2) -- ++(-3,0);
		\foreach \x/\y/\n in 
			{0/0/1,1/0/2,2/0/3,3/0/4, 1/1/5,2/1/6,3/1/7,4/1/8,
			1/2/9,2/2/10,3/2/11,4/2/12,5/2/13,6/2/14,7/2/15,
			2/3/16,3/3/17,4/3/18,5/3/19,6/3/20,7/3/21,8/3/22,
			3/4/23,4/4/24,5/4/25,6/4/26,7/4/27,8/4/28,9/4/29}
			\fill[blue] (\x,\y) circle (4pt);

		\draw (0,0) -- ++(1,2) -- ++(3,3) -- ++(7,0) 
			++(-3,-3) -- ++(-3,0) -- ++(-1,-2);
		\end{scope}
	}
   \end{tikzpicture}
   \captionwidth=14cm 
   \caption{A tiling of  the plane with copies of the unfolded square-tiled surface. An infinite number of equal parallelograms will be uncovered (teleports). The marked points form the set $\widehat{P}$.}
   \label{fig:twocylprimePP}
   \end{center} 
\end{figure}
We conclude that the following points of the plane will be marked: 
\begin{itemize}
   \item[--] all integer points outside the teleports, 
   \item[--] all integer points on the upper and right sides of the teleports. 
\end{itemize}
Consider the set $\widehat{P}$ of marked points. Denote by $\widehat{s}$ and $\widehat{t}$ the elements of the symmetric group $Sym(\widehat{P})$ which permute the points $(p,q)\in \widehat{P}$ in the following way:
\vs
$\qquad \widehat{s} (p,q) \;\;=\;\; 
  \begin{cases}
  \ (p+1,q), & \textrm{if }  (p+1,q) \textrm { is marked}; \\
  \ (p+1,q)+(\ell-k,0), & \textrm{otherwise}, \\
  \end{cases}
$ 
\vs
$
 \qquad \widehat{t} (p,q) \;\;=\;\;
  \begin{cases}
 	\ (p,q+1), & \textrm{if } (p,q+1)  \textrm{ is marked}; \\
 	\ (p,q+1) + (\al,a), & \textrm{if } (p,q+1) \textrm{ lies on the lower side of a teleport} \\
	& \textrm{and doesn't coincides with the right end of the side}, \\
	& \textrm{and } (p,q) \textrm{ is situated under the teleport}; \\		 
 	\ (p,q+1) - (\ell-k,0), & \textrm{if } (p,q+1) \textrm{ lies strictly inside a teleport}, \\
	& \textrm{and } (p,q) \textrm{ is situated to the right of the teleport}, \\
	& \textrm{including its right side}. 		 
  \end{cases}
$  

Let $\widehat{G}$ be the permutation group generatеd by the pair $\widehat{s}$\, and\, $\widehat{t}$. For all $i\in\lsem 1, m\rsem$, denote by $\widehat{\Delta}_{i}$ the subset $\widehat{P}$ consisting of the block $\Delta_{i}$ and all its duplicates in the copies of the unfolded surface. Then we have a decomposition of the set $\widehat{P}$ into disjoint parts:
  $$\widehat{P}=\widehat{\Delta}_1 \sqcup \ldots \sqcup \widehat{\Delta}_m\,,$$
  which is $\widehat{G}$-invariant, \ie under the action of any permutation from $\widehat{G}$, each part $\widehat{\Delta}_{i}$ is sent to some $\widehat{\Delta}_{j}$\,. (It suffices to check this for the permutations $\widehat{s}$ and $\widehat{t}$, since they generate the group $\widehat{G}$.) Hence, all subsets $\widehat{\Delta}_1,\ldots,\widehat{\Delta}_m \subseteq \widehat{P}$ are blocks for the group $\widehat{G}$. We are going to show that  $\widehat{\Delta}_1=\widehat{P}$, \ie $m=1$. 

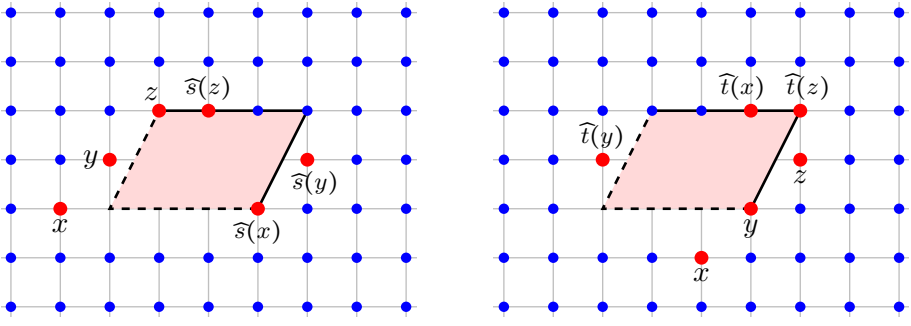
\begin{figure}[htbp]
   \begin{center}
   \begin{tikzpicture}[scale=.65, >=latex,shorten >=0pt,shorten <=0pt,line width=1pt,bend angle=20, 
   		circ/.style={circle,inner sep=2pt,draw=black!40,rounded corners, text centered},
   		tblue/.style={dashed,blue,rounded corners}]
	
	\clip (1.8,-2.2) rectangle (10.2,4.2);

      	\draw[line width=.5pt,black!25!white] (-6,-6) grid (15,15);

	\foreach \A in {-7,0,7,14}
	\foreach \B/\shift in {-5/-4, 0/0, 5/4}
	{
		\begin{scope}[xshift=\A cm +\shift cm,yshift=\B cm]
		\fill[fill=red!15!white] (4,0) -- ++(1,2) -- ++(3,0) -- ++(-1,-2) -- ++(-3,0);
		\draw[dashed] (7,0) -- ++(-3,0) -- ++(1,2);
		\draw (5,2) -- ++(3,0) -- ++(-1,-2);
		\foreach \x/\y/\n in 
			{0/0/1,1/0/2,2/0/3,3/0/4, 1/1/5,2/1/6,3/1/7,4/1/8,
			1/2/9,2/2/10,3/2/11,4/2/12,5/2/13,6/2/14,7/2/15,
			2/3/16,3/3/17,4/3/18,5/3/19,6/3/20,7/3/21,8/3/22,
			3/4/23,4/4/24,5/4/25,6/4/26,7/4/27,8/4/28,9/4/29}
			\fill[blue] (\x,\y) circle (3pt);

		\end{scope}
	}
	
	\fill[red] (3,0) node[below,black]{$x$} circle (4pt);
	\fill[red] (4,1) node[left,black]{$y$} circle (4pt);
	\fill[red] (5,2) node[above,black,xshift=-3pt]{$z$} circle (4pt);
	\fill[red] (7,0) node[below,black]{\footnotesize $\widehat{s}(x)$} circle (4pt);
	\fill[red] (8,1) node[below,black,xshift=3pt]{\footnotesize $\widehat{s}(y)$} circle (4pt);
	\fill[red] (6,2) node[above,black]{\footnotesize $\widehat{s}(z)$} circle (4pt);
   \end{tikzpicture}
   $\qquad$
   \begin{tikzpicture}[scale=.65, >=latex,shorten >=0pt,shorten <=0pt,line width=1pt,bend angle=20, 
   		circ/.style={circle,inner sep=2pt,draw=black!40,rounded corners, text centered},
   		tblue/.style={dashed,blue,rounded corners}]
	
	\clip (1.8,-2.2) rectangle (10.2,4.2);

      	\draw[line width=.5pt,black!25!white] (-6,-6) grid (15,15);

	\foreach \A in {-7,0,7,14}
	\foreach \B/\shift in {-5/-4, 0/0, 5/4}
	{
		\begin{scope}[xshift=\A cm +\shift cm,yshift=\B cm]
		\fill[fill=red!15!white] (4,0) -- ++(1,2) -- ++(3,0) -- ++(-1,-2) -- ++(-3,0);
		\draw[dashed] (7,0) -- ++(-3,0) -- ++(1,2);
		\draw (5,2) -- ++(3,0) -- ++(-1,-2);
		\foreach \x/\y/\n in 
			{0/0/1,1/0/2,2/0/3,3/0/4, 1/1/5,2/1/6,3/1/7,4/1/8,
			1/2/9,2/2/10,3/2/11,4/2/12,5/2/13,6/2/14,7/2/15,
			2/3/16,3/3/17,4/3/18,5/3/19,6/3/20,7/3/21,8/3/22,
			3/4/23,4/4/24,5/4/25,6/4/26,7/4/27,8/4/28,9/4/29}
			\fill[blue] (\x,\y) circle (3pt);

		\end{scope}
	}
	
	\fill[red] (6,-1) node[below,black]{$x$} circle (4pt);
	\fill[red] (7,0) node[below,black]{$y$} circle (4pt);
	\fill[red] (8,1) node[below,black]{$z$} circle (4pt);
	\fill[red] (7,2) node[above,black,xshift=-3pt]{\footnotesize $\widehat{t}(x)$} circle (4pt);
	\fill[red] (4,1) node[above,black]{\footnotesize $\widehat{t}(y)$} circle (4pt);
	\fill[red] (8,2) node[above,black,xshift=3pt]{\footnotesize $\widehat{t}(z)$} circle (4pt);
   \end{tikzpicture}
   \captionwidth=14cm 
   \caption{Actions of the permutations $\widehat{s}$ and $\widehat{t}$ on the set $\widehat{P}$.}
   \label{fig:twocylprimest}
   \end{center} \vspace{-0.3cm}
\end{figure}
  
  Connect two marked points $x$ and $y$ by a segment if they lie in the same block $\widehat{\Delta}_i$ for some $i\in\lsem 1, m\rsem$. Denote by $seg(\widehat{P})$ the set of obtained segments in the plane. To each segment $x y$, where $x\neq y$, corresponds two vectors $\vv{x y}$ and $\vv{y x}$ with integer coordinates, and to the segments $x x$ corresponds the zero-vector $\vv{0}$. Let $V$ be the set of the vectors corresponding to the segments from $seg(\widehat{P})$. 
  The set $V$ has the following properties:

\begin{itemize}
   \item[1.] \emph{If\, $\vv{x y} \in V$, then also\, $\vv{x y} \pm \vv{(\ell,0)} \in V$.} 
   
   Indeed, by construction of the blocks $\widehat{\Delta}_{i}$\,, together with each point $y$ a block also contains the points $y \pm \vv{(\ell,0)}$.

   \item[2.] \emph{If\, $\vv{x y} \in V$, then also\, $\vv{x y} \pm \vv{(\al+\be,a+b)} \in V$.} 
   
   Indeed, by construction of the blocks $\widehat{\Delta}_{i}$\,, together with each point $y$ a block also contains  the points $y \pm \vv{(\al+\be,a+b)}$.

   \item[3.] \emph{If\, $\vv{x y} \in V$ and the vector $\vv{x y}$ cannot be be presented as a linear combination  
   $$c\cdot\vv{(\ell,0)} + d\cdot \vv{(\al+\be,a+b)}$$
   with integer coefficients $c$ and $d$, then the vectors $\vv{x y} \pm \vv{(\ell-k,0)}$ also belong to $V$. }
   
   We can assume that the coordinates of the vector $\vv{x y}=\vv{(p,q)}$ satisfy the inequalities 
   $$|p| < l \quad \textrm{and}\quad |q|< a+b$$ 
   in view of already proven properties 1 and 2. Besides, the vector $\vv{x y}$ is nonzero, since it is not a linear combination of the vectors $\vv{(\al+\be,a+b)}$ and $\vv{(\ell,0)}$. Let us, moreover, suppose that $p\ge 0$ and $q\ge 0$ (other cases can be treated similarly). The vector $\vv{x y}$ corresponds to a segment $x y$ from the set $seg(\widehat{P})$. Consider the following two situations:
   
   a) Suppose that the vector $\vv{x y}$ is not horizontal (its ordinate is nonzero). Then we can apply the permutation $\widehat{s}$ and $\widehat{t}$ to the points $x$ and $y$ simultaneously, so that
   \begin{itemize}
	\item[--] the image $x_{1}$ of the point $x$ lies on a horizontal line intersecting the interior of teleports or their lower sides;
	
   	\item[--] the image $y_{1}$ of the point $y$ lies on a horizontal line not containing teleports or intersecting teleports by upper sides;
	
	\item[--] moreover, $\vv{x_{1}y_{1}} = \vv{x y}$.
   \end{itemize}
   Such a segment $x_{1}y_{1}$ will belong to the set $seg(\widehat{P})$, since if $x$ and $y$ are in the same blockе $\widehat{\Delta}_{i}$\,, then the images of these points for the action of an arbitrary permutation $\widehat{w}\in\widehat{G}$ also belong to the same block, namely $\widehat{\Delta}_{j}=\widehat{w}(\widehat{\Delta}_{i})$. Let us now apply the permutation $\widehat{s}$or $\widehat{s}^{\; -1}$ several times to the points $x_{1}$ and $y_{1}$\,, so that the point $x_{1}$ pass through the teleport situated respectively to the right or to the left of it (\cf Figure \ref{fig:twocylprimelk}). A new segment $x_{2}y_{2}$ will be an element of the set $seg(\widehat{P})$, and
   $$\vv{x_{2}y_{2}} = \vv{x y} - \vv{(\ell-k,0)}\qquad\textrm{or}\qquad \vv{x_{2}y_{2}} = \vv{x y} + \vv{(\ell-k,0)}$$
   respectively, as required.
 
 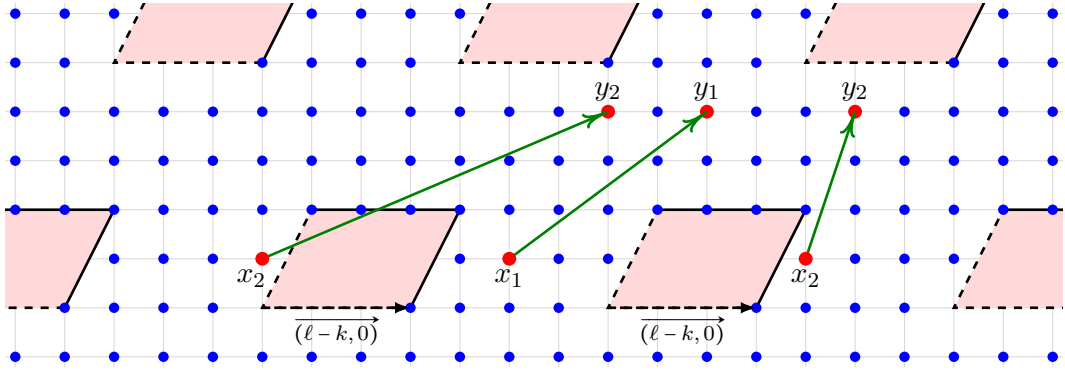
\begin{figure}[htbp]
   \begin{center}
   \begin{tikzpicture}[scale=.65, >=latex,shorten >=0pt,shorten <=0pt,line width=1pt,bend angle=20, 
   		circ/.style={circle,inner sep=2pt,draw=black!40,rounded corners, text centered},
   		tblue/.style={dashed,blue,rounded corners}]
	
	\clip (-8.2,-1.2) rectangle (13.2,6.2);

      	\draw[line width=.4pt,black!15!white] (-10,-6) grid (15,15);

	\foreach \A in {-14,-7,0,7,14}
	\foreach \B/\shift in {-5/-4, 0/0, 5/4}
	{
		\begin{scope}[xshift=\A cm +\shift cm,yshift=\B cm]
		\fill[fill=red!15!white] (4,0) -- ++(1,2) -- ++(3,0) -- ++(-1,-2) -- ++(-3,0);
		\draw[dashed] (7,0) -- ++(-3,0) -- ++(1,2);	
		\draw (5,2) -- ++(3,0) -- ++(-1,-2);	
		\foreach \x/\y/\n in 
			{0/0/1,1/0/2,2/0/3,3/0/4, 1/1/5,2/1/6,3/1/7,4/1/8,
			1/2/9,2/2/10,3/2/11,4/2/12,5/2/13,6/2/14,7/2/15,
			2/3/16,3/3/17,4/3/18,5/3/19,6/3/20,7/3/21,8/3/22,
			3/4/23,4/4/24,5/4/25,6/4/26,7/4/27,8/4/28,9/4/29}
			\fill[blue] (\x,\y) circle (3pt);

		\end{scope}
	}

      	\path[->,dashed,>=latex] 	
		(-3,0) edge node[below]{\scriptsize $\vv{(\ell-k,0)}$} +(3,0)
		(4,0) 	edge node[below]{\scriptsize $\vv{(\ell-k,0)}$} +(3,0);

	\fill[red] (6,4) node[above,black]{$y_{1}$} circle (4pt);
	\fill[red] (4,4) node[above,black]{$y_{2}$} circle (4pt);
	\fill[red] (9,4) node[above,black]{$y_{2}$} circle (4pt);

	\path[->,>=arcs,green!50!black] 	
			(2,1) edge +(4,3)
			(-3,1) edge +(7,3)
			(8,1) edge +(1,3);

	\fill[red] (-3,1) node[below,xshift=-4pt,black]{$x_{2}$} circle (4pt);
	\fill[red] (8,1) node[below,black]{$x_{2}$} circle (4pt);
	\fill[red] (2,1) node[below,black]{$x_{1}$} circle (4pt);
   \end{tikzpicture}
   \captionwidth=14cm 
   \caption{Vectors\, $\overline{x_{1}y_{1}}$\; and\; $\overline{x_{2}y_{2}}=\overline{x_{1}y_{1}} \pm \overline{(\ell-k,0)}$ belong to the set $V$.}
   \label{fig:twocylprimelk}
   \end{center} \vspace{-0.3cm}
\end{figure}
   
   b) Suppose that the vector $\vv{x y}$ is horizontal. If $\vv{x y} + \vv{(\ell-k,0)} = \vv{0}$ then we are done, since by definition $\vv{0}\in V$. Let now $\vv{x y} + \vv{(\ell-k,0)}$ be different from $\vv{0}$. Since the vector $\vv{x y}$ is not collinear to the vector $\vv{(\al,a)}$, then according to the propertу 4.a) below, there exists a segment $x_{1} y_{1}\in seg(\widehat{P})$ such that 
   $$\vv{x_{1}y_{1}} = \vv{x y} + \vv{(\al,a)}\,.$$
   The vector $\vv{x_{1}y_{1}}$ is no more horizontal, and so by already proven property 3.a) we get a segment $x_{2}y_{2}\in seg(\widehat{P})$, for wich
   $$\vv{x_{2}y_{2}} \;=\; \vv{x_{1} y_{1}} + \vv{(\ell-k,0)} \;=\; \vv{x y} + \vv{(\ell-k,0)} + \vv{(\al,a)}\,.$$
   The vector $\vv{x y} + \vv{(\ell-k,0)}$ is not horizontal and not equal to $\vv{0}$, from where follows that $\vv{x_{2}y_{2}}$ and $\vv{(\al,a)}$ are not collinear. From the property 4.a) for the vector $\vv{x_{2}y_{2}}$ follows that the vector
   $$\vv{x_{3}y_{3}} \;=\; \vv{x_{2} y_{2}} - \vv{(\al,a)} \;=\; \vv{x y} + \vv{(\ell-k,0)}$$
   also belongs to the set $V$.
   
   To construct the vector $\vv{x y} - \vv{(\ell-k,0)}$, one proceeds by analogy.
   
%
%
%

   \item[4.] \emph{If\, $\vv{x y} \in V$ and the vector $\vv{x y}$ cannot be presented as a linear combination  
   $$c\cdot\vv{(\ell,0)} + d\cdot \vv{(\al+\be,a+b)}$$
   with integer coefficients $c$ and $d$, then the vectors $\vv{x y} \pm \vv{(\al,a)}$ also belong to $V$. }
   
   As for the property 3, we consider two situations:
   
   a) The vector $\vv{x y}$ is not collinear to the vector $\vv{(\al,a)}$. One can show this  by analogy with the property 3.a).
   
   b) The vector $\vv{x y}$ is collinear to the vector $\vv{(\al,a)}$. A proof uses the properties 3.a) and 4.a).  
\end{itemize}

Let us show now that $V=\Z^{2}$. By the initial assumption, the block $\Delta_{1}$ for $G$ contains at least two points $x, y\in P$. These points also lie in the block $\widehat{\Delta}_{1}$ for the group $\widehat{G}$. Hence, in the set $V$ there is at least one vector which cannot be presented as a  linear combination  $c\cdot\vv{(\ell,0)} + d\cdot \vv{(\al+\be,a+b)}$ with integer coefficients $c$ and $d$, namely the vector $\vv{x y}$.

\vs
  Let $\vv{v}\in \Z^{2}$ be an arbitrary vector with integer coordinates. According to Lemma \ref{lemm:twocylinderprim0}, the vectors 
  $$\vv{(\ell,0)}, \qquad 
  \vv{(\al+\be,a+b)} = \vv{(\al,a)} + \vv{(\be,b)}, \qquad 
  \vv{(\ell - k,0)}  = \vv{(\ell,0)} - \vv{(k,0)} \quad\textrm{ and }\quad 
  \vv{(\al,a)}$$
generate the lattice $\Z^{2}$. Therefore, the vector $\vv{v} - \vv{x y}\in \Z^{2}$ is of the form 
$$\vv{v} - \vv{x y} \;=\; m_1 \cdot \vv{(\ell,0)} + m_2 \cdot \vv{(\al+\be,a+b)} + m_3 \cdot \vv{(\ell - k,0)} + m_{4}\cdot \vv{(\al,a)}$$
for some integers $m_{1}$, $m_{2}$, $m_{3}$ and $m_{4}$. Among all such presentations, we shall choose that, in which the sum $|m_{3}| + |m_{4}|$ is minimal. From the properties 3 and 4 follows that the vector 
$$\vv{w} \;=\; \vv{x y} + m_1 \cdot \vv{(\ell,0)} + m_2 \cdot \vv{(\al+\be,a+b)}$$
belongs to the set $V$. 
Since the vector $\vv{x y}$ is not a linear combination  $c\cdot\vv{(\ell,0)} + d\cdot \vv{(\al+\be,a+b)}$ with integer $c$ and $d$, then this is also true for the vector $\vv{w}$. We may apply the property 3 for $|m_{3}|$ times, and after that the property 4 for $|m_{4}|$ times, in order to show that  together with the vector $\vv{w}$ the set $V$ also contains the vector 
$$\vv{v} \;=\; \vv{w} + m_3 \cdot \vv{(\ell - k,0)} + m_{4}\cdot \vv{(\al,a)}.$$
Indeed, in view of the  minimality of the expression $|m_{3}| + |m_{4}|$\,, each of the $|m_{3}|+ |m_{4}| - 1$ times that we will apply the properties 3 and 4, only those vectors will occur which are not linear combinations $c\cdot\vv{(\ell,0)} + d\cdot \vv{(\al+\be,a+b)}$ for any  integers $c$ and $d$. 

We conclude that the set $V$ contains all vector with integer coordinates, and in particular, the vectors $\vv{(1,0)}$ and $\vv{(0,1)}$. This means that any two adjacent (vertically or horizontally) marked points lie in the same block, \ie \emph{all} marked points lie in the same block: $\widehat{\Delta}_{1}=\widehat{P}$.  Hence, the block $\Delta_{1}$ for $G$ is trivial (coincides with the set $P$), from where follows the primitivity of the group $G$.
\end{proof}


\vs
\begin{lemm}\label{lemm:twocylinderprim1} 
Let $a, b, k, \ell$ be positive integers with $a$ and $b$ being coprime. The number of distinct pairs $(\al,\be)$ of integers such that
\begin{equation}\label{eq:akbl}
0\le \al < k, \quad 0\le \be < \ell \qquad\textrm{and}\qquad k\land \ell\land (a\be - b\al) = 1,
\end{equation}
is equal to\, $k \ell\,\frac{\phi(k\land \ell)}{k\land \ell}$.
\end{lemm}
\begin{proof} 
Let $k=k' d$ and $\ell=\ell' d$, where $d=k\land l$ is the greatest common divisor. Denote by $f(k,l)$ the number of pairs $(\al,\be)$ of integers satisfying the conditions \eqref{eq:akbl}. Notice that if for some $\al', \be'\in\Z$ the differences $\al - \al'$ and $\be - \be'$ are divisible by $d$\,, then 
$$d\land (a\be - b\al) = d \land \bigg( a(\be-\be') - b(\al-\al') + a\be' - b\al' \bigg) = d \land (a\be' - b\al') =1.$$
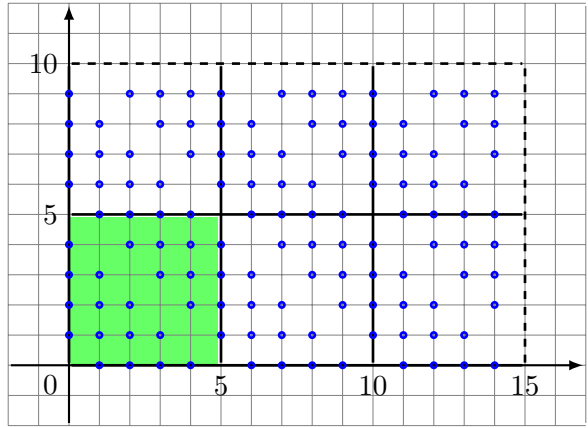
\begin{figure}[htbp]
   \begin{center}
   \begin{tikzpicture}[scale=.4, >=latex,shorten >=1pt,shorten <=1pt,line width=1pt,bend angle=20, 
   		circ/.style={circle,inner sep=2pt,draw=black!40,rounded corners, text centered},
   		tblue/.style={dashed,blue,rounded corners}]

      \fill[green!60!white] (0,0) rectangle (4.9,4.9);
      \draw[help lines] (-2,-2) grid (17,12);
      \draw[->,line width=.8pt] (-2,0) -- (17,0);
      \draw[->,line width=.8pt] (0,-2) -- (0,12);
      \node[left] at (0,-.65){$0$};
      \draw (0,0) -- (15,0); 
      \foreach \y in {5} 
      	\draw (0,\y) node[left]{$\y$} -- ++(15,0);
      \draw[dashed] (0,10) node[left]{$10$} -- ++(15,0);

      \draw (0,0) -- (0,10);
      \foreach \x in {5,10}
       \draw (\x,0) node[below]{$\x$} -- ++(0,10);
      \draw[dashed] (15,0) node[below]{$15$} -- ++(0,10);
      
      \foreach \xsh in {0,5,10} {
      \foreach \ysh in {0,5}{
      	\begin{scope}[xshift=\xsh cm, yshift = \ysh cm]
		\foreach \y/\x in {0/1,0/2,0/3,0/4,1/0,1/1,1/2,1/3,2/0,2/1,2/2,2/4,3/0,3/1,3/3,3/4,4/0,4/2,4/3,4/4}
      			\draw[blue] (\x,\y) circle (2.5pt); 
	\end{scope}
      }
      }
   \end{tikzpicture}
   \caption{Integer points $(\al,\be)\in[0,15[\times[0,10[$ such that $15\land 10\land (3\be - 2\al)=1$.}
   \label{fig:integerpoints}
   \end{center} 
\end{figure}
Thus, in order to count the number of integer points $(\al,\be)$ in the semi-open rectangle $[0,k[\times [0,\ell[$ with condition $d\land (a\be - b\al)$, it suffices to found the number of such points in the square $[0,d[\times [0,d[$. More precisely, by proving the equality $f(d,d) = d\cdot\phi(d)$, we will get the required: 
$$f(k,l) = k'\ell' f(d,d) = k'\ell' d\phi(d) = (k'd) (\ell'd) \frac{\phi(d)}{d} = k \ell\,\frac{\phi(k\land \ell)}{k\land \ell}.$$
See an example in Figure \ref{fig:integerpoints} for $a=3$, $b=2$, $k=15$, $\ell=10$, $d = 5$.

\begin{figure}[htbp]
   \begin{center}
   \begin{tikzpicture}[scale=.35, >=latex,shorten >=0pt,shorten <=0pt,line width=1pt,bend angle=20, 
   		circ/.style={circle,inner sep=2pt,draw=black!40,rounded corners, text centered},
   		tblue/.style={dashed,blue,rounded corners}]
	
      \node at (0,-6) {};	
      \fill[green!60!white] (0,0) rectangle (4.9,4.9);
      \draw[help lines] (-2,-2) grid (7,7);
      \draw[->,line width=.6pt] (-2,0) -- (7,0);
      \draw[->,line width=.6pt] (0,-2) -- (0,7);
      \node[left] at (0,-.65){$0$};
      \draw[green!60!black] (0,0) -- (5,0); 
      \draw[densely dashed,green!60!black] (5,5) -- (0,5);
      \draw (0,5) node[left]{$5$};

      \draw[green!60!black] (0,0) -- (0,5);
      \draw[densely dashed,green!60!black] (5,5) -- (5,0);
      \draw (5,0) node[below]{$5$};
      
		\foreach \y/\x in {0/1,0/2,0/3,0/4,1/0,1/1,1/2,1/3,2/0,2/1,2/2,2/4,3/0,3/1,3/3,3/4,4/0,4/2,4/3,4/4}
      			\draw[blue] (\x,\y) circle (3pt); 
   \end{tikzpicture}
   \begin{tikzpicture}[scale=.4, >=latex,shorten >=0pt,shorten <=0pt,line width=1pt,bend angle=20, 
   		circ/.style={circle,inner sep=2pt,draw=black!40,rounded corners, text centered},
   		tblue/.style={dashed,blue,rounded corners}]
	\node (a) at (0,-8) {};	
	\path[->] (0,0) edge[bend left,bend angle=30] node[above]{\small $A=\begin{pmatrix} -2 & 3 \\ 1 & -2 \end{pmatrix}$} (6,0);
   \end{tikzpicture}
   \begin{tikzpicture}[scale=.35, >=latex,shorten >=0pt,shorten <=0pt,line width=1pt,bend angle=20, 
   		circ/.style={circle,inner sep=2pt,draw=black!40,rounded corners, text centered},
   		tblue/.style={dashed,blue,rounded corners}]

      \fill[green!60!white] (-10,5) -- (0,0) -- (15,-10) -- (5,-5) -5 (-10,5);
      \draw[help lines] (-12,-12) grid (17,7);
      \draw[->,line width=.6pt] (-12,0) -- (17,0);
      \draw[->,line width=.6pt] (0,-12) -- (0,7);
      \node[right] at (0,.65){$0$};
      \node[right] at (0,5){$5$};
      \node[below,xshift=-.1cm] at (-10,0) {$-10$};
      \node[below,xshift=-.1cm] at (-5,0) {$-5$};
      \node[above] at (5,0) {$5$};
      \node[above] at (15,0) {$15$};
      \node[left] at (0,-5) {$-5$};
      \node[left] at (0,-10) {$-10$};
      
		\draw[green!60!black] (0,0) -- (-10,5); 
		\draw[densely dashed,green!60!black] (-10,5) -- (5,-5);
		\draw[green!60!black] (0,0) -- (15,-10);
		\draw[densely dashed,green!60!black] (5,-5) -- (15,-10);
		      
		\foreach \y/\x in {0/1,0/2,0/3,0/4,1/0,1/1,1/2,1/3,2/0,2/1,2/2,2/4,3/0,3/1,3/3,3/4,4/0,4/2,4/3,4/4}
		      			\draw[blue] (-2*\x + 3*\y,\x - 2*\y) circle (3pt); 
   \end{tikzpicture}
   \caption{Integer $5$-prime points in the parallelogram $A([0,5[\times[0,5[)$.}
   \label{fig:dprime}
   \end{center} 
\end{figure}
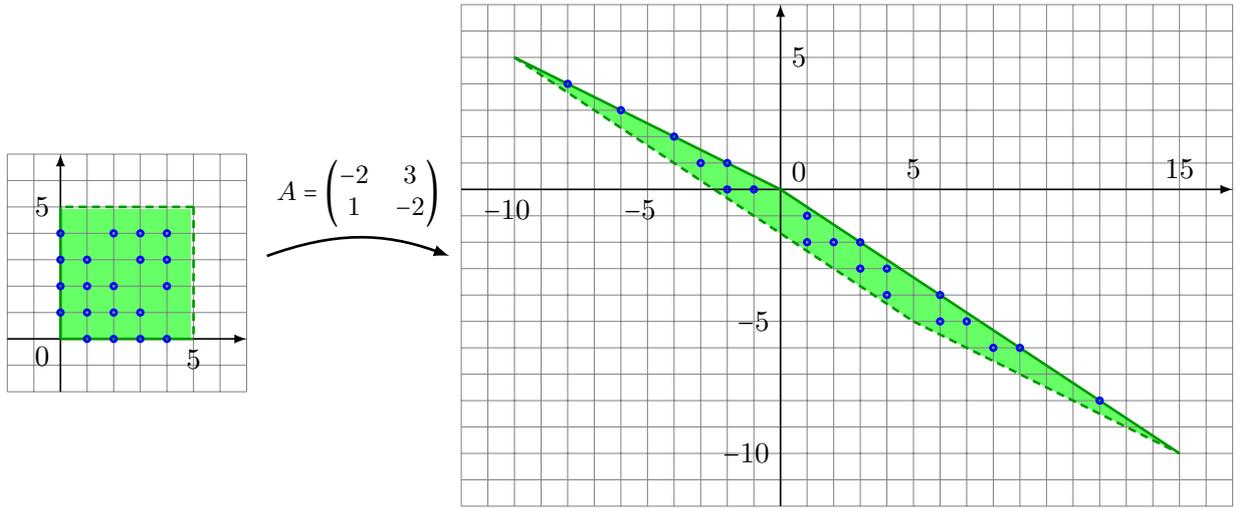
So, let us show that the number of pairs $(\al,\be)\in\Z^{2}$ with conditions
\begin{equation*}
0\le \al < d, \quad 0\le \be < d \qquad\textrm{and}\qquad d\land (a\be - b\al) = 1
\end{equation*}
is equal to $d\phi(d)$ for any natural $d$. The fact that $a$ and $b$ are coprime guarantees  the existence of integers  $a'$ and $b'$ such that
$a' a - b' b = 1$, \ie the determinant of the integer matrix
$$A = \begin{pmatrix} -b & a \\ -a' & b' \end{pmatrix}$$
equals $1$.
An integer point $(x,y)$ will be called \deff{$d$-prime} if its first coordinate $x$ is coprime with $d$. We want to find the number of integer points $(\al,\be)$ in the semi-open square $K=[0,d[\times [0,d[$, for which the image
$$A\cdot \begin{pmatrix} \al \\ \be \end{pmatrix} = \begin{pmatrix} a\be - b\al \\ b'\be - a'\al  \end{pmatrix}$$
is a $d$-prime point: $d \land (a\be - b\al) =1$. In other words, we are interested in the  number of $d$-prime points in the parallelogram $A(K)$, \cf Figure \ref{fig:dprime}. 
Remark at once that the number of $d$-prime integer points $(x,y)$ in the square $K$ is obviously equal to $\phi(d)\cdot d$\,, since the first coordinate $x\in [0,d[$ can take $\phi(d)$ values for each of the $d$ values of the second coordinate $y\in [0,d[$. 

Let us show now that under the action of any integer matrix\footnote{One denotes by $\slzz$ the special linear group over $\Z$, \ie the group of integer matrices $2\times 2$ with determinant \nolinebreak $1$.} $A\in\slzz$ on $K$,  the number of $d$-prime integer points doesn't change. This is true for the matrices $T=\begin{pmatrix} 1&1 \\ 0&1 \end{pmatrix}$ and $U=\begin{pmatrix} 1&0 \\ 1&1 \end{pmatrix}$ as shown in Figure \ref{fig:Tdprime}. Indeed, the parallelograms $T(K)$ and $U(K)$ are obtained from the square $K$ via translation of one of its parts by the vector $(d,0)$ and $(0,d)$ respectively, the coordinates of which are divisible by $d$ (this does not change the number of $d$-prime integer points).
\begin{figure}[htbp]
   \begin{center}
   \begin{tikzpicture}[scale=.35, >=latex,shorten >=0pt,shorten <=0pt,line width=1pt,bend angle=20, 
   		circ/.style={circle,inner sep=2pt,draw=black!40,rounded corners, text centered},
   		tblue/.style={dashed,blue,rounded corners}]

    \begin{scope}[cm={0,1,1,0,(0,0)}]
      \fill[green!60!white] (0,0) -- (4.9,4.9) -- ++(4.9,0) -- (4.9,0) -- (0,0);
	\fill[red!20!white] (0,0.3) -- (0,4.9) -- ++(4.6,0) -- (0,0.3);
      \draw[help lines] (-2,-2) grid (12,7);
	\draw[densely dashed,red!20!white] (0,5) -- ++(5,0);
      \draw[->,line width=.6pt] (-2,0) -- (12,0);
      \draw[->,line width=.6pt] (0,-2) -- (0,7);
      \draw[green!60!black] (0,0) -- (5,0); 
      \draw[densely dashed,green!60!black] (5,5) -- (10,5);
      \draw[green!60!black] (0,0) -- (5,5);
      \draw[densely dashed,green!60!black] (10,5) -- (5,0);
      
	\foreach \y in {1,...,4}{
		\foreach \x in {0,...,4}
      			\draw[blue] (\x,\y) circle (3pt); 
	}

	\draw[->,red,very thick] (2.5,5) -- ++(0,1) -- ++(5,0) -- ++(0,-1);

       \clip (0,0) -- (4.5,4.5) -- (9,4.5) -- (4.5,0) -- (0,0);
	\foreach \y in {1,...,4}{
		\foreach \x in {0,...,4}
      			\draw[blue] (\x+5,\y) circle (3pt); 
	}
    \end{scope}		
      \node[left] at (0,-.65){$0$};
      \draw (0,5) node[left]{$5$};
      \draw (5,0) node[below]{$5$};
      \draw (0,10) node[left]{$10$};
   \end{tikzpicture}
   \begin{tikzpicture}[scale=.4, >=latex,shorten >=0pt,shorten <=0pt,line width=1pt,bend angle=20, 
   		circ/.style={circle,inner sep=2pt,draw=black!40,rounded corners, text centered},
   		tblue/.style={dashed,blue,rounded corners}]
	\node (a) at (0,-4) {};	
	\path[<-] (0,0) edge[bend left,bend angle=30] node[above]
		{\small $U=\begin{pmatrix} 1& 0 \\ 1 & 1 \end{pmatrix}$} (5.5,0);
   \end{tikzpicture}
   \begin{tikzpicture}[scale=.35, >=latex,shorten >=0pt,shorten <=0pt,line width=1pt,bend angle=20, 
   		circ/.style={circle,inner sep=2pt,draw=black!40,rounded corners, text centered},
   		tblue/.style={dashed,blue,rounded corners}]
	
      \fill[green!60!white] (0,0) rectangle (4.9,4.9);
      \draw[help lines] (-2,-2) grid (7,7);
      \draw[->,line width=.6pt] (-2,0) -- (7,0);
      \draw[->,line width=.6pt] (0,-2) -- (0,7);
      \node[left] at (0,-.65){$0$};
      \draw[green!60!black] (0,0) -- (5,0); 
      \draw[densely dashed,green!60!black] (5,5) -- (0,5);
      \draw (0,5) node[left]{$5$};

      \draw[green!60!black] (0,0) -- (0,5);
      \draw[densely dashed,green!60!black] (5,5) -- (5,0);
      \draw (5,0) node[below]{$5$};
      
	\foreach \x in {1,...,4}{
		\foreach \y in {0,...,4}
      			\draw[blue] (\x,\y) circle (3pt); 
	}
   \end{tikzpicture}
   \begin{tikzpicture}[scale=.4, >=latex,shorten >=0pt,shorten <=0pt,line width=1pt,bend angle=20, 
   		circ/.style={circle,inner sep=2pt,draw=black!40,rounded corners, text centered},
   		tblue/.style={dashed,blue,rounded corners}]
	\node (a) at (0,-4) {};	
	\path[->] (0,0) edge[bend left,bend angle=30] node[above]
		{\small $T=\begin{pmatrix} 1& 1 \\ 0 & 1 \end{pmatrix}$} (5.5,0);
   \end{tikzpicture}
   \begin{tikzpicture}[scale=.35, >=latex,shorten >=0pt,shorten <=0pt,line width=1pt,bend angle=20, 
   		circ/.style={circle,inner sep=2pt,draw=black!40,rounded corners, text centered},
   		tblue/.style={dashed,blue,rounded corners}]

      \fill[green!60!white] (0,0) -- (4.9,4.9) -- ++(4.9,0) -- (4.9,0) -- (0,0);
	\fill[red!20!white] (0,0.3) -- (0,4.9) -- ++(4.6,0) -- (0,0.3);
      \draw[help lines] (-2,-2) grid (12,7);
	\draw[densely dashed,red!20!white] (0,5) -- ++(5,0);
      \draw[->,line width=.6pt] (-2,0) -- (12,0);
      \draw[->,line width=.6pt] (0,-2) -- (0,7);
      \node[left] at (0,-.65){$0$};
      \draw[green!60!black] (0,0) -- (5,0); 
      \draw[densely dashed,green!60!black] (5,5) -- (10,5);
      \draw[green!60!black] (0,0) -- (5,5);
      \draw[densely dashed,green!60!black] (10,5) -- (5,0);
      \draw (0,5) node[left]{$5$};
      \draw (5,0) node[below]{$5$};
      \draw (10,0) node[below]{$10$};
      
	\foreach \x in {1,...,4}{
		\foreach \y in {0,...,4}
      			\draw[blue] (\x,\y) circle (3pt); 
	}

	\draw[->,red,very thick] (2.5,5) -- ++(0,1) -- ++(5,0) -- ++(0,-1);

       \clip (0,0) -- (4.5,4.5) -- (9,4.5) -- (4.5,0) -- (0,0);
	\foreach \x in {1,...,4}{
		\foreach \y in {0,...,4}
      			\draw[blue] (\x+5,\y) circle (3pt); 
	}
   \end{tikzpicture}
   \captionwidth=14cm 
   \caption{The number of $d$-prime points in the parallelograms $T(K)$ and $U(K)$ and in the square $K$ is the same.}
   \label{fig:Tdprime}
   \end{center} 
\end{figure}
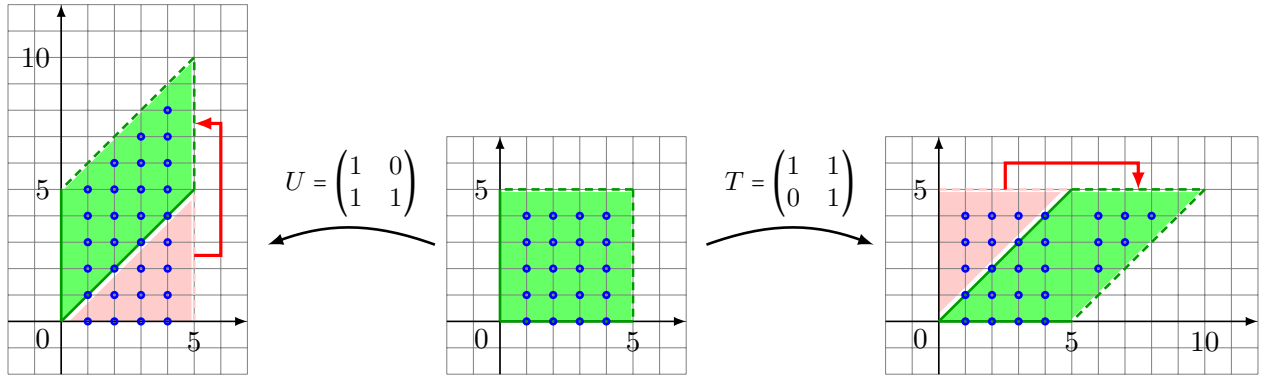

Recall that the group $\slzz$ is generated by the pair of elementами $T$ and $U$, \ie the matrix  $A$ is a word in the alphabet $\set{T,U}$:
$$A = T^{\ep_{1}} U^{\ep_{2}}\ldots T^{\ep_{m-1}}U^{\ep_{m}},\qquad \textrm{where }\; \ep_{i}\in \lsem -1, 1\rsem.$$
Hence, the set of integer points in the semi-open parallelogram $A(K)$ consists of integer points in the square $K$, that are translated by vectors with coordinates divisible by $d$. We conclude that in the parallelogram $A(K)$ and in the square $K$, there is the same number of $d$-prime points, namely $d\cdot\phi(d)$. This completes the proof of the equality $f(d,d) = d\cdot\phi(d)$, from where follows the required equality $f(k,l) = k \ell\,\frac{\phi(k\land \ell)}{k\land \ell}$. 
\end{proof}


\vs
\begin{theo}\label{th:A}  
The number of pairs of permutations from $S_{n}$ with a $3$-cycle commutator and which generate $A_{n}$ or $S_{n}$\,, is equal to
\begin{equation}\label{eq:A:formula}
   \#\A(n) = \frac{3}{8} (n - 2)J_{2}(n) \, n!
\end{equation}
for any natural $n>2$.
\end{theo}
\begin{proof} 
We are going to use Jordan's theorem (Proposition \ref{prop:jordan}), Lemma \ref{lemm:twotypes}, Lemma \ref{lemm:onecylinderprim}, Proposition \ref{prop:A1}, Lemma \ref{lemm:twocylinderprim}, Lemma \ref{lemm:twocylinderprim1}, the M\"{o}bius inversion formula (Proposition \ref{prop:mobiusformula}), Dirichlet convolutions, Proposition \ref{prop:fgh} and Ramanujan's formula (Proposition \ref{prop:ramanujan}).

Consider the pair of permutations $(s,t)\in\snsn$ generating the symmetric or the alternating group of degree $n$ with condition that the commutator $[s,t]$ is a $3$-cycle:
\begin{equation}\label{eq:stzyxA}
  [s,t] = (z\;y\;x), \quad \textrm{\ie}\quad t s t\inv = (x\;y\;z) s.
\end{equation}
Such a permutation $t\in\sn$ and such a triple $(x,y,z)\in\lsem 1,n\rsem^{3}$ will be called \deff{allowed for $s$}. We shall always assume that $z = \max\{x,y,z\}$. Since the group $G=\group{s,t}$ contains a $3$-cycle, then by Jordan's theorem this group coincides with $\sn$ or $\an$ if and only if it is primitive, \ie \emph{the square-tiled surface $O(s,t)$ is primitive}.

According to Lemma \ref{lemm:twotypes}, one of the following two situations occurs:
\begin{itemize}
\item[(a)] All three integers $x,y,z$ lie in same cycle $s$ and the square-tiled surface $O(s,t)$ is one-cylinder with some parameters $(k,a,b,c)$, as shown in Figure \ref{fig:onecylinder}. By Lemma \ref{lemm:onecylinderprim}, the surface $O(s,t)$ is primitive  if and only if $k=1$ and $a\land b\land c = 1$. In particular, then the permutation $s$ is a cycle of maximal length, and so we find the number of required pairs $(s,t)$ via Proposition \ref{prop:A1}:
\begin{equation}\label{eq:FnA}
   F(n) \;:=\; \frac{n}{6}\bigg( J_{2}(n) - 3 J_{1}(n) \bigg)\, n!\,.
\end{equation}

\item[(b)] Two integers $x$ and $y$ lie in the same cycle of the permutationи $s$, and $z$ lies in another cycle. Moreover, the length of the latter cycle is equal to the $s$-distance from $y$ to $x$. In this case, the square-tiled surface $O(s,t)$ is two-cylinder with parameters $(a,b,k,\ell,\al,\be)$, as shown in Figure \ref{fig:twocilinderA}. According to Lemma \ref{lemm:twocylinderprim}, the surface $O(s,t)$ is primitive  if and only if 
$$a\land b = 1 \quad\textrm{ and }\quad k\land \ell\land (a\be - b\al) = 1.$$
Such a permutation $t\in\sn$ and such a triple $(x,y,z)\in\lsem 1,n\rsem^{3}$ will be called \deff{allowed for $s$ of second kind}.
\begin{figure}[htbp]
   \begin{center}
   \begin{tikzpicture}[scale=.5, >=latex,shorten >=1pt,shorten <=1pt,line width=1pt,bend angle=20, 
   		circ/.style={circle,inner sep=2pt,draw=black!40,rounded corners, text centered},
   		tblue/.style={dashed,blue,rounded corners}]
        
      	\draw[step=1] (0,0) grid +(11,3);
      	\draw[step=1] (0,-8) grid +(5,4);
	\draw (0,0)+(.5,.4) node {$y$}; 
	\draw (5,0)+(.5,.5) node {$x$}; 
	\draw (0,-8)+(.5,.5) node {$z$}; 
	
	\path[<->,thin,>=arcs]  (0,-10.7) edge node[below] {$k$} ++ (5,0)
						(0,5.7) edge node[above] {$\ell$} ++ (11,0)
						(-1.7,-8) edge node[left] {$a$} ++ (0,4)
						(-1.7,0) edge node[left] {$b$} ++ (0,3);
	\foreach \x in {0,10}
		\draw[thin] (\x,1) -- ++(0,1);

      	\draw[line width=2pt,blue] (5,0) -- ++(0,3);

	\path[<->,thin,>=arcs,blue]  (0,3.7) edge node[above] {$\be$} ++ (4,0)
						(4,3.7) edge node[above] {$\ell-\be$} ++ (7,0)
						(0,-8.7) edge node[below] {$\ell-\be$} ++ (7,0)
						(7,-8.7) edge node[below] {$\be$} ++ (4,0);

	\path[<->,thin,>=arcs,blue]  (0,-3.5) edge node[above] {$\al$} ++ (2,0)
						(2,-3.5) edge node[above] {$k-\al$} ++ (3,0)
						(0,-.5) edge node[below] {$k-\al$} ++ (3,0)
						(3,-.5) edge node[yshift=-.08cm,below] {$\al$} ++ (2,0);
	\foreach \x in {0,4,11}
		\draw[thin,blue] (\x,3) -- ++(0,1);
	\foreach \x in {0,7,11}
		\draw[thin,blue] (\x,-8) -- ++(0,-1);
	
	\foreach \x in {0,2,5}
		\draw[thin,blue] (\x,-4) -- ++(0,.8);
	\foreach \x in {0,3,5}
		\draw[thin,blue] (\x,0) -- ++(0,-.8);
   \end{tikzpicture}
   \caption{A two-cylinder square-tiled surface with parameters $(a,b,k,\ell,\al,\be)$.}
   \label{fig:twocilinderA}
   \end{center} \vspace{-0.3cm}
\end{figure}
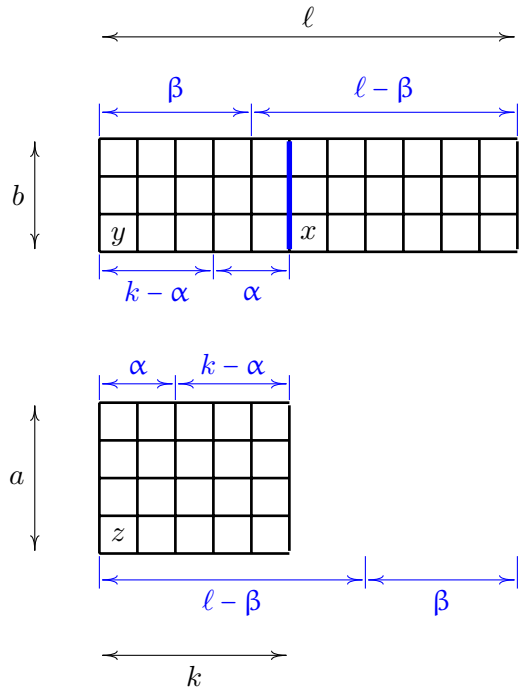
Let us find the number of pairs $(s,t)\in\snsn$\,, where $s$ decomposes into a disjoint product of $a$ cycles of length $k$ and $b$ cycles of length $\ell$, and the permutation $t$ is allowed for $s$ of second kind:
\setlength{\jot}{8pt} 
\begin{eqnarray*}
G(n) &:=& \hspace{-.2cm} \sum_{\substack{a,b,k,\ell\in\N \\ k<\ell \\ a k + b \ell = n\\ a\land b=1}}\hspace{-.3cm} \#\sett{(s,t)\in\snsn}{ \type{s}\,={k^{a}\ell^{b}}, \;\; t \textrm{ is allowed for } s \textrm{ of second kind}}\\
&=& \hspace{-.2cm} \sum_{\substack{a,b,k,\ell\in\N \\ k<\ell \\ a k + b \ell = n\\ a\land b=1}}\hspace{-.3cm} \#\sett{s\in\sn}{ \type{s}=k^{a}\ell^{b} }
\;\times\; \#\sett{t\in\sn}{ t \textrm{ is allowed for } s_{0} \textrm{ of second kind}},
\end{eqnarray*}
where $s_{0}$ is a fixed permutation of cycle type $k^{a}\ell^{b}$. If a triple $(x,y,z)$ is also fixed, then the number of allowed $t$ for $s_{0}$ of second kind with condition \eqref{eq:stzyxA} is equal to
$$(a-1)!\, (b-1)!\, k^{a} \ell^{b}\; \dfrac{\phi(d)}{d}, \quad\textrm{ where } d=k\land \ell\,.$$
Indeed, the square-tiled surface $O(s_{0},t)$ in Figure \ref{fig:twocilinderA} consists of $a$ rows of length $k$ and $b$ rows of length $\ell$, which correspond to the cycles of $s_{0}$. The permutation $t$ indicates how these rows are glued: 
\begin{itemize}
   \item[$\bullet$] the row with the number $z$ must by at the bottom, and $z$ is in its left corner;
   \item[$\bullet$] the remaining $a-1$ rows of length $k$ are glued from above in an arbitrary order, whilst a cyclic permutation of the numbers is permitted in any row -- this makes $(a-1)!\, k^{a-1}$ variants;
   \item[$\bullet$] on top of an obtained rectangle the row with the numbers $y$ and $x$ is glued as shown in Figure, and a twist of length $\al$ is applied;
   \item[$\bullet$] the remaining $b-1$ rows of length $\ell$ are glued from above in an arbitrary order, whilst a cyclic permutation of the numbers is permitted in any row -- this makes $(b-1)!\, \ell^{b-1}$ variants;
   \item[$\bullet$] finally, the top of the construction is glued to the bottom via a twist of length $\be$.
\end{itemize}
The obtained surface must be primitive. Then by Lemma \ref{lemm:twocylinderprim1}, there are $k \ell\; \frac{\phi(d)}{d}$ choices for the twists $\al$ and $\be$, where  $d=k\land \ell$. Hence, the permutation $t$ can be chosen in $(a-1)!\, (b-1)!\, k^{a} \ell^{b}\; \frac{\phi(d)}{d}$ ways. So,
\begin{eqnarray*}
G(n) &=& \hspace{-.2cm} \sum_{\substack{a,b,k,\ell\in\N \\ k<\ell \\ a k + b \ell = n\\ a\land b=1}}\hspace{-.3cm} \#\sett{s\in\sn}{ \type{s}= k^{a}\ell^{b} }
\;\times\; (a-1)!\, (b-1)!\, k^{a} \ell^{b}\; \dfrac{\phi(k\land l)}{k\land l} \\
& &\qquad\qquad\times\; \#\sett{ (x,y,z)\in \lsem 1,n\rsem^{3} }{\textrm{the triple } (x,y,z) \textrm{ is allowed for } s_{0} \textrm{ of second kind}} \\
\end{eqnarray*}
It is easy to check that the group $\sn$ has 
$$\frac{n!}{a!\,b!\,k^{a} \ell^{b}}$$
permutations of cycle type $k^{a}\ell^{b}$. Finally, the number of allowed triples $(x,y,z)$ of second kind for a fixed permutation $s_{0}$ of cycle type $k^{a}\ell^{b}$ is equal to $$a b k \ell\,.$$ Indeed, the integer $z$ belongs to one of the $a$ cycles of length $k$ (there are $a\cdot k$ variants) and the integer $y$ belongs to one of the $b$ cycles of length $\ell$ (there are $b\cdot \ell$ variants). Therefore, we obtain the relation
\begin{eqnarray*}
G(n) &=& \hspace{-.2cm} \sum_{\substack{a,b,k,\ell\in\N \\ k<\ell \\ a k + b \ell = n\\ a\land b=1}} 
\frac{n!}{a!\,b!\,k^{a} \ell^{b}} 
\;\times\; (a-1)!\, (b-1)!\, k^{a} \ell^{b}\; \dfrac{\phi(k\land l)}{k\land l} 
\;\times\; a b k \ell \\
&=& n! \hspace{-.2cm} \sum_{\substack{a,b,k,\ell\in\N \\ k<\ell \\ a k + b \ell = n\\ a\land b=1}}
\dfrac{\phi(k\land l)}{k\land l} \;\times\; k \ell\,.
\end{eqnarray*}
Denote\, $g(n) = \drob{G(n)}{n!}$\, and consider the following function
$$\widetilde{g}(n) := \sum_{\substack{a,b,k,\ell\in\N \\ k<\ell \\ a k + b \ell = n}}
\dfrac{\phi(k\land l)}{k\land l} \;\times\; k \ell\,,$$
where integers $a, b$ are not necessarily coprime. Denote $c = a\land b$ and notice that the equality $a k + b \ell = n$ is equivalent to the equality $a' k + b' \ell = \drob{n}{c}$\,, where the integers $a'=\drob{a}{c}$ and $b'=\drob{b}{c}$ are coprime. Thus,
$$\widetilde{g}(n) = \sum_{c \mid n} g(n/c) = \sum_{c \mid n} g(c)\,,$$
and by the M\"{o}bius inversion formula (Proposition \ref{prop:mobiusformula}), we get
\begin{equation}\label{eq:convg}
   g(n) = \sum_{c \mid n}  \mu(c) \widetilde{g}(n/c),\quad\textrm{ \ie}\quad g = \mu\ast \widetilde{g}\,.
\end{equation}

Further, the greatest common divisor  $d=k\land l$ also divides the integer $n= a k + b \ell$, from where
$$\widetilde{g}(n) \;= \sum_{\substack{a,b,k,\ell\in\N \\ k<\ell \\ a k + b \ell = n}}
\dfrac{\phi(k\land l)}{k\land l} \; k \ell
\;= \sum_{d\mid n} \dfrac{\phi(d)}{d} \Bigg(\sum_{\substack{a,b,k,\ell\in\N \\ k<\ell \\ a k + b \ell = n \\ k\land l = d}} k \ell \Bigg)\,.$$
In the last sum, let us replace the pair $(k, \ell)$ by the pair $(d\cdot k' , d\cdot \ell')$, where the integers $k'$ and $\ell'$ are coprime 
$$
\widetilde{g}(n) \; 
= \sum_{d\mid n} \dfrac{\phi(d)}{d} \Bigg(\sum_{\substack{a,b,k,\ell\in\N \\ k<\ell \\ a k + b \ell = n \\ k\land l = d}} k \ell \Bigg)\;
= \sum_{d\mid n} \dfrac{\phi(d)}{d} \Bigg( \sum_{\substack{a,b,k',\ell'\in\N \\ k'<\ell' \\ a k' + b \ell' = n \\ k'\land l' = 1}} d^{2} k' \ell' \Bigg)\;
= \sum_{d\mid n} d \phi(d) \Bigg(\sum_{\substack{a,b,k',\ell'\in\N \\ k'<\ell' \\ a k' + b \ell' = \drob{n}{d} \\ k'\land l' = 1}} k' \ell' \Bigg)\,,
$$
or in terms of the Dirichlet convolution:
\begin{equation}\label{eq:convgtilde}
   \widetilde{g} = (\id\cdot \phi) \ast h,\quad\textrm{ where}\quad h(n) \;= \sum_{\substack{a,b,k',\ell'\in\N \\ k'<\ell' \\ a k' + b \ell' = n \\ k'\land l' = 1}} k' \ell'\,.
\end{equation}
Consider also the function
$$
\widetilde{h}(n) \;= \sum_{\substack{a,b,k,\ell\in\N \\ k<\ell \\ a k + b \ell = n }} k \ell \,.
$$
Since the equality $a k + b \ell = n$ is equivalent to the equality $a k' + b \ell' = \drob{n}{d}$ with $k\ell = d^{2}k'\ell'$, then
$$\widetilde{h}(n) = \sum_{d \mid n} d^{2} h(n/d),\quad\textrm{ \ie}\quad \widetilde{h} = \id_{2}\ast h\,.$$
Taking into account Proposition \ref{prop:fgh}, we have
$$
\widetilde{h} = \id_{2}\cdot \big(\one \ast \frac{h}{\id_{2}}\big), \quad\textrm{ \ie}\quad \frac{\widetilde{h}}{\id_{2}} = \one \ast \frac{h}{\id_{2}}\,,
$$
and so by the M\"{o}bius inversion formula, we get
\begin{equation}\label{eq:convh}
   \frac{h}{\id_{2}} = \mu \ast \frac{\widetilde{h}}{\id_{2}},\quad\textrm{ from where}\quad 
   h = (\id_{2}\cdot\mu) \ast \widetilde{h}\,.
\end{equation}

We conclude that the relations \eqref{eq:convg}, \eqref{eq:convgtilde} and \eqref{eq:convh} give
\begin{equation*}
   g \;=\; \mu\ast \widetilde{g} \;=\; \mu \ast (\id\cdot \phi)\ast h \;=\; \mu\ast (\id\cdot\phi) \ast (\id_{2}\cdot\mu) \ast \widetilde{h}\,.
\end{equation*}
By Proposition \ref{prop:fgh} and the equality \eqref{eq:idmuphii}, one has 
$$(\id\cdot\phi) \ast (\id_{2}\cdot\mu)  = \id\cdot \big(\phi\ast (\id\cdot\mu)\big)\qquad \textrm{and}\qquad  \phi\ast (\id\cdot\mu) = \mu\,,$$
from where
\begin{equation}\label{eq:convgg}
   g \;=\; \mu\ast (\id\cdot\mu)\ast \widetilde{h}\,.
\end{equation}
We already know from the proof of Theorem \ref{th:B} that
\begin{equation}\label{eq:htilde}
   \widetilde{h} = \frac{5}{24}\si_{3} + \frac{1}{2}\si_{2} + \frac{1}{24}\si_{1} - \frac{3}{4}\id\cdot \si_{1}\,,
\end{equation}
due to the relation \eqref{eq:sumklB}. The equalities $\mu\ast \si_{i} = \id_{i}$ and $\mu\ast \id_{i} = J_{i}$ (see \eqref{eq:musik} and \eqref{eq:idkmu} in Appendix) imply that
\begin{equation*}
   (\id\cdot\mu)\ast \mu\ast \si_{i} \;=\;  (\id\cdot\mu)\ast \id_{i} \;=\; \id\cdot (\mu\ast \id_{i-1}) \;=\; 
   \begin{cases}
   \id\cdot (\mu\ast\one) = \id\cdot\ep = \ep, & \textrm{if } i=1;\\
   \id\cdot J_{i-1}, & \textrm{if } i > 1,
   \end{cases}
\end{equation*}
where, by definition, $\ep(n) = 0$ for $n>1$. Further,
\begin{equation*}
   \mu\ast (\id\cdot\mu)\ast  (\id\cdot \si_{1}) \;=\; \mu\ast \id\cdot(\mu\ast\si_{1}) \;=\; \mu\ast \id\cdot\id \;=\; \mu\ast\id_{2} = J_{2}\,.
\end{equation*}
Put the formula \eqref{eq:htilde} for $\widetilde{h}$ into the relation \eqref{eq:convgg}:
\begin{equation*}
   g = \frac{5}{24}\id\cdot J_{2} + \frac{1}{2}\id\cdot J_{1} + \frac{1}{24}\ep - \frac{3}{4} J_{2}\,,
\end{equation*}
\ie 
\begin{equation}\label{eq:GnA}
   G(n) \;=\; g(n)\, n! \;=\; 
   \Bigg(\frac{5}{24}n J_{2}(n) + \frac{1}{2}n J_{1}(n) - \frac{3}{4} J_{2}(n)\Bigg)\, n!\,,
\end{equation}
 for any natural $n>2$.
\end{itemize}
By summing the formulas \eqref{eq:FnA} and \eqref{eq:GnA}, we obtain
\begin{equation*}
   \#\A(n) \;=\; F(n) + G(n) 
   =\; \Bigg( \frac{n}{6} J_{2}(n) - \frac{n}{2} J_{1}(n)  + \frac{5n}{24} J_{2}(n) + \frac{n}{2} J_{1}(n) - \frac{3}{4} J_{2}(n)\Bigg)\, n! \;
   =\; \frac{3}{8} (n - 2)J_{2}(n) \, n!\,, 
\end{equation*}
as required.
\end{proof}

\vss
\begin{corr}\label{corr:A:bounds} 
For any $\eps > 0$, there exists a positive integer $N=N(\eps)$ such that 
$$n^{3 - \eps}\cdot n! \;<\; \#\A(n) \;<\; \frac{3}{8} n^{3}\cdot n!$$
for all natural $n \ge N$.
\end{corr}
\begin{proof}
From Theorem \ref{th:A} and the definition of Jordan's totient function follows that
\begin{equation*}
   \frac{8}{3}\cdot \frac{\#\A(n)}{n!} \;=\; (n - 2)J_{2}(n) \;<\; n\cdot n^{2} \;=\; n^{3}
\end{equation*}
for all natural $n$.
Let us now prove the lower bound for the cardinality of the set $\A(n)$. For any $\delta>0$, there exists a number $K=K(\delta)$ such that 
\begin{equation}\label{eq:corr:A:bounds:J2}
	n^{2-\delta} < J_{2}(n)\qquad \textrm{for } n\ge K.
\end{equation}
Indeed, the function 
$$f(n) = \frac{n^{2-\delta}}{J_{2}(n)}$$
is multiplicative. For a power of a prime $p^{\al}$, its value
$$f(p^{\al}) \;=\; \frac{(p^{\al})^{2-\delta}}{ (p^{\al})^{2}  \left(1 - \drob{1}{p^{2}}\right)  }
	\;=\; \frac{1}{p^{\al\delta}}\cdot \frac{p^{2}}{p^{2}-1}$$
tends to zero as $p^{\al}\to +\infty$. According to Proposition \ref{prop:appA:epsilon}, we have $f(n)\to 0$ as $n\to +\infty$. In particular, there exists a number $K$ such that
$$f(n) < 1 \quad \textrm{for all } n\ge K.$$
We conclude that if $n\ge K$ and $n\ge 4$, then
\begin{equation*}
    \frac{\#\A(n)}{n!} \;=\; \frac{3}{8}(n - 2)J_{2}(n) 
   \;>\; \frac{3}{8}\cdot \frac{n}{2} \cdot n^{2-\delta} \;=\; \frac{3}{16} n^{3-\delta}.
\end{equation*}
For an arbitrary $\eps > 0$, take $\delta = \eps/2$. Then 
\setlength{\jot}{8pt} 
\begin{eqnarray*}
\frac{\#\A(n)}{n!} &>& \frac{3}{16}n^{3- \delta} \;=\; \frac{3}{16}n^{3 - \eps + \eps/2}  
			\;=\; \frac{3 n^{\eps/2}}{16}\cdot n^{3-\eps} \\
	&>& n^{3-\eps}  \qquad \textrm{for }\; n\ge N,
\end{eqnarray*}
where $N=N(\eps)$ is a positive integer such that $N\ge K$, $N \ge 4$ \,and\, $N^{\eps/2} > \drob{16}{3}$. This completes the proof.
\end{proof}

\begin{corr}\label{corr:A:proba} 
The probability that a pair of permutations from $\sn$ with a $3$-cycle commutator generates $\an$ or $\sn$\,, tends to zero as $n\to+\infty$.
\end{corr}
\begin{proof}
According to Corollary \ref{corr:A:bounds}, the upper bound\; $\#\A(n) \;<\; \frac{3}{8} n^{3}\cdot n!$\; holds for any $n$. On the other hand, due to Corollary \ref{corr:B:bounds} with $\eps=2$, there exists an $N$ such that 
$$\#\B(n) > \psi_{0}(n)\cdot n! = n\, P(n)\cdot n!$$ 
for all $n\ge N$. Therefore,
\begin{equation*}
	0 \;<\; \frac{\#\A(n)}{\#\B(n)} \;<\; \frac{\frac{3}{8} n^{2}}{P(n)} \;\;\underset{n\to+\infty}{\longto}\;\; 0,
	\qquad\textrm{from where }\quad \frac{\#\A(n)}{\#\B(n)} \;\;\underset{n\to+\infty}{\longto}\;\; 0,
\end{equation*}
which proves the required statement.
\end{proof}

\newpage
\setcounter{equation}{0}
\renewcommand{\theequation}{\thesection.\arabic{equation}}
\setcounter{table}{0}
\renewcommand{\thetable}{\thesection.\arabic{table}}
\newshadetheorem{theoA}{Theorem}
\newshadetheorem{propA}{Proposition}
\newshadetheorem{lemmA}{Lemma}
\newshadetheorem{corrA}{Corollary}

\renewcommand{\theshadetheoA}{\thesection.\arabic{theoA}}
\renewcommand{\theshadepropA}{\thesection.\arabic{propA}}
\renewcommand{\theshadelemmA}{\thesection.\arabic{lemmA}}
\renewcommand{\theshadecorrA}{\thesection.\arabic{corrA}}
\titleformat{\section}{\Large\bfseries\scshape}{Appendix\; \thesection}{0.5cm}{}
\renewcommand{\contentsname}{Appendix}
\appendix

\section{Background in number theory and permutation groups}
\subsection{Arithmetic functions}\label{sec:arithmetic}
\noindent The majority of facts about arithmetic functions, which are given in this section, can be found in the classical textbook by Apostol \cite{apostol}.
\begin{defi} 
Any function $f: \N\to \CC$ defined on the set of positive integers and taking its values in the set of complex numbers is called \deff{arithmetic}.  
\end{defi}
\noindent Denote by $\F$ the set of all arithmetic functions, and by $\F^{\ast}$ the subset functions different from zero at the point $1$:
\begin{equation}\label{eq:F1}
   \F^{\ast} = \sett{f\in\F}{f(1)\neq 0}.
\end{equation}
For any two arithmetic functions $f$ and $g$ one defines their \deff{product} $f\cdot g$: 
\begin{equation}\label{eq:product} 
   (f\cdot g)(n) = f(n) g(n),
\end{equation}
their \deff{discrete convolution} $f\dc g$:
\begin{equation}\label{eq:discrete}
   (f\dc g)(1) = 0\quad \textrm{and}\quad (f\dc g)(n) = \sum_{k=1}^{n-1} f(k) g(n-k)\quad \textrm{for } n>1,
\end{equation}
and also the \deff{Dirichlet convolution} $f\ast g$:
\begin{equation}\label{eq:dirichlet}
   (f\ast g)(n) = \sum_{d\mid n} f(d) g(n/d).
\end{equation}
The latter operation provides the structure of an abelian group on the set $\F^{\ast}$, since it possesses the following properties:
\setlength{\leftmargini}{5.2cm} 
\setlength{\labelsep}{0.5cm} 
\begin{itemize}
\item[(\emph{identity element})] $f\ast \ep = \ep\ast f = f$\quad for any function $f\in\F$, where 
\begin{equation}\label{eq:delta}
\ep(n)=
\begin{cases} 
1, & \textrm{if } n=1\\
0, & \textrm{if } n>1 
\end{cases}    
\end{equation}

\item[(\emph{inverse element})] for any $f\in\F^{\ast}$ there exists a unique function $\widetilde{f}\in\F^{\ast}$ such that \; $f\ast \widetilde{f} = \widetilde{f}\ast f=\ep$, \; such a function satisfies the recurrent relation
\begin{equation}\label{eq:dirichletinverse}
   \widetilde{f}(1)=\frac{1}{f(1)}\;\; \textrm{ and }\;\; \widetilde{f}(n) = \frac{-1}{f(1)} \sum_{d\mid n,\, d<n} \widetilde{f}(d) f(n/d) \;\; \textrm{ for } n>1;
\end{equation}

\item[(\emph{associativity})] $f\ast (g\ast h) = (f\ast g)\ast h$\quad for any $f, g, h\in\F$;

\item[(\emph{commutativity})] $f\ast g = g\ast f$\quad for any $f, g\in\F$.
\end{itemize}

\vs
\begin{defi} 
An arithmetic function $f$ is \deff{multiplicative}, if $f(1)=1$ and the equality
\begin{equation}\label{eq:multiplicative}
   f(m n) = f(m) f(n)\quad \textrm{holds for any coprime } m, n\in\N. 
\end{equation}
\end{defi}
\noindent By $\Fm$ we denote the set of all multiplicative arithmetic functions. It is a subgroup of the group $\F^{\ast}$, since
\setlength{\leftmargini}{5.2cm} 
\begin{itemize}
  \item[$\begin{pmatrix}\emph{multiplicativity}\\ \emph{of the convolution}\end{pmatrix}$] the Dirichlet convolution of two multiplicative functions is multiplicative; 
  \item[$\begin{pmatrix}\emph{multiplicativity}\\ \emph{of the inverse element}\end{pmatrix}$] the function $\widetilde{f}$ is multiplicative for any function $f\in\Fm$.
\end{itemize}

\vs It is easy to show that  the following arithmetic functions are multiplicative:
\begin{itemize}
   \item[$\one(n)$] \deff{constant function}, $\one(n)=1$

   \item[$\id(n)$] \deff{identical function}, $\id(n)=n$

   \item[$\id_{k}(n)$] \deff{power function of order} $k$, 
   \begin{equation}\label{eq:idk}
   \id_{k}(n)=n^{k},
   \end{equation}
   in particular, $\id_{0}=\one$ and $\id_{1}=\id$

   \item[$\mu(n)$] \deff{M\"{o}bius function} defined as 
   \begin{equation}\label{eq:mu}
   \mu(n)=
   \begin{cases}
   1, & \textrm{if } n=1\\
   (-1)^{r}, & \textrm{if } \al_{1}=\ldots=\al_{r}=1\\
   0, & \textrm{в ином случае} 
   \end{cases}
   \end{equation}
   where $n=p_{1}^{\al_{1}}\cdots\, p_{r}^{\al_{r}}$ is the prime factorization of $n$

   \item[$\tau(n)$] \deff{number of divisors } of $n$, 
   \begin{equation}\label{eq:tau}
   \tau(n)=\prod_{i=1}^{r} (\al_{i}+1),
   \end{equation}
   where $n=p_{1}^{\al_{1}}\cdots\, p_{r}^{\al_{r}}$ is the prime factorization of $n$

   \item[$\si(n)$] \deff{sum of divisors } of $n$, 
   \begin{equation}\label{eq:si}
   \si(n)=\prod_{i=1}^{r} \frac{p_{i}^{\al_{i}+1}-1}{p_{i}-1},
   \end{equation}
   where $n=p_{1}^{\al_{1}}\cdots\, p_{r}^{\al_{r}}$ is the prime factorization of $n$

   \item[$\si_{k}(n)$] \deff{sigma-function of order $k$} is the sum of $k^\textrm{th}$ powers of the divisors of $n$, 
   \begin{equation}\label{eq:sik}
   \si_{k}(n)= \sum_{d\mid n} d^{k} =\prod_{i=1}^{r} \frac{p_{i}^{k(\al_{i}+1)}-1}{p_{i}^{k}-1},
   \end{equation}
   in particular, $\si_{0}=\tau$ and $\si_{1}=\si$

   \item[$\phi(n)$] \deff{Euler's totient function} is the number of positive integers from $1$ to $n$ which are coprime with $n$,
   \begin{equation}\label{eq:phi}
   \phi(n) = n\mhs \prod_{\substack{p | n\\ p\;\textrm{prime}}} \mhs \left(1-\frac{1}{p}\right)
   \end{equation}

   \item[$J_{k}(n)$] \deff{Jordan's totient function of order $k>0$} is the number of tuples $(a_{1},\ldots,a_{k})$ of $k$ positive integers from $1$ to $n$, the greatest common divisor of which is coprime with $n$,
   \begin{equation}\label{eq:Jk}
   J_{k}(n) = n^{k}\mhs \prod_{\substack{p | n\\ p\;\textrm{prime}}} \mhs\left(1-\frac{1}{p^{k}}\right),
   \end{equation}
   in particular, $J_{1}=\phi$.
\end{itemize}
For a proof of the formulas \eqref{eq:tau} -- \eqref{eq:Jk}, it is sufficient to check them for the case that $n$ is a power of a prime, and to use the multiplicativity property. 

\begin{propA}[\textbf{M\"{o}bius inversion formula}]\label{prop:mobiusformula} 
  Let $f$ and $g$ be arbitrary arithmetic functions. The equality 
  $$
  f(n) = \sum_{d\mid n} g(d)
  $$
  holds for any $n\in\N$ if and only if one has
  $$
  g(n) = \sum_{d\mid n} \mu(d) f(n/d)
  $$
  for any $n\in\N$. In other words, the equalities $f = g\ast \one$ and $g = f\ast \mu$ are equivalent.
\end{propA}

\begin{propA} 
  Let $f$ be a multiplicative arithmetic function. Then the equality
  \begin{equation}\label{eq:sumproduct}
   \sum_{d\mid n} \mu(d) f(d) \;= \prod_{\substack{p | n\\ p\;\textrm{\emph{prime}}}} \mhs(1 - f(p))
  \end{equation}
  holds for any $n\in\N$.
\end{propA}
\begin{proof} 
Consider the function $g(n)=\sum_{d\mid n} \mu(d) f(d)=(\mu\cdot f)\ast \one$. It is multiplicative as the Dirichlet convolution of two multiplicative functions $\mu\cdot f$ and $\one$. Further,
$$
g(p^{\al}) = \mu(1) f(1) + \mu(p) f(p) + \dots + \mu(p^{\al}) f(p^{\al}) = 1\cdot 1 + (- 1)\cdot f(p) = 1 - f(p)
$$
for each power of a prime $p$. Therefore, $$g(n)=g(p_{1}^{\al_{1}})\cdots g(p_{r}^{\al_{r}}) = (1 - f(p_{1}))\cdots (1 - f(p_{r}))$$ for any natural $n$ with prime factorization $p_{1}^{\al_{1}}\cdots p_{r}^{\al_{r}}$.
\end{proof}

From this proposition for $f(n)=1/n^{k}$ we obtain the equality
\begin{equation}\label{eq:mud}
   \sum_{d\mid n} \frac{\mu(d)}{d^{k}} \;= \prod_{\substack{p | n\\ p\;\textrm{prime}}} \mhs\left(1 - \frac{1}{p^{k}}\right).
\end{equation}

The following relations can be proved directly by definition of the Dirichlet convolution or by using the multiplicativity property:
\begin{eqnarray} 
\one\ast \one &=& \tau \\
\one\ast \id_{k} &=& \si_{k}\\ 
\one\ast \mu &=& \ep\\
\one\ast\phi &=& \id\\
\one\ast J_{k} &=& \id_{k}\\ 
\id_{k}\ast \id_{\ell} &=& \si_{k-\ell}\cdot \id_{\ell}\\
\id_{k}\ast \mu &=& J_{k} \label{eq:idkmu}\\
\mu\ast \si_{k} &=& \id_{k} \label{eq:musik}\\
\tau\ast \phi &=& \si\\
(\id\cdot \mu)\ast \si &=& \one \label{eq:idmusi}\\ 
(\id\cdot \mu)\ast \phi &=& \mu. \label{eq:idmuphii}
\end{eqnarray}

\noindent For instance, to show the equality \eqref{eq:idmusi} one considers the function $f = (\id\cdot \mu)\ast \si$, which is multiplicative as well as the functions $\id \cdot \mu$ and $\si$. If $p$ is a prime number, then
\begin{eqnarray*} 
f(p^{\al}) &=& \sum_{d\mid p^{\al}} d \mu(d) \si(p^{\al}/d) = \sum_{k=0}^{\al} p^{k} \mu(p^{k}) \si(p^{\al-k})  
= 1\cdot\mu(1)\si(p^{\al}) + p\cdot\mu(p)\si(p^{\al-1}) \\ 
&=& (1+p+\dots+p^{\al}) - p(1+p+\dots+p^{\al-1}) = 1,
\end{eqnarray*}
from where $f(n) = f(p_{1}^{\al_{1}})\cdots f(p_{r}^{\al_{r}}) = 1$ for any $n=p_{1}^{\al_{1}}\cdots p_{r}^{\al_{r}}$.

\vss
The following statement is given in the textbook of Hardy and Wright \cite[теорема 316]{hardy-wright79}.
\begin{propA}\label{prop:appA:epsilon} 
If an arithmetic function $f(n)$ is multiplicative and $f(p^{\al})\to 0$ as $p^{\al}\to +\infty$, where $p^{\al}$ is a power of a prime, then $f(n)\to 0$ as $n\to +\infty$.
\end{propA}

\vss
An arithmetic function $f$ is \deff{completely multiplicative}, if $f(m n) = f(m) f(n)$ for any positive integers $m$ and $n$ (not just coprime). 

\begin{propA}\label{prop:fgh} 
  If a function $f$ is completely multiplicative, then
  $$
  f\cdot (g\ast h) = (f\cdot g)\ast (f\cdot h)
  $$
  for any functions $g, h\in\F$.
\end{propA}
\begin{proof}
   Since $f$ is completely multiplicative, then $f(n) = f(d)f(n/d)$ for any positive integer $n$ and any of its divisor $d$. Therefore,
   $$
   f\cdot (g\ast h)(n) = f(n) \sum_{d\mid n} g(d) h(n/d) = \sum_{d\mid n} f(d)g(d)\, f(n/d) h(n/d) = (f\cdot g)\ast (f\cdot h)(n),
   $$
   as required.
\end{proof}

The function $\id$ is multiplicative, so the equality \eqref{eq:idmusi} can be established, by Proposition \ref{prop:fgh} and the M\"{o}bius inversion formula:
$$
(\id\cdot \mu)\ast \si = \one \quad\iff\quad 
\id\cdot \left(\mu\ast \frac{\si}{\id}\right) =\one \quad\iff\quad
\frac{\si}{\id}\ast \mu= \frac{\one}{\id} \quad\iff\quad
\frac{\one}{\id} \ast\one  = \frac{\si}{\id} \quad\iff\quad
\one\ast \id  = \si.
$$
The last equality is obviously true.

\vs
In 1916, Srinivasa Ramanujan expressed discrete convolutions of sigma-functions of odd order via linear combinations of functions $\si_{k}$ and $\id\cdot \si_{k}$ with rational coefficients. For instance, he proved the proposition bellow (\cf collected papers \cite{ramanujan}).
\begin{propA}[\textbf{Ramanujan's formulas}, 1916]\label{prop:ramanujan} 
  The following equalities
  $$(\si \dc \si)(n) = \frac{5}{12}\si_{3}(n) + \frac{1}{12}\si(n) - \frac{1}{2}n \si(n),$$
  $$(\si \dc \si_{3})(n) = \frac{7}{80}\si_{5}(n) + \frac{1}{24}\si_{3}(n) - \frac{1}{240}\si(n) - \frac{1}{8}n \si_{3}(n)$$
  hold for any natural $n$.
\end{propA}

\subsection{Partitions of a positive integer}\label{sec:partitions}
\noindent A \deff{partition} of a number $n\in \N$ is its presentation as a sum
\begin{equation}\label{eq:na}
n = a_{1} + \dots + a_{r}
\end{equation}
of positive integers $a_{1},\ldots,a_{r}\in\N$. Two sums that differ only in the order of their summands are considered to be the same partition.
\begin{defi} 
The function $P(n)$ equal to the number of distinct partitions of $n$ is called the \deff{partition function}.
\end{defi}

\noindent The values of $P(n)$ for the first 18 positive integers are given in Table \ref{table:partinion}.

\begin{table}[htdp]
\caption{$P(n)$ for $n\in\lsem 1, 18\rsem$.}
\mvs
\begin{center}
\begin{tabular}{|c|c|c|c|c|c|c|c|c|c|c|c|c|c|c|c|c|c|c|c|c|} 
\hline
$n$ & 1&2&3&4&5&6&7&8&9&10&11&12&13&14&15&16&17&18\\ \hline\hline
$P(n)$ & 1&2&3&5&7&11&15&22&30&42&56&77&101&135&176&231&297&385\\
\hline
\end{tabular}
\end{center}
\label{table:partinion}
\end{table}%

The generating function of the sequence $\{P(n)\}_{n=0}^{\infty}$\,, where $P(0)=1$, is the following infinite product:
\begin{equation}\label{eq:genpartition}
   F(x) = \prod_{k=1}^{\infty} \frac{1}{1 - x^{k}}.
\end{equation}
Indeed, the equality \eqref{eq:na} can be re-written as $n=\al_{1} k_{1} + \dots+\al_{s} k_{s}$\,, meaning that among $a_{1}, \ldots, a_{r}$ there are exactly $\al_{1}$ numbers $k_{1}$, $\al_{2}$ numbers $k_{2}$, etc., exactly $\al_{s}$ numbers $k_{s}$ with condition $k_{1} < \ldots < k_{s}$. Therefore,
\begin{eqnarray*}
F(x) &=& \prod_{k=1}^{\infty} \frac{1}{1 - x^{k}} = \prod_{k=1}^{\infty} (1 + x^{k} + x^{2k} + \dots + x^{\al k}+ \dots) 
= 1 +\mhs\mhs \sum_{ \substack{s\in \N \\ \al_{1}, \ldots, \al_{s}\in\N \\ k_{1} < \ldots < k_{s}\in\N} }\mhs x^{\al_{1} k_{1} + \dots+\al_{s} k_{s}}\\
&=& \sum_{n=0}^{\infty} P(n) x^{n}.
\end{eqnarray*}

In their famous paper \cite[1918]{hardy-ramanudjan} the mathematicians Hardy and   Ramanujan  obtained the asymptotic estimate of the partition function:
\begin{equation}\label{eq:pn}
   P(n)\sim \frac{1}{4n\sqrt{3}} \exp\left(\pi\sqrt{\frac{2n}{3}}\right) \quad\textrm{as } n\to\infty.
\end{equation}

\begin{propA}\label{prop:sidcP} 
The discrete convolution of the functions $\si$ and $P$ satisfies the equality
$$
(\si\dc P)(n) = n P(n) - \si(n)
$$
for any positive integer $n$.
\end{propA}
\begin{proof} 
Take the natural logarithm of both sides of the equality \eqref{eq:genpartition}
$$
   \log{(F(x))} = \sum_{k=1}^{\infty} \log{\left(\frac{1}{1 - x^{k}}\right)}
$$
and differentiate the new equality:
\begin{equation}\label{eq:xfx}
  \frac{F'(x)}{F(x)} = \sum_{k=1}^{\infty} \frac{k x^{k-1}}{(1-x^{k})^{2}}\,,\quad \textrm{from where}\quad 
  x F'(x) = F(x) \sum_{k=1}^{\infty} \frac{k x^{k}}{1-x^{k}}.   
\end{equation}
An infinite formal series of the form\; $\sum\limits_{k=1}^{\infty} a_{k}\, \dfrac{x^{k}}{1-x^{k}}$\; is called the \deff{Lambert series} with coefficients $a_{1}, a_{2}, \dots$ (\cf the classical textbook by Knopp \cite[\S 58]{knopp}). As well known, the following relations hold for sigma-functions:
\begin{equation}\label{eq:lambert}
   \sum_{k=1}^{\infty} k^{m}\,\frac{ x^{k}}{1-x^{k}} = \sum_{n=1}^{\infty} \si_{m}(n) x^{n}
\end{equation}
for any nonnegative integer $m$. In particular, when $m=1$ we get
$$
   \sum_{k=1}^{\infty} \frac{k x^{k}}{1-x^{k}} = \sum_{n=1}^{\infty} \si(n) x^{n}.
$$
Because\; $F(x) = \sum_{n=0}^{\infty} P(n) x^{n}$, then\; $x F'(x) = \sum_{n=1}^{\infty} n P(n) x^{n}$ and we can write the equality \eqref{eq:xfx} as
\begin{align*}
\sum_{n=1}^{\infty} n P(n) x^{n} &= \left(\sum_{n=0}^{\infty} P(n) x^{n} \right) \left(\sum_{n=1}^{\infty} \si(n) x^{n} \right) = \sum_{n=1}^{\infty} \left(\sum_{k=0}^{n-1} P(k) \si(n-k) \right) x^{n} \\
&= \sum_{n=1}^{\infty} \big(\si(n) + (\si\dc P)(n)\big) x^{n}.
\end{align*}
Comparing the coefficients of $x^{n}$ leads to the relation $(\si\dc P)(n) = n P(n) - \si(n)$.
\end{proof}

\subsection{Permutation groups}\label{sec:permutations}
\noindent Let $M$ and $L$ be some sets (finite or infinite). The elements of a set will  be also called \deff{points}.

\setlength{\leftmargini}{2.5em} 
\setlength{\labelsep}{0.5em} 
\begin{description}
\item[Injection] from the set $M$ to the set $L$ is a map $f : M\into L$ with the property: $f(x) \neq f(y)$ for any $x\neq y$ from $M$. In other words, an injection sends distinct elements of $M$ to \emph{distinct} elements of \nolinebreak $L$.

\item[Surjection] from the set $M$ to the set $L$ is a map $f : M\onto L$ with the propertyм: for any $z\in L$ there exists at least one $x \in M$ such that $z=f(x)$. In other words, under surjection the image of the set $M$ is the whole set $L$.

\item[Bijection] from the set $M$ to the set $L$ is a map $f : M\to L$, which is injective and surjective at the same time. If a map $f$ is a bijection, then the \deff{inverse map} $f\inv: L\to M$ is defined, namely
$$\textrm{for any } z\in L \textrm{ we define}\qquad f\inv(z) = x,\quad \textrm{ where } x\in M \textrm{ such that } z=f(x).$$
The existence of such an element $x\in M$ follows from the surjectivity of $f$, and its uniqueness  follows from the injectivity $f$.

\item[Permutation] of the set $M$ is any bijection of the set $M$ with itself. 

For an explicit definition of a permutation $s: M\to M$, it is presented as a matrix $2\times |M|$,  where the first row consists of all elements of the set $M$ in some order, and under each element there is its image for the action of $s$ in the second row:
$$s = \begin{pmatrix} x_{1} & \ldots & x_{n} \\ s(x_{1}) & \ldots & s(x_{n}) \end{pmatrix}.$$
For instance, the record $s=\begin{pmatrix} 1 & 2 & 3 \\ 3 & 1 & 2 \end{pmatrix}$ means that $M=\set{1,2,3}$ and $s: 1\mapsto 3$, $2\mapsto 1$, $3\mapsto 2$.

\item[Product] of two permutations $s: M\to M$ and $t:M\to M$, denoted by $s\circ t$, is defined to be the composition of the maps:
$$(s\circ t)(x) = s(t(x)) \quad \textrm{for any elementа } x\in M.$$
The symbol of composition is often omitted, and one writes shortly $st$.

The product of permutations has group features:
\end{description}
\setlength{\leftmargini}{5.2cm} 
\setlength{\labelsep}{0.5cm} 
\begin{itemize}
\item[(\emph{identity element})] $s\circ \id = \id\circ s = s$\quad for any permutationи $s: M\to M$, where $\id$ is \deff{identity permutation} defined as
\begin{equation*}
   \id(x) = x\quad \textrm{for all } x\in M;
\end{equation*}

\item[(\emph{inverse element})] $s\circ s\inv = s\inv\circ s = \id$\quad for any permutation $s: M\to M$, where $s\inv$ is the inverse map (\cf above):
$$\textrm{if }\;\; s = \begin{pmatrix} x_{1} & \ldots & x_{n} \\ y_{1} & \ldots & y_{n} \end{pmatrix}, \textrm{ then }\;\; s\inv = \begin{pmatrix} y_{1} & \ldots & y_{n} \\ x_{1} & \ldots & x_{n} \end{pmatrix};$$

\item[(\emph{associativity})] $s\circ (t\circ u) = (s\circ t)\circ u$\quad for any three permutations $s, t, u: M\to M$.
\end{itemize}

\setlength{\leftmargini}{2.5em} 
\setlength{\labelsep}{0.5em} 
\begin{description}
\item[Symmetric group] of the set $M$ is the group of \emph{all} permutations of $M$. It is denoted by $Sym(M)$. 

If $M=\set{1,2,\ldots,n}$, then one uses the shortened designation  $S_{n}$\,, its order of this group is $n!$.

\item[Permutation group] on the set $M$ is an arbitrary subgroup of the symmetric group $Sym(M)$.

\item[Degree] of a permutation group on $M$ is defined to be the cardinality of the set $M$. \\ For instance, the degree of the group $S_{n}$ equals $n$.

\item[Support] of a permutation $s\in Sym(M)$ is the set of those elements of $M$, which move under the action of $s$. The support is denoted by
$$\supp\; s := \sett{x\in M}{s(x)\neq x}.$$
The set of all points, which are stabilized, is denoted by
\begin{align*}
\fix\; s &:= \sett{x\in M}{s(x) = x} \\
&\,= M\setminus \supp\; s.
\end{align*}

\item[Graph] of permutations $s_{1},\ldots,s_{g}$ of the set $M$ is the following $g$-colored directed graph (possibly with loops): 

$\bullet$\; its vertices are the elements of the set $M$, 

$\bullet$\; two vertices $x$ and $y$ of the graph are connected by a directed edge with the $i^{\textrm{th}}$ color and the starting vertex at $x$ if and only if\; $s_{i}(x) = y$.

The graph of permutations $s_{1},\ldots,s_{g}$ will be denoted by $\graph{s_{1},\ldots,s_{g}}$. In Figure \ref{fig:graph}, one can see the graph of the permutations $s=\begin{pmatrix} 1&2&3&4&5&6&7 \\ 2&3&4&1&6&5&7 \end{pmatrix}$\; and\; $t=\begin{pmatrix} 1&2&3&4&5&6&7 \\ 1&5&3&2&7&4&6 \end{pmatrix}$.
\begin{figure}[htbp]
   \begin{center}
   \begin{tikzpicture}[scale=.8, >=latex,shorten >=0pt,line width=1pt,bend angle=20,
        circ/.style={circle,inner sep=2pt,draw=black!40,rounded corners, text centered}]
      
      \node[circ] (1) at (0,0) {1};
      \node[circ] (2) at (2,2) {2};
      \node[circ] (3) at (4,0) {3};
      \node[circ] (4) at (2,-2) {4};
      \node[circ] (5) at (6,2) {5};
      \node[circ] (6) at (6,-2) {6};
      \node[circ] (7) at (8,0) {7};
      
      \path[->,green!50!black] (1) edge (2)
      				(2) edge (3)
				(3) edge (4)
				(4) edge (1)
				(5) edge[bend left] (6)
				(6) edge[bend left] (5)
				(7) edge[loop right] ();
      \path[->,blue,dashed] (4) edge (2)
      				(2) edge (5)
				(5) edge (7)
				(7) edge (6)
				(6) edge (4)
				(1) edge[loop left] ()
				(3) edge[loop right] ();
       \begin{scope}[scale=0.4,xshift=25cm,yshift=3cm]
    	\draw[->,green!60!black] (0,2) -- (3,2);
    	\draw[->,dashed,blue] (0,0) -- (3,0);
	\draw (3,2) node[right] {$= s$};
	\draw (3,0) node[right] {$= t$};
       \end{scope}
   \end{tikzpicture}
   \caption{$\graph{s,t}$}
   \label{fig:graph}
   \end{center} \vspace{-0.3cm}
\end{figure}

\item[Cycle ($k$-cycle)] on the set $M$ is a permutation $s$ that rearrange some $k$ elements $x_{1},\dots,x_{k}$ from $M$ by the rule
$$s(x_{1})=x_{2},\quad s(x_{2})=x_{3},\quad\ldots,\quad s(x_{k})=x_{1},$$
and the other elements don't move. In this case, one writes $s=(x_{1}\; \ldots\; x_{k})$. The number $k$ is called the \deff{length} of the cycle.
\end{description}
\begin{lemmA}\label{prop:cycles} 
Each permutation $s$ of $M$ is a product of disjoint cycles.
\end{lemmA}
\begin{proof} 
It suffices to consider the graph $\graph{s}$ of the given permutation, and notice that distinct cycles of this graph correspond to disjoint cycles of the permutation $s$. For instance, from Figure \ref{fig:graph} it follows that $s=(1\;2\;3\;4)(5\;6)$ and $t=(2\;5\;7\;6\;4)$.
\end{proof}
When presenting a permutation as a product of disjoint cycles, the cycles of length one are normally omitted (they correspond to the points which are stabilized by the permutation). 

\begin{description}
\item[Cycle type] of a permutation $s\in\sn$ is the formal expression $1^{l_{1}} 2^{l_{2}} \ldots k^{l_{k}}$ meaning that in the decomposition of $s$ into disjoint cycles, there are exactly $l_{1}$ cycles of length $1$, exactly $l_{2}$ cycles of length $2$, and so on, exactly $l_{k}$ cycles of length $k$. In particular, the following equality holds
$$l_{1}+ 2 l_{2} + \dots+k l_{k} = n.$$
The cycle type of a permutation $s$ will be denoted by $\type{s}$. For instance, if $s=(2\; 3)(4\;5\;6\;7\;8)(10\;11)\in S_{15}$\,, then\; $\type{s} = 1^{6} 2^{2} 5^{1}$.

\emph{Two permutations has the same cycle type if and only if they are conjugate}. This follows from the fact that when conjugating a cycle $(a_{1}\;a_{2}\;\dots\;a_{k})$ of an arbitrary permutation $s$, one replaces the points $a_{i}$ in the cycle by their images for the action of $s$\,:
$$s\, (a_{1}\;a_{2}\;\dots\;a_{k})\, s\inv = (s(a_{1})\;s(a_{2})\;\dots\;s(a_{k})).$$
Thus, any two cycles $(a_{1}\;a_{2}\;\dots\;a_{k})$ and $(b_{1}\;b_{2}\;\dots\;b_{k})$ of the same length are conjugate:
$$(b_{1}\;b_{2}\;\dots\;b_{k}) = \begin{pmatrix} a_{1} & a_{2} & \dots & a_{k} \\ b_{1} & b_{2} & \dots & b_{k}\end{pmatrix} (a_{1}\;a_{2}\;\dots\;a_{k}) \begin{pmatrix} a_{1} & a_{2} & \dots & a_{k} \\ b_{1} & b_{2} & \dots & b_{k}\end{pmatrix}\inv.$$

All permutations from $\sn$ with the same cycle type form a \textbf{conjugacy class} of the group $\sn$. 
\end{description}

\begin{lemmA}\label{lemm:conjugacynumber} 
The number of elementов in the conjugacy class $|C|$ of the group $\sn$ corresponding to a cycle type $1^{l_{1}} 2^{l_{2}} \ldots k^{l_{k}}$ can be calculated via the formula
\begin{equation}\label{eq:conjugacynumber}
	|C| = \frac{n!}{1^{l_{1}} l_{1}!\; 2^{l_{2}} l_{2}!\; \cdots\; k^{l_{k}}l_{k}!}\,.
\end{equation}
\end{lemmA}

\begin{description}
\item[Flag] of a permutation $s\in \sn$ is a nondecreasing sequence $a_{1} \le \ldots \le a_{r}$ of positive integers, which are the lengths of the disjoint cycles in the decomposition of $s$ (including the cycles of length $1$). The next equality holds
$$a_{1} + \dots + a_{r} = n\,,$$
\ie one gets a partition of $n$.
The flag of the permutation $s$ is denoted by $\flag{s}$. For instance, if $s=(2\; 3)(4\;5\;6\;7\;8)(10\;11)\in S_{15}$\,, then\; $\flag{s} = 1,1,1,1,1,1,2,2,5$.

As well as for the cycle types, \emph{two permutations have the same flag if and only if they are conjugate}. We conclude with the following statement:
\end{description}
\begin{lemmA}\label{lemm:numberclasses} 
The number of conjugacy classes in $\sn$ is equal to $P(n)$.
\end{lemmA}

\begin{description}
\item[Transposition] is a cycle of length two. According to the relation \eqref{eq:conjugacynumber}, in the group $\sn$ there are exactly $\frac{n!}{1^{n-2}(n-2)!\; 2^{1}1!} = \frac{n(n-1)}{2}$ transpositions.
\end{description}
\begin{lemmA} 
Each permutation in $\sn$ is a product of transposition $($possibly intersecting$)$.
\end{lemmA}
\begin{proof} 
Due to Lemma \ref{prop:cycles}, it suffices to prove the statement for an arbitrary cycle. But for cycles it is true, since  $(x_{1}\; \ldots\; x_{k}) = (x_{1}\; x_{2})(x_{2}\; x_{3})\dots (x_{k-1}\;x_{k})$ when $k> 1$.
\end{proof}

\begin{description}
\item[Even] permutation is a permutation which is a product \emph{even} number of transpositions. The other permutations are called \textbf{odd}. For instance, all cycles of odd length are even permutations, and cycles of even length are odd.

\item[Signature] of a permutation is defined to be $1$ or $-1$ depending on whether the permutation is even or odd respectively. The signature of a permutations $s\in Sym(M)$ is denoted by $\sign{s}$:
\begin{equation*}
	\sign{s} = 
	\begin{cases} 
	1, & \textrm{if } s \textrm{ is even}; \\
	-1, & \textrm{if } s \textrm{ is odd}.
	\end{cases}
\end{equation*}
For instance, $\sign{\id} = \sign{1\;2\;3} = 1$ \,and\, $\sign{1\;2}=-1$. For any permutations $s$ and $t$ the following relation holds:
$$\sign{s\circ t} = \sign{s}\cdot \sign{t}.$$

\item[Alternating group] of the set $M$ is the subgroup ofthose permutations from $Sym(M)$, which are products of \emph{even} number of transpositions (\ie the subgroup of all even permutations). It is denoted by $Alt(M)$, or shortly $A_{n}$ for $M=\set{1,2,\dots,n}$. The order of $A_{n}$ equals $\drob{n!}{2}$ when $n>1$.

\item[Transitive] permutation group on the sets $M$ is a subgroup $G\subseteq Sym(M)$ such that, for any points $x, y\in M$, there exists at least one permutation $s\in G$ sending $x$ to \nolinebreak $y$. 

It easy to show that the group generated by permutationми $s_{1},\ldots,s_{g}$ is transitive if and only if the graph $\graph{s_{1},\ldots,s_{g}}$ is connected.

\item[Block] for a permutation group $G\subseteq Sym(M)$ is a \emph{nonempty} subset $\De\subseteq M$ such that 
$$s(\De) = \De\quad \textrm{or}\quad s(\De)\cap \De=\emptyset\qquad \textrm{for any permutation } s\in G.$$
Evidently, the subsets $\set{x}$  consisting of one element (\emph{singletons}) and the whole set $M$ are blocks for $G$. They are called \deff{trivial}.

For example, for the Klein four-group $K_{4}=\set{\id, (1\;2)(3\;4), (1\;3)(2\;4), (1\;4)(2\;3)}\subset S_{4}$ there are exactly $11$ blocks:
$$\set{1},\; \set{2},\; \set{3},\; \set{4},\; \set{1,2,3,4},\; \set{1,2},\; \set{1,3},\; \set{1,4},\; \set{2,3},\; \set{2,4},\; \set{3,4},$$
the first five of which are trivial.

Remark that if $\De$ is a block for a group $G$, then its image $s(\De)$ for the action of any permutation $s\in G$ is also a block for $G$.
\end{description}

\begin{defi} 
A permutation group on $M$ is said to be \deff{primitive} if \emph{all} blocks for it are trivial.
\end{defi}

Let us show that for $|M|>2$ \emph{any primitive group $G\subseteq Sym(M)$ is automatically transitive}. Indeed, consider the orbit $\De$ of an arbitrary point $x\in M$ for the action of $G$, \ie
$$\De = \sett{s(x)}{s\in G}.$$ 
Since $s(\De) = \De$ for any permutation $s\in G$, then $\De$ is a block, and so it is either a singleton or the whole set $M$. In the first case, we obtain $s(x)=x$ for all $s\in G$ and $x\in M$, \ie  the group $G$ consists of the only permutation $\id$. So, it cannot be primitive when $|M|>2$. In the second case, the group $G$ is transitive.  

The converse is false: a transitive group is not always primitive. Remark, however, that \emph{for a finite set $M$ and a transitive group $G\subseteq Sym(M)$, the cardinality of any block $\De$ divides the cardinality of all set $M$}. Indeed, any two images $s_{1}(\De)$ and $s_{2}(\De)$ of the block for the action of permutations $s_{1}, s_{2}\in G$ either coincide or  don't intersect, by definition of a block. Besides, all images together cover the whole set $M$, since the group $G$ is transitive. Hence, we get a partition of $M$,
$$
M = s_{1}(\De) \sqcup \ldots \sqcup s_{k}(\De)
$$
into several disjoint subsets, the cardinalities of which are equal to the cardinality of the block $\De$. Thus $|\De|$ divides $|M|$. Moreover, each of these subset is also a block for $G$, from where follows the lemma below.

\begin{lemmA}\label{lem:primitive} 
Consider a transitive permutation group $G\subseteq Sym(M)$ and some point $x\in M$. The group $G$ is primitive if and only if the only blocks containing the point $x$ are $\set{x}$ and $M$.
\end{lemmA}

It is easy to check that for $|M|>2$ the symmetric and the alternating groups of the set $M$ are primitive. In 1873, the French mathematician Jordan \cite{jordan73} found a simple criterion whether a given \emph{primitive} group coincides with $Alt(M)$ or $Sym(M)$.
\begin{propA}[\textbf{Jordan's theorem}, 1873]\label{prop:jordan}
Let $G$ be a primitive subgroup of $S_n$.
\begin{enumerate}
   \item If $G$ contains a transposition\index{Transposition}, then $G=S_n$.
   \item If $G$ contains a 3-cycle then, either $G=\sn$ or $\an$.
   \item In general, if $G$ contains a cycle of prime order $p\leqslant n-3$, then either $G=A_n$ or $G=S_n$.
\end{enumerate}
\end{propA}
A proof of this theorem can be found in the textbook of Wielandt \cite[Theorem 13.9]{wielandt}.

\vs
By means of Jordan's theorem, it is not difficult to obtain the following statement (see for example the paper by  Isaacs and Zieschang \cite{isaacs-z}):
\begin{propA}\label{prop:isaacs} 
Let $m$ and $n$ be positive integers, $1<m<n$. Two cycles $(1\;2\;\dots\;m)$ and $(1\;2\;\dots\;n)$ generate the group $A_{n}$ if $m$ and $n$ are both odd, and the whole group $S_{n}$ otherwise. 
\end{propA}

\newpage
\subsection{Representations of the symmetric group}\label{sec:representations}

\begin{description}
\item[Vector space] over $\CC$ is an additive group $V$ together with a rule of multiplication $\la\cdot \vv{v}$ of elements $\vv{v}\in V$ by numbers $\la\in \CC$. Moreover, the following axioms must be satisfied for any $\vv{u},\vv{v}\in V$ \,and\, $\la,\mu\in \CC$\,:
\renewcommand{\labelenumi}{\textbf{\arabic{enumi})}} %
\setlength{\labelsep}{1.25em} 
\begin{enumerate}
\item $1\cdot \vv{v} \;=\; \vv{v}$,
\item $\la\cdot (\vv{u} + \vv{v}) \;=\; \la\cdot\vv{u} + \la\cdot \vv{v}$,
\item $(\la+\mu)\cdot \vv{v} \;=\; \la\cdot\vv{v} + \mu\cdot\vv{v}$,
\item $(\la\mu)\cdot \vv{v} \;=\; \la\cdot (\mu\cdot \vv{v})$.
\end{enumerate}
\setlength{\labelsep}{.7em} 
Elements of the space $V$ are called \textbf{vectors}, and the numbers from $\CC$ are called \textbf{scalars}.

An example of a vector space over $\CC$ is the set $\CC^{d}$ of  $d$-tuples of complex numbers, where $d\in\N$. Addition vectorов and multiplication by scalars are defined as
\begin{eqnarray*}
(x_{1},x_{2},\ldots,x_{d}) + (y_{1},y_{2},\ldots,y_{d}) &=& (x_{1}+y_{1},x_{2}+y_{2},\ldots,x_{d}+y_{d})\,,\\
\la\cdot (x_{1},x_{2},\ldots,x_{d}) &=& (\la x_{1},\la x_{2},\ldots,\la x_{d})\,.
\end{eqnarray*}

\item[Subspace] of a vector space $V$ is any subset $U\subseteq V$ which is itself a vector space with rules of addition of vectorов and multiplication by scalars, inherited from $V$. 

For a subset $U\subseteq V$ to be a  subspace, it is necessary and sufficient that the following properties hold:
\renewcommand{\labelenumi}{\textbf{\arabic{enumi})}} %
\setlength{\labelsep}{1.25em} 
\begin{enumerate}
\item $\vv{0}\in U$,
\item if $\vv{u}$ and $\vv{v}$ lie in $U$, then their sum $\vv{u}+\vv{v}$ also lies in $U$, 
\item if $\vv{v}\in U$, then also $\la\cdot \vv{v}\in U$ for any scalar $\la\in\CC$.
\end{enumerate}
\setlength{\labelsep}{.7em} 
The set $\set{\vv{0}}$ and the entire $V$ are obviously subspaces in $V$, they are called the \textbf{trivial subspaces}. All other subspaces are called \textbf{proper}.

An example of a proper subspace in $\CC^{d}$ for $d>1$ is the set of $d$-tuples of complex numbers $(x_{1},x_{2},\ldots,x_{d})$ with the condition $x_{1}=0$. 

\item[Matrix] of size $d\times d$\; is a square table having $d$ rows and $d$ columns and filled with numbers:
\begin{equation*}
A = 
\begin{pmatrix} 
a_{11} & a_{12} & \cdots & a_{1d}\\
a_{21} & a_{22} &  & a_{2d} \\
\vdots &  & \ddots & \vdots \\
a_{d1} & a_{d2} & \cdots & a_{dd}\\
\end{pmatrix}	
\end{equation*}
One writes shortly $A=(a_{ij})_{d\times d}$ or $A=(a_{ij})$,  where $a_{ij}$ is the element of the matrix at the intersection of the $i^{\textrm{th}}$ row and the $j^{\textrm{th}}$ column. A matrix is said to be \textbf{complex} if all its elements are complex (\ie of the form $a+b\sqrt{-1}$ for some real numbers $a$ and $b$).

The complex $d\times d$  matrices act on the vectors from the space $\CC^{d}$ in the following way:
\begin{equation*}
	A\cdot \vv{v} \;=\; 
\begin{pmatrix} 
a_{11} & a_{12} & \cdots & a_{1d}\\
a_{21} & a_{22} &  & a_{2d} \\
\vdots &  & \ddots & \vdots \\
a_{d1} & a_{d2} & \cdots & a_{dd}\\
\end{pmatrix}	
\cdot 
\begin{pmatrix} x_{1} \\ x_{2} \\ \vdots \\ x_{d} \end{pmatrix}
\;=\;
\begin{pmatrix} 
a_{11} x_{1} + a_{12} x_{2} + \cdots + a_{1d} x_{d} \\
a_{21} x_{1} + a_{22} x_{2} + \cdots + a_{2d} x_{d} \\
\vdots \\
a_{d1} x_{1} + a_{d2} x_{2} + \cdots + a_{dd} x_{d} 
\end{pmatrix}
\,,
\end{equation*}
where a vector $\vv{v}\in \CC^{d}$ is written as a  column with $d$ complex numbers. For instance,
\begin{equation*}
\begin{pmatrix} 
a&b \\ c&d 
\end{pmatrix}	
\cdot
\begin{pmatrix} x \\ y \end{pmatrix}
=
\begin{pmatrix} a x + b y \\ c x + d y \end{pmatrix}\,.
\end{equation*}
It is easy to check that for any vectors $\vv{u},\vv{v}\in \CC^{d}$ and scalars $\la,\mu\in \CC$  the following equality holds:
\begin{equation*}
	A\cdot (\la \vv{u} + \mu \vv{v}) = \la\cdot (A\cdot \vv{u}) + \mu\cdot (A\cdot \vv{v}).
\end{equation*}

\item[Identity matrix] is a matrix with ones on the main diagonal (from the upper left corner to the lower right corner) and zeros everywhere else. The identity $d\times d$ matrix is denoted by $I_{d}$\,:
\begin{equation*}
	I_{d} = 
\begin{pmatrix} 
1 & 0 & \cdots & 0\\
0 & 1 &  & 0 \\
\vdots &  & \ddots & \vdots \\
0 & 0 & \cdots & 1\\
\end{pmatrix}\,.	
\end{equation*}

\item[Trace] of a matrix is defined as the sum of the numbers on the main diagonal. The trace of a matrix $A$  is denoted by $\tr(A)$\,:
$$\tr(A) = a_{11} + a_{22} + \cdots + a_{dd}\,.$$
For instance, $\tr(I_{d}) = d$.

\item[Determinant] of a $d\times d$ matrix is given by the following formula:
\begin{equation}\label{eq:det}
	\det(A) = \sum_{s\in S_{d}} \sign{s}\cdot a_{1,s(1)}\, a_{2,s(2)}\,\ldots\, a_{d,s(d)}\,, 
\end{equation}
where one sums all products of $d$ elements of the matrix, no two of which are in the same row or the same column. In a product $a_{1,s(1)} a_{2,s(2)}\,\cdots\, a_{d,s(d)}$ the first multiplier lies in the column $s(1)$, the second multiplier lies in the column $s(2)$ and so on, where $s$ is some permutation of the numbers from $1$ to $d$. In the sum, such a product is counted with a coefficient $+1$ or $-1$ depending on the parity of $s$ (recall that $\sign{s}$ denotes the signature of the permutation).

For instance, for the $2\times 2$ matrices one has
\begin{equation*}
\tr 
\begin{pmatrix} a & b \\ c & d \end{pmatrix} = a+d
\qquad\textrm{and}\qquad
\det
\begin{pmatrix} a & b \\ c & d \end{pmatrix} = a d - b c\,.
\end{equation*}

\item[Product of matrices] $A=(a_{ij})_{d\times d}$ and $B=(a_{ij})_{d\times d}$ is the $d\times d$ matrix with the element $\sum\limits_{k=1}^{d} a_{ik} b_{kj}$ at the intersection of the $i^{\textrm{th}}$ row and the $j^{\textrm{th}}$ column:
\begin{equation*}
	A\cdot B = 
\begin{pmatrix} 
a_{11} b_{11} + a_{12} b_{21} +\cdots a_{1d} b_{d1} & a_{11} b_{12} + a_{12} b_{22} +\cdots a_{1d} b_{d2} & \cdots & a_{11} b_{1d} + a_{12} b_{2d} +\cdots a_{1d} b_{dd}\\
a_{21} b_{11} + a_{22} b_{21} +\cdots a_{2d} b_{d1}  & a_{21} b_{12} + a_{22} b_{22} +\cdots a_{2d} b_{d2} &  & a_{21} b_{1d} + a_{22} b_{2d} +\cdots a_{2d} b_{dd} \\
\vdots &  & \ddots & \vdots \\
a_{d1} b_{11} + a_{d2} b_{21} +\cdots a_{dd} b_{d1}  & a_{d1} b_{12} + a_{d2} b_{22} +\cdots a_{dd} b_{d2} & \cdots & a_{d1} b_{1d} + a_{d2} b_{2d} +\cdots a_{dd} b_{dd}\\
\end{pmatrix}\,.	
\end{equation*}
In some sense, when multiplying two matrices ``the rows of the first matrix are multiplied by the  columns of the second matrix''. The product of two matrices is not always commutative. For instance,
\begin{equation*}
\begin{pmatrix} 1&1 \\ 0&1 \end{pmatrix}
\begin{pmatrix} 1&0 \\ 1&1 \end{pmatrix}=
\begin{pmatrix} 2&1 \\ 1&1 \end{pmatrix}\,,
\qquad\textrm{but}\qquad
\begin{pmatrix} 1&0 \\ 1&1 \end{pmatrix}
\begin{pmatrix} 1&1 \\ 0&1 \end{pmatrix}=
\begin{pmatrix} 1&1 \\ 1&2 \end{pmatrix}\,.	
\end{equation*}
The product of matrices has the following properties:
\renewcommand{\labelenumi}{\textbf{\alph{enumi}.}} %
\begin{enumerate}
   \item For any $d\times d$ matrix $A$ one has the equality $$A\cdot I_{d} = I_{d}\cdot A = A$$ 
   
   \item If $\det(A)\neq 0$, then there exists a (unique) matrix denoted by $A^{-1}$ such that $$A\cdot A^{-1} = A^{-1}\cdot A = I_{d}$$ Such a matrix is called the \textbf{inverse of $A$}.
   
   \item For any three $d\times d$ matrices $A$, $B$ and $C$, one has the equality
   $$A\cdot (B\cdot C) = (A\cdot B)\cdot C$$

   \item The determinant of a product equals a product of the determinants,
   $$\det(A\cdot B) = \det(A)\cdot \det(B)$$
\end{enumerate}

In view of these properties, all complex $d\times d$ matrices with \emph{nonzero} determinant form a group, which is called the \textbf{general linear group} and denoted by $\gl_{d}(\CC)$.
$$\gl_{d}(\CC) := \sett{ A=(a_{ij})_{d\times d} }{ a_{ij}\in\CC\;\textrm{ and }\; \det(A)\neq 0 }.$$

\item[Homomorphism] from a group $G$ to a group $K$ is a map $f: G\to K$ such that 
$$f(g_{1}g_{2}) = f(g_{1})\cdot f(g_{2})\qquad\textrm{for any } g_{1}, g_{2}\in G.$$
In other words, under homomorphisms the image of a product (of two elements from $G$) is a product of their images (in $K$). 

If a homomorphism $f$ is moreover a bijection, then it is called an \textbf{isomorphism}. Isomorphism is a bijective homomorphism.

\item[Representation] of a group $\sn$ is any homomorphism $\rho: \sn\to \gl_{d}(\CC)$, where $d\in\N$. 

The natural number $d$ is called the \textbf{dimension} of the representationя $\rho$ and  denoted  $\dim \rho$.

Given a representation $\rho: \sn\to \gl_{d}(\CC)$, to each permutation $s\in \sn$ corresponds a $d\times d$ complex matrix $\rho(s)$, so that the matrix of a product of two permutations $\rho(st)$ is equal to the product of matrices $\rho(s)\cdot \rho(t)$ for any $s,t\in \sn$.

Two representations $\rho: \sn\to \gl_{d}(\CC)$ and $\eta: \sn\to \gl_{k}(\CC)$ are said to be \textbf{equivalent} and one writes $\rho\sim \eta$, if $d=k$ and there exists a matrix $A\in\gl_{d}(\CC)$ such that 
\begin{equation*}
	\rho(s) = A\cdot \eta(s)\cdot A\inv \qquad\textrm{for any permutation } s\in \sn\,.
\end{equation*}

\item[Irreducible representation] of the group $\sn$ is such representation $\rho: \sn\to \gl_{d}(\CC)$\,, for which there is \emph{no} proper subspace $V$ of the vector space $\CC^{d}$ with the property:
$$\rho(s)\cdot \vv{v} \;\in\; V \qquad\textrm{for any permutation } s\in \sn \textrm{ and any vector } \vv{v}\in V .$$
In other words, only the subspace $\set{\vv{0}}$ and the entire space $\CC^{d}$ are invariant under the action of all matrices $\rho(s)$.

\end{description}

\begin{lemmA}\label{lemm:numberrepr} 
The number of pairwise nonequivalent irreducible representations of the group $\sn$ is equal to the number of conjugacy classes in $\sn$\,, \ie to the number of partitions of $n$.
\end{lemmA}

\begin{description}
\item[Character] of a representation $\rho$ of $\sn$ is the function $\upchi_{\rho} : \sn\to \CC$ such that
\begin{equation*}
	\upchi_{\rho}(s) = \tr(\rho(s))\qquad\textrm{for any permutation } s\in \sn.
\end{equation*}

This means that a character $\upchi_{\rho}$ is given by a tuple of \,$n!$\, complex numbers (possibly equal), which are the traces of the matrices $\rho(s)$ for $s\in\sn$. Here are some properties of the characters:

\renewcommand{\labelenumi}{\textbf{\alph{enumi}.}} %
\begin{enumerate}
   \item For any representationя $\rho: \sn\to \gl_{d}(\CC)$, one has 
   $$\upchi_{\rho}(\id) = \dim \rho\,.$$ 
   Indeed, the homomorphism $\rho$ sends the identity permutation $\id$ to the identity matrix $I_{d}$\,, \ie $\upchi_{\rho}(\id) = \tr(\rho(\id)) = \tr(I_{d}) = d = \dim \rho$.
   
   \item If permutations $s$ and $t$ are conjugate, then
   $$\upchi_{\rho}(s) = \upchi_{\rho}(t).$$
   This means that any character $\upchi_{\rho}$ of the group $\sn$ can take at most $P(n)$ distinct values, because at elements of each conjugacy class the same value (\cf Lemma \ref{lemm:numberclasses}).
   
   \item If representations $\rho$ and $\eta$ of the group $\sn$ are equivalent, then their characters are equal:
   $$\upchi_{\rho}(s) = \upchi_{\eta}(s)\qquad \textrm{for any permutationи } s\in\sn\,.$$

   \item For any representation $\rho: \sn\to \gl_{d}(\CC)$ and any permutations $s\in\sn$ the value of the character $\upchi_{\rho}(s)$ is \emph{integer}. This can be proved by means of  Galois theory. 
\end{enumerate}
According to Lemma \ref{lemm:numberrepr} the groups $S_{2}$, $S_{3}$ and $S_{4}$ have respectively $P(2) = 2$, $P(3) = 3$ and $P(4)=5$ pairwise nonequivalent irreducible representations. In the tables below, one finds the characters of those representations, where in the first row we indicated permutations from distinct conjugacy classes of $\sn$\,, in the second row -- the number of permutations in those classes, and in the other rows -- the values of the characters at those permutations.
\begin{table}[htdp]
\caption{The characters of the irreducible representations of $S_{2}$ and $S_{3}$.}
\begin{center}
\mvs
\begin{tikzpicture}
\node[below] at (0,0) {
$
\begin{array}{c|cc}
C & \id & (1\;2) \\ \hline
|C| & 1 & 1 \\ \hline\hline
\upchi_{1} & 1 & 1\\
\upchi_{2} & 1 & -1\phantom{-}
\end{array}
$};
	
\node[below] at (6,0) {
$
\begin{array}{c|ccc}
C & \id & (1\;2) & (1\;2\;3)\\ \hline
|C| & 1 & 3 & 2 \\ \hline\hline
\upchi_{1} & 1 & 1 & 1\\
\upchi_{2} & 1 & -1\phantom{-} & 1\\
\upchi_{3} & 2 & 0 & -1\phantom{-}
\end{array}	
$};
\end{tikzpicture}
\end{center}
\label{table:charcters2s3}
\end{table}%
\mvs\mvs

\begin{table}[htdp]
\caption{The characters of the irreducible representations of the group $S_{4}$.}
\begin{center}
\mvs
$
\begin{array}{c|ccccc}
C & \id & (1\;2) & (1\;2\;3) & (1\;2)(3\;4) & (1\;2\;3\;4)\\ \hline
|C| & 1 & 6 & 8 & 3 & 6 \\ \hline\hline
\upchi_{1} & 1 & 1 & 1 & 1 & 1 \\
\upchi_{2} & 1 & -1\phantom{-} & 1 & 1 & -1\phantom{-}\\
\upchi_{3} & 2 & 0 & -1\phantom{-} & 2 & 0 \\
\upchi_{4} & 3 & 1 & 0 & -1\phantom{-} & -1\phantom{-}\\
\upchi_{5} & 3 & -1\phantom{-} & 0 & -1\phantom{-} & 1
\end{array}	
$
\end{center}
\label{table:charcters2s3}
\end{table}%

\end{description}

\begin{propA}[\textbf{Frobenius's formula}]\label{prop:frobenius} 
Consider an arbitrary permutation $w\in \sn$ and its conjugacy class $C$. The number of pairs of permutations $s$ and $t$ from $\sn$ such that $\type{[s,t]} = \type{w}$ is equal to
\begin{equation*}
	n!\cdot |C|\cdot \sum_{\rho} \frac{\upchi_{\rho}(w)}{\dim \rho}\,, 
\end{equation*}
 where the sum is over all $($pairwise nonequivalent$)$ irreducible representations $\rho$ of  $\sn$.
\end{propA}
For instance, if $w=(1\;2\;3)$ then according to the relation \eqref{eq:conjugacynumber} the number of permutations conjugate to $w$ is equal to $|C|=\frac{n!}{1^{n-3}(n-3)!\; 3^{1} 1!} = \frac{n(n-1)(n-2)}{3}$. Therefore, the number of pairs $(s,t)\in\snsn$ with a 3-cycle commutator is equal to the expression
\begin{equation*}
	n!\cdot\frac{n(n-1)(n-2)}{3}\cdot \sum_{\rho} \frac{\upchi_{\rho}((1\;2\;3))}{\dim \rho}\,. 
\end{equation*}
A proof of the generalized Frobenius's formula can be found in the work \cite{jones} and in the survey \cite{zagier}.

\subsection{Square-tiled surfaces}\label{sec:origamis}
\begin{defi}
A \deff{square-tiled surface}, or shortly an \deff{origami}, is a finite collection of copies of the unitary Euclidian square together with a gluing of the edges:
\begin{itemize}
   \item the right edge of each square is identified to the left edge of some square,
   	
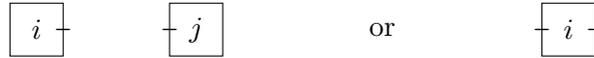
\begin{center} 
   	\begin{tikzpicture}[scale=.7]
   	\draw (0,0) rectangle +(1,1);
	\node at (.5,.5) {$i$};
	\node at (1,.5) {$-$};

   	\draw (3,0) rectangle +(1,1);
	\node at (3.5,.5) {$j$};
	\node at (3,.5) {$-$};
	\node at (7,.5) {or};

   	\draw (10,0) rectangle +(1,1);
	\node at (10.5,.5) {$i$};
	\node at (10,.5) {$-$};
	\node at (11,.5) {$-$};
	\end{tikzpicture}
	\end{center}
   \item the top edge of each square is identified to the bottom edge of some square.
   	\begin{center} \mhs\mhss
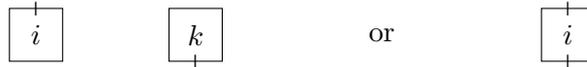
   	\begin{tikzpicture}[scale=.7]
   	\draw (0,0) rectangle +(1,1);
	\node at (.5,.5) {$i$};
	\node[rotate=90] at (.5,1) {$-$};

   	\draw (3,0) rectangle +(1,1);
	\node at (3.5,.5) {$k$};
	\node[rotate=90] at (3.5,0) {$-$};
	\node at (7,.5) {or};

   	\draw (10,0) rectangle +(1,1);
	\node at (10.5,.5) {$i$};
	\node[rotate=90] at (10.5,1) {$-$};
	\node[rotate=90] at (10.5,0) {$-$};
	\end{tikzpicture}
	\end{center}
\end{itemize}
\end{defi}
The notion of square-tiled surfaces came into sight in 1970s through the works of William P.\ Thurston \cite{thurston} and William A. Veech \cite{veech1}. Nowadays, these objects are being actively studied.

\vs Consider an arbitrary origami $O$ with $n$ squares. Number its squares by the integers from $1$ to $n$. Then we will obtain a pair of permutations $(s,t)\in\snsn$ that indicates how the squares are glued in the horizontal and the vertical directions:
\begin{eqnarray*}
   s(i) &=& j, \quad \textrm{if the right edge of the $i^\textrm{th}$ square is glued to the left edge of the $j^\textrm{th}$ one}, \\
   t(i) &=& k, \quad \textrm{if the top edge of the $i^\textrm{th}$ square is glued to the bottom edge of the $k^\textrm{th}$ one}. 
\end{eqnarray*}

Note that by definition, a square-tiled surface doesn't depend on the numbering of its squares. If we number them otherwise, then we will get another pair of permutations $(s',t')$, which is conjugate to $(s,t)$:
$$s' = u s u\inv\quad\textrm{and}\quad t' = u t u\inv\qquad \textrm{for some } u\in\sn.$$

The origami corresponding to a given pair of permutations $(s,t)\in\snsn$ will be denoted by $O(s,t)$. 
The permutation group $G=\group{s,t}$ generated by the pair $(s,t)$ is called the \deff{monodromy group} of the surface $O=O(s,t)$. One says that an origami is \deff{connected} if its monodromy group is transitive. An origami is \deff{primitive} if its monodromy group is primitive. The reader can find more information on origamis and their monodromy groups in the  Ph.D.\ thesis of David Zmiaikou \cite[2011]{zmiaikou}.

\subsection{Lattices in $\R^{2}$}\label{sec:lattice}
\begin{description}
\item[Basis] of the space $\R^{2}$ is a pair of its vectors $\vv{v}_{1}$ and $\vv{v}_{2}$ such that any vector $\vv{w}\in \R^{2}$ can be presented as a linear combination 
$$\vv{w} = r_{1} \vv{v}_{1} + r_{2} \vv{v}_{2}$$
with real coefficients $r_{1}$ and $r_{2}$.

For instance, the vectors $\vv{(1,0)}$ and $\vv{(0,1)}$ form a basis for the space $\R^{2}$, which is called \textbf{canonical}. It is easy to show that  a pair of vectors $\vv{(x_{1},y_{1})}$ and $\vv{(x_{2},y_{2})}$ is a basis for $\R^{2}$ if and only if the following determinant of the matrix of these vectors 
$$\det \begin{pmatrix} x_{1} & y_{1} \\ x_{2} & y_{2} \end{pmatrix} = x_{1} y_{2} - x_{2} y_{1}$$
differs from zero.

\item[Lattice] generated by a basis $(\vv{v}_{1}, \vv{v}_{2})$ is the set of linear combinations 
$$a_{1} \vv{v}_{1} + a_{2} \vv{v}_{2}$$
with \emph{integer} coefficients $a_{1}$ and $a_{2}$.

We conclude that a lattice is an additive subgroup of $\R^{2}$ with two generators.
As a simple example of a lattice, one has $\Z^{2}$. Note that two distinct bases  can generate the same lattice. For instance,
$$\Z^{2} 	= \sett{ a_{1} \vv{(1,0)} + a_{2} \vv{(0,1)}}{ a_{1},a_{2}\in\Z }
		= \sett{ a_{1} \vv{(3,2)} + a_{2} \vv{(7,5)}}{ a_{1},a_{2}\in\Z }.$$
By the way, a pair of vectors $\vv{(x_{1},y_{1})}$ and $\vv{(x_{2},y_{2})}$ with integer coordinates generate the lattice $\Z^{2}$ if and only if the determinant of the matrix for these vectors equals $1$ or $-1$: 
$$\det \begin{pmatrix} x_{1} & y_{1} \\ x_{2} & y_{2} \end{pmatrix} = \pm 1.$$
All integer $2\times 2$ matrices with determinantем $1$ or $-1$ form the group, which is denoted by $\glzz$.

\vs
\begin{figure}[htbp]
   \begin{center}
   \begin{tikzpicture}[scale=.4, >=latex,shorten >=0pt,shorten <=0pt,line width=1pt,bend angle=20, 
   		circ/.style={circle,inner sep=2pt,draw=black!40,rounded corners, text centered},
   		tblue/.style={dashed,blue,rounded corners}]
	
      	\node[left] at (0,-0.5){\small $O$};
      	\node[right] at (3,-0.5){\small $A$};
      	\node[above] at (2,2.3){\small $B$};
      
      	\clip (-14.5,-7.5) rectangle +(30,15);

      	\fill[green!60!white] (0,0) -- ++(2,2) -- ++(3,0) -- ++(-2,-2) -- ++(-3,0);
	\foreach \x in {-8,...,8}
		\foreach \y in {-8,...,8}
      			\fill[blue] (\x*3 + \y*2,\y*2) circle (4pt); 
      	\path[->,line width=1pt] (0,0) 	edge node[below,yshift=1pt]{\small $\vv{u}$} (3,0)
      					 		edge node[above,xshift=-4pt,yshift=-2pt]{\small $\vv{v}$} (2,2);

   \end{tikzpicture}
   \caption{Решетка в $\R^{2}$.}
   \label{fig:lattice}
   \end{center} \vspace{-0.3cm}
\end{figure}
An arbitrary lattice $L\subset \R^{2}$ is usually illustrated in the real plane as a set of points: for each vector $\vv{(x,y)} \in L$ one draws the point with coordinates $(x,y)$. Let us find a basis $(\vv{u}, \vv{v})$ for a lattice $L$ as follows (\cf Figure \ref{fig:lattice}): 
\begin{itemize}
   \item Let $A$ be the closest lattice point to the origin of the plane $O$ lying on the horizontal axis and having a positive first coordinate. Take $\vv{u}=\vv{OA}$. 

   \item Let $B$ be  the lattice point  such that the largest of the distances from it to the points $O$ and $A$ is the least (among all lattice points). Take $\vv{v}=\vv{OB}$. 
\end{itemize}
It is not difficult to show that the basis $(\vv{u}, \vv{v})$ of $\R^{2}$, indeed, generates the lattice $L$. 
The parallelogram constructed on the vectors $\vv{u}$ and $\vv{v}$ is called the \textbf{fundamental parallelogram} of the lattice $L$.
\end{description}

\subsection{Several useful sums}
\noindent For any positive integer $n$, the following relations hold:
\setlength{\jot}{10pt} 
\begin{align}
1^{2} + 2^{2} + \dots + n^{2} &= \frac{n(n+1)(2n+1)}{6} = \frac{2n^{3}+3n^{2}+n}{6}, \label{eq:n2}\\ 
1^{3} + 2^{3} + \dots + n^{3} &= \frac{n^{2}(n+1)^{2}}{4} = \frac{n^{4}+2n^{3}+n^{2}}{4}, \label{eq:n3}\\
1^{4} + 2^{4} + \dots + n^{4} &= \frac{n(n+1)(2n+1)(3n^{2}+3n-1)}{30} = \frac{6n^{5}+15n^{4}+10n^{3}-n}{30}.\label{eq:n4}
\end{align}
In 1631, the German mathematician Faulhaber obtained a general formula for the sum of powers of the first $n$ positive integers.

Using the formula \eqref{eq:n2}, or else by induction, one can express the sum
\begin{equation}\label{eq:knk}
   \sum_{k=2}^{n} k(n-k) = \frac{(n+3)(n-1)(n-2)}{6}.
\end{equation}

Let $r$ be an arbitrary positive integer. The foolowing equality
\begin{equation}\label{eq:kk1}
   \sum_{k=1}^{n} k(k-1)\cdots (k-r) = \frac{(n+1)n(n-1)\dots (n-r)}{r+2}
\end{equation}
can be proved by induction on $n$. For $n=1$, it is trivially true. Suppose that it is true for $n$ and show it for $n+1$:
\begin{align*} 
  \sum_{k=1}^{n+1} k(k-1)\cdots (k-r) &= (n+1)n\cdots (n-r+1) + \sum_{k=1}^{n} k(k-1)\cdots (k-r) \\
  & = (n+1)n\cdots (n-r+1) + \frac{(n+1)n\cdots (n-r)}{r+2} \\
  & = (n+1)n\cdots (n-r+1)\left(1 + \frac{n-r}{r+2}\right) \\
  & = \frac{(n+2)(n+1)n\cdots (n+1-r)}{r+2}.
\end{align*}
Thus, the equality \eqref{eq:kk1} holds for any natural $n$.

\vs The following infinite series is well-known:
\begin{equation}\label{eq:pi2}
   1 + \frac{1}{2^{2}} + \frac{1}{3^{2}} + \dots + \frac{1}{n^{2}} + \dots = \frac{\pi^{2}}{6}.
\end{equation}
Let $p_{k}$ denote the $k^{\textrm{th}}$ prime number ($p_{1}=2$, $p_{2}=3$, $p_{3}=5$ and so on). Then by the formula for the sum of an infinite geometric progression,
\begin{eqnarray*}
 \frac{1}{\prod\limits_{k=1}^{\infty}\left( 1 - \dfrac{1}{p_{k}^{2}} \right)}  &=& 
   \left( 1 + \frac{1}{p_{1}^{2}} +  \dots + \frac{1}{p_{1}^{2m}} + \cdots  \right)
   \left( 1 + \frac{1}{p_{2}^{2}} +  \dots + \frac{1}{p_{2}^{2m}} + \cdots  \right) 
   \cdots \left( 1 + \frac{1}{p_{k}^{2}} + \dots + \frac{1}{p_{k}^{2m}} + \cdots  \right)
   \cdots\\
   &=& 1 + \frac{1}{2^{2}} + \frac{1}{3^{2}} + \dots + \frac{1}{n^{2}} + \dots 
\end{eqnarray*}
from where we get
\begin{equation}\label{eq:6pi2}
   \prod_{k=1}^{\infty}\left( 1 - \frac{1}{p_{k}^{2}} \right) = \frac{6}{\pi^{2}} \,.
\end{equation}

\newpage
\setcounter{equation}{0}
\setcounter{table}{0}
\section{Programs (Sage)}
\subsection{Calculating $\#\A(n)$ and $\#\B(n)$ by Theorem \ref{theor:A}}

\begin{footnotesize}
\begin{lstlisting}[language=Python] 
# Calculating the cardinalities of A(n) and B(n) divided by n!
# and the probability #A(n)/#B(n) multiplied by P(n)/n, 
# where P(n) is the partition function

def A(n):    # the cardinality of A(n) divided by n!
    A = 3*(n-2)*n*n*prod([1-1/(p*p) for p in prime_divisors(n)])/8
    return A
    
def B(n):    # the cardinality of B(n) divided by n!
    B = 0
    for k in range(1,n+1):
        B = B + ( sigma(k,3) - 2*k*sigma(k,1) + sigma(k,1) ) \
                * Partitions(n-k).cardinality()
    return 3*B/8
    
for n in range(3,256):
    a = A(n)
    b = B(n)
    proba = numerical_approx(Partitions(n).cardinality()*a/(n*b), digits=6) 
    print n,a,b,proba
\end{lstlisting}
\end{footnotesize}

\vvs
\subsection{Calculating $\#\A_{1}(n)$, $\#\B_{1}(n)$, $\#\A_{2}(n)$ and $\#\B_{2}(n)$ by Theorem  \ref{theor:B}}
\begin{footnotesize}
\begin{lstlisting}[language=Python] 
# Calculating the cardinalities of A1(n), B1(n), A2(n) and B2(n) divided by n!
# and the fractions #A1(n)/#B1(n)  and  n #A1(n)/#B1(n)

def A1(n):    # the cardinality of A1(n) divided by n!
    A1 = n*n*prod([1-1/(p*p) for p in prime_divisors(n)]) - 3*euler_phi(n)
    return n*A1/6
    
def B1(n):    # the cardinality of B1(n) divided by n!
    B1 = n*(n-1)*(n-2)/6
    return B1
    
def A2(n):    # the cardinality of A2(n) divided by n!
    A2 = A1(n) + (n+1)*(n-2)/2
    return A2
    
def B2(n):    # the cardinality of B2(n) divided by n!
    B2 = (n-1)*(n-2)*(n^2+5*n+12)/24
    return B2

for n in range(3,256):
    a1 = A1(n); a2 = A2(n) 
    b1 = B1(n); b2 = B2(n)
    proba1 = numerical_approx(a1/b1, digits=6) 
    proba2 = numerical_approx(n*a2/b2, digits=6) 
    print n,a1,b1,proba1,a2,b2,proba2

\end{lstlisting}
\end{footnotesize}

\newpage
\subsection{The table for the first program}

\begin{footnotesize}

\begin{longtable}{l|rr|r}
$n$ & $\dfrac{\#\mathcal{A}(n)}{n!}$ & $\dfrac{\#\mathcal{B}(n)}{n!}$ & $\dfrac{P(n)}{n}\cdot\dfrac{\#\mathcal{A}(n)}{\#\mathcal{B}(n)}$ \\
\hline\hline
$n=$ 3 & 3 & 3 & 1.00000 \\
$n=$ 4 & 9 & 12 & 0.937500 \\
$n=$ 5 & 27 & 42 & 0.900000 \\
$n=$ 6 & 36 & 99 & 0.666667 \\
$n=$ 7 & 90 & 231 & 0.834879 \\
$n=$ 8 & 108 & 462 & 0.642857 \\
$n=$ 9 & 189 & 882 & 0.714286 \\
$n=$ 10 & 216 & 1596 & 0.568421 \\
$n=$ 11 & 405 & 2772 & 0.743802 \\
$n=$ 12 & 360 & 4620 & 0.500000 \\
$n=$ 13 & 693 & 7524 & 0.715587 \\
$n=$ 14 & 648 & 11949 & 0.522937 \\
$n=$ 15 & 936 & 18480 & 0.594286 \\
$n=$ 16 & 1008 & 28182 & 0.516393 \\
$n=$ 17 & 1620 & 42108 & 0.672137 \\
$n=$ 18 & 1296 & 62139 & 0.446097 \\
$n=$ 19 & 2295 & 90216 & 0.656057 \\
$n=$ 20 & 1944 & 129690 & 0.469924 \\
$n=$ 21 & 2736 & 183876 & 0.561173 \\
$n=$ 22 & 2700 & 258720 & 0.475312 \\
$n=$ 23 & 4158 & 359667 & 0.630812 \\
$n=$ 24 & 3168 & 496650 & 0.418605 \\
$n=$ 25 & 5175 & 678942 & 0.596967 \\
$n=$ 26 & 4536 & 922824 & 0.460530 \\
$n=$ 27 & 6075 & 1243284 & 0.544727 \\
$n=$ 28 & 5616 & 1666434 & 0.447497 \\
$n=$ 29 & 8505 & 2216676 & 0.603969 \\
$n=$ 30 & 6048 & 2934960 & 0.384934 \\
$n=$ 31 & 10440 & 3860076 & 0.596934 \\
$n=$ 32 & 8640 & 5055468 & 0.445899 \\
$n=$ 33 & 11160 & 6582114 & 0.521136 \\
$n=$ 34 & 10368 & 8536704 & 0.439728 \\
$n=$ 35 & 14256 & 11013387 & 0.550426 \\
$n=$ 36 & 11016 & 14158620 & 0.388524 \\
$n=$ 37 & 17955 & 18115944 & 0.579588 \\
$n=$ 38 & 14580 & 23103531 & 0.432035 \\
$n=$ 39 & 18648 & 29339079 & 0.508238 \\
$n=$ 40 & 16416 & 37143414 & 0.412550 \\
$n=$ 41 & 24570 & 46842642 & 0.570360 \\
$n=$ 42 & 17280 & 58906848 & 0.371388 \\
$n=$ 43 & 28413 & 73816743 & 0.566278 \\
$n=$ 44 & 22680 & 92254470 & 0.420026 \\
$n=$ 45 & 27864 & 114926262 & 0.480236 \\
$n=$ 46 & 26136 & 142810932 & 0.419963 \\
$n=$ 47 & 37260 & 176935080 & 0.558966 \\
$n=$ 48 & 26496 & 218698536 & 0.371720 \\
$n=$ 49 & 41454 & 269577000 & 0.544565 \\
$n=$ 50 & 32400 & 331556148 & 0.399143 \\
$n=$ 51 & 42336 & 406749651 & 0.489689 \\
$n=$ 52 & 37800 & 497949144 & 0.411073 \\
$n=$ 53 & 53703 & 608155506 & 0.549707 \\
$n=$ 54 & 37908 & 741282927 & 0.365691 \\
$n=$ 55 & 57240 & 901553268 & 0.520940 \\
$n=$ 56 & 46656 & 1094417478 & 0.401052 \\
$n=$ 57 & 59400 & 1325794470 & 0.482739 \\
$n=$ 58 & 52920 & 1603220388 & 0.407041 \\
$n=$ 59 & 74385 & 1934935068 & 0.541996 \\
$n=$ 60 & 50112 & 2331328074 & 0.346238 \\
$n=$ 61 & 82305 & 2803785600 & 0.539700 \\
$n=$ 62 & 64800 & 3366550440 & 0.403639 \\
$n=$ 63 & 79056 & 4035301935 & 0.468165 \\
$n=$ 64 & 71424 & 4829450472 & 0.402460 \\
$n=$ 65 & 95256 & 5770461060 & 0.511113 \\
$n=$ 66 & 69120 & 6884707896 & 0.353444 \\
$n=$ 67 & 109395 & 8201402319 & 0.533481 \\
$n=$ 68 & 85536 & 9756209694 & 0.398106 \\
$n=$ 69 & 106128 & 11588746128 & 0.471741 \\
$n=$ 70 & 88128 & 13747002864 & 0.374382 \\
$n=$ 71 & 130410 & 16284447375 & 0.529809 \\
$n=$ 72 & 90720 & 19265466990 & 0.352699 \\
$n=$ 73 & 141858 & 22761858636 & 0.528094 \\
$n=$ 74 & 110808 & 26859653700 & 0.395234 \\
$n=$ 75 & 131400 & 31654953792 & 0.449320 \\
$n=$ 76 & 119880 & 37262223054 & 0.393222 \\
$n=$ 77 & 162000 & 43809552402 & 0.510005 \\
$n=$ 78 & 114912 & 51448804176 & 0.347403 \\
$n=$ 79 & 180180 & 60349988160 & 0.523371 \\
$n=$ 80 & 134784 & 70713852366 & 0.376361 \\
$n=$ 81 & 172773 & 82765353858 & 0.464001 \\
$n=$ 82 & 151200 & 96768892143 & 0.390741 \\
$n=$ 83 & 209223 & 113021049603 & 0.520528 \\
$n=$ 84 & 141696 & 131869111374 & 0.339544 \\
$n=$ 85 & 215136 & 153702652488 & 0.496764 \\
$n=$ 86 & 174636 & 178976363832 & 0.388745 \\
$n=$ 87 & 214200 & 208200528591 & 0.459865 \\
$n=$ 88 & 185760 & 241968535644 & 0.384795 \\
$n=$ 89 & 258390 & 280946616720 & 0.516650 \\
$n=$ 90 & 171072 & 325907179719 & 0.330309 \\
$n=$ 91 & 269136 & 377717284035 & 0.502002 \\
$n=$ 92 & 213840 & 437379463038 & 0.385463 \\
$n=$ 93 & 262080 & 506019158148 & 0.456722 \\
$n=$ 94 & 228528 & 584933843868 & 0.385161 \\
$n=$ 95 & 301320 & 675580163907 & 0.491329 \\
$n=$ 96 & 216576 & 779633241618 & 0.341784 \\
$n=$ 97 & 335160 & 898973004888 & 0.512081 \\
$n=$ 98 & 254016 & 1035756508248 & 0.375874 \\
$n=$ 99 & 314280 & 1192404057075 & 0.450542 \\
$n=$ 100 & 264600 & 1371686078070 & 0.367611 \\
$n=$ 101 & 378675 & 1576710505656 & 0.510014 \\
$n=$ 102 & 259200 & 1811027505411 & 0.338536 \\
$n=$ 103 & 401778 & 2078617830114 & 0.509029 \\
$n=$ 104 & 308448 & 2384020004256 & 0.379189 \\
$n=$ 105 & 355968 & 2732321142936 & 0.424746 \\
$n=$ 106 & 328536 & 3129309626520 & 0.380603 \\
$n=$ 107 & 450765 & 3581471185215 & 0.507146 \\
$n=$ 108 & 309096 & 4096172615898 & 0.337824 \\
$n=$ 109 & 476685 & 4681663620546 & 0.506245 \\
$n=$ 110 & 349920 & 5347299365100 & 0.361200 \\
$n=$ 111 & 447336 & 6103551249411 & 0.448927 \\
$n=$ 112 & 380160 & 6962273875326 & 0.371008 \\
$n=$ 113 & 531468 & 7936730093922 & 0.504520 \\
$n=$ 114 & 362880 & 9041912000331 & 0.335164 \\
$n=$ 115 & 536976 & 10294582596087 & 0.482668 \\
$n=$ 116 & 430920 & 11713663668234 & 0.377046 \\
$n=$ 117 & 521640 & 13320300598374 & 0.444400 \\
$n=$ 118 & 454140 & 15138327574623 & 0.376790 \\
$n=$ 119 & 606528 & 17194365434631 & 0.490192 \\
$n=$ 120 & 407808 & 19518380501214 & 0.321125 \\
$n=$ 121 & 647955 & 22143821814114 & 0.497235 \\
$n=$ 122 & 502200 & 25108295294676 & 0.375652 \\
$n=$ 123 & 609840 & 28453753086099 & 0.444743 \\
$n=$ 124 & 527040 & 32227301758050 & 0.374812 \\
$n=$ 125 & 691875 & 36481462023930 & 0.479913 \\
$n=$ 126 & 482112 & 41275139033940 & 0.326239 \\
$n=$ 127 & 756000 & 46673966702055 & 0.499171 \\
$n=$ 128 & 580608 & 52751476182174 & 0.374141 \\
$n=$ 129 & 704088 & 59589548978388 & 0.442882 \\
$n=$ 130 & 580608 & 67279817460768 & 0.356562 \\
$n=$ 131 & 830115 & 75924259717614 & 0.497810 \\
$n=$ 132 & 561600 & 85636879631886 & 0.328931 \\
$n=$ 133 & 848880 & 96544472830806 & 0.485686 \\
$n=$ 134 & 666468 & 108788647157883 & 0.372561 \\
$n=$ 135 & 775656 & 122526804400938 & 0.423714 \\
$n=$ 136 & 694656 & 137934560548434 & 0.370880 \\
$n=$ 137 & 950130 & 155207002394706 & 0.495886 \\
$n=$ 138 & 646272 & 174561588816615 & 0.329778 \\
$n=$ 139 & 992565 & 196239739104123 & 0.495274 \\
$n=$ 140 & 715392 & 220510316860680 & 0.349126 \\
$n=$ 141 & 920736 & 247671631144434 & 0.439535 \\
$n=$ 142 & 793800 & 278055607314156 & 0.370731 \\
$n=$ 143 & 1065960 & 312030305244999 & 0.487132 \\
$n=$ 144 & 736128 & 350004914182638 & 0.329218 \\
$n=$ 145 & 1081080 & 392432894620998 & 0.473237 \\
$n=$ 146 & 863136 & 439817967777015 & 0.369875 \\
$n=$ 147 & 1023120 & 492718025456790 & 0.429262 \\
$n=$ 148 & 898776 & 551752293927450 & 0.369259 \\
$n=$ 149 & 1223775 & 617606194024050 & 0.492409 \\
$n=$ 150 & 799200 & 691039911568743 & 0.314983 \\
$n=$ 151 & 1273950 & 772894401091422 & 0.491872 \\
$n=$ 152 & 972000 & 864101647122714 & 0.367701 \\
$n=$ 153 & 1174176 & 965692065073758 & 0.435260 \\
$n=$ 154 & 984960 & 1078806762045012 & 0.357832 \\
$n=$ 155 & 1321920 & 1204706649809355 & 0.470727 \\
$n=$ 156 & 931392 & 1344787094124012 & 0.325130 \\
$n=$ 157 & 1432665 & 1500589083803934 & 0.490326 \\
$n=$ 158 & 1095120 & 1673816743061772 & 0.367514 \\
$n=$ 159 & 1322568 & 1866350995369287 & 0.435268 \\
$n=$ 160 & 1092096 & 2080270461947106 & 0.352517 \\
$n=$ 161 & 1511136 & 2317868159571888 & 0.478471 \\
$n=$ 162 & 1049760 & 2581676436549903 & 0.326084 \\
$n=$ 163 & 1604043 & 2874487304423154 & 0.488870 \\
$n=$ 164 & 1224720 & 3199382197526328 & 0.366272 \\
$n=$ 165 & 1408320 & 3559756690115340 & 0.413341 \\
$n=$ 166 & 1270836 & 3959355961257543 & 0.366090 \\
$n=$ 167 & 1725570 & 4402304811032634 & 0.487945 \\
$n=$ 168 & 1147392 & 4893149912268492 & 0.318521 \\
$n=$ 169 & 1778049 & 5436896157324996 & 0.484627 \\
$n=$ 170 & 1306368 & 6039056997339408 & 0.349635 \\
$n=$ 171 & 1642680 & 6705698391579594 & 0.431752 \\
$n=$ 172 & 1413720 & 7443498705358386 & 0.364942 \\
$n=$ 173 & 1919133 & 8259801797584260 & 0.486620 \\
$n=$ 174 & 1300320 & 9162688266939000 & 0.323896 \\
$n=$ 175 & 1868400 & 10161039427496706 & 0.457236 \\
$n=$ 176 & 1503360 & 11264622040335510 & 0.361488 \\
$n=$ 177 & 1827000 & 12484165311019248 & 0.431691 \\
$n=$ 178 & 1568160 & 13831461547260159 & 0.364143 \\
$n=$ 179 & 2126655 & 15319458735881658 & 0.485365 \\
$n=$ 180 & 1384128 & 16962380066992416 & 0.310514 \\
$n=$ 181 & 2199015 & 18775835026110534 & 0.484961 \\
$n=$ 182 & 1632960 & 20776961284889271 & 0.354055 \\
$n=$ 183 & 2019960 & 22984557853851825 & 0.430621 \\
$n=$ 184 & 1729728 & 25419253353884082 & 0.362600 \\
$n=$ 185 & 2253096 & 28103665458294714 & 0.464481 \\
$n=$ 186 & 1589760 & 31062600353936955 & 0.322328 \\
$n=$ 187 & 2397600 & 34323243346612374 & 0.478146 \\
$n=$ 188 & 1848096 & 37915395208033464 & 0.362547 \\
$n=$ 189 & 2181168 & 41871699788260650 & 0.420943 \\
$n=$ 190 & 1827360 & 46227923813707419 & 0.346969 \\
$n=$ 191 & 2585520 & 51023228442343290 & 0.483042 \\
$n=$ 192 & 1751040 & 56300500326000918 & 0.321915 \\
$n=$ 193 & 2667888 & 62106675129029124 & 0.482677 \\
$n=$ 194 & 2032128 & 68493129084028803 & 0.361844 \\
$n=$ 195 & 2334528 & 75516064029414129 & 0.409153 \\
$n=$ 196 & 2053296 & 83236970116449534 & 0.354235 \\
$n=$ 197 & 2837835 & 91723083663501654 & 0.481965 \\
$n=$ 198 & 1905120 & 101047933645862742 & 0.318547 \\
$n=$ 199 & 2925450 & 111291885452415525 & 0.481617 \\
$n=$ 200 & 2138400 & 122542786024724076 & 0.346649 \\
$n=$ 201 & 2679336 & 134896609028552784 & 0.427714 \\
$n=$ 202 & 2295000 & 148458215846526687 & 0.360802 \\
$n=$ 203 & 3039120 & 163342120438969941 & 0.470573 \\
$n=$ 204 & 2094336 & 179673386486211618 & 0.319412 \\
$n=$ 205 & 3069360 & 197588533129562652 & 0.461117 \\
$n=$ 206 & 2434536 & 217236592093804140 & 0.360305 \\
$n=$ 207 & 2922480 & 238780179315327894 & 0.426116 \\
$n=$ 208 & 2491776 & 262396739637218088 & 0.357964 \\
$n=$ 209 & 3353400 & 288279813695479650 & 0.474681 \\
$n=$ 210 & 2156544 & 316640502633826371 & 0.300809 \\
$n=$ 211 & 3489255 & 347708964313511493 & 0.479637 \\
$n=$ 212 & 2653560 & 381736136035547706 & 0.359490 \\
$n=$ 213 & 3190320 & 418995500128643802 & 0.425991 \\
$n=$ 214 & 2730348 & 459785108688456483 & 0.359354 \\
$n=$ 215 & 3542616 & 504429665453630301 & 0.459617 \\
$n=$ 216 & 2496096 & 553282904139290952 & 0.319249 \\
$n=$ 217 & 3715200 & 606730041005201778 & 0.468465 \\
$n=$ 218 & 2886840 & 665190567041523327 & 0.358899 \\
$n=$ 219 & 3468528 & 729121134982345413 & 0.425185 \\
$n=$ 220 & 2825280 & 799018835131758714 & 0.341512 \\
$n=$ 221 & 3973536 & 875424591304180236 & 0.473653 \\
$n=$ 222 & 2708640 & 958927001865843087 & 0.318421 \\
$n=$ 223 & 4121208 & 1050166330936025529 & 0.477826 \\
$n=$ 224 & 3068928 & 1149839009723107200 & 0.350957 \\
$n=$ 225 & 3612600 & 1258702323871495356 & 0.407508 \\
$n=$ 226 & 3217536 & 1377579685178100327 & 0.358027 \\
$n=$ 227 & 4347675 & 1507366132665974724 & 0.477256 \\
$n=$ 228 & 2928960 & 1649034503019098316 & 0.317202 \\
$n=$ 229 & 4463955 & 1803641881580649222 & 0.476976 \\
$n=$ 230 & 3250368 & 1972336820806911186 & 0.342681 \\
$n=$ 231 & 3957120 & 2156366899923527676 & 0.411664 \\
$n=$ 232 & 3477600 & 2357087164124558250 & 0.357005 \\
$n=$ 233 & 4702698 & 2575968977502449760 & 0.476429 \\
$n=$ 234 & 3157056 & 2814609883916021835 & 0.315656 \\
$n=$ 235 & 4630176 & 3074743966840613277 & 0.456915 \\
$n=$ 236 & 3664440 & 3358253365177864734 & 0.356925 \\
$n=$ 237 & 4399200 & 3667180388444863818 & 0.422958 \\
$n=$ 238 & 3670272 & 4003740957196110207 & 0.348338 \\
$n=$ 239 & 5076540 & 4370338761889150869 & 0.475634 \\
$n=$ 240 & 3290112 & 4769580940928898822 & 0.304328 \\
$n=$ 241 & 5205420 & 5204294616553387620 & 0.475376 \\
$n=$ 242 & 3920400 & 5677545173275598931 & 0.353497 \\
$n=$ 243 & 4743603 & 6192655557546174948 & 0.422337 \\
$n=$ 244 & 4051080 & 6753227575531262298 & 0.356157 \\
$n=$ 245 & 5143824 & 7363164404629280352 & 0.446579 \\
$n=$ 246 & 3689280 & 8026695395935559994 & 0.316314 \\
$n=$ 247 & 5556600 & 8748402315038539272 & 0.470514 \\
$n=$ 248 & 4250880 & 9533248210691409798 & 0.355509 \\
$n=$ 249 & 5104008 & 10386607984040095500 & 0.421614 \\
$n=$ 250 & 4185000 & 11314301971302511335 & 0.341469 \\
$n=$ 251 & 5882625 & 12322631534338835217 & 0.474134 \\
$n=$ 252 & 3888000 & 13418418106425455256 & 0.309565 \\
$n=$ 253 & 5963760 & 14609044603394782002 & 0.469097 \\
$n=$ 254 & 4572288 & 15902500797767970771 & 0.355316 \\
$n=$ 255 & 5246208 & 17307431474208903984 & 0.402796 \\
$n=$ 256 & 4681728 & 18833189129072476098 & 0.355162
\end{longtable}

\end{footnotesize}

\newpage
\subsection{The table for the second program}

\begin{footnotesize}

\begin{longtable}{l|rrr|rrr|l}
$n$ & $\dfrac{\#\mathcal{A}_1(n)}{n!}$ & $\dfrac{\#\mathcal{B}_1(n)}{n!}$ & $\dfrac{\#\mathcal{A}_1(n)}{\#\mathcal{B}_1(n)}$ & $\dfrac{\#\mathcal{A}_2(n)}{n!}$ & $\dfrac{\#\mathcal{B}_2(n)}{n!}$ & $n\cdot\dfrac{\#\mathcal{A}_2(n)}{\#\mathcal{B}_2(n)}$ & $\;n$ \\
\hline\hline
$n=$ 3 & 1 & 1 & 1.00000 & 3 & 3 & 3.00000 & $\;n=$ 3 \\
$n=$ 4 & 4 & 4 & 1.00000 & 9 & 12 & 3.00000 & $\;n=$ 4 \\
$n=$ 5 & 10 & 10 & 1.00000 & 19 & 31 & 3.06452 & $\;n=$ 5 \\
$n=$ 6 & 18 & 20 & 0.900000 & 32 & 65 & 2.95385 & $\;n=$ 6 \\
$n=$ 7 & 35 & 35 & 1.00000 & 55 & 120 & 3.20833 & $\;n=$ 7 \\
$n=$ 8 & 48 & 56 & 0.857143 & 75 & 203 & 2.95567 & $\;n=$ 8 \\
$n=$ 9 & 81 & 84 & 0.964286 & 116 & 322 & 3.24224 & $\;n=$ 9 \\
$n=$ 10 & 100 & 120 & 0.833333 & 144 & 486 & 2.96296 & $\;n=$ 10 \\
$n=$ 11 & 165 & 165 & 1.00000 & 219 & 705 & 3.41702 & $\;n=$ 11 \\
$n=$ 12 & 168 & 220 & 0.763636 & 233 & 990 & 2.82424 & $\;n=$ 12 \\
$n=$ 13 & 286 & 286 & 1.00000 & 363 & 1353 & 3.48780 & $\;n=$ 13 \\
$n=$ 14 & 294 & 364 & 0.807692 & 384 & 1807 & 2.97510 & $\;n=$ 14 \\
$n=$ 15 & 420 & 455 & 0.923077 & 524 & 2366 & 3.32206 & $\;n=$ 15 \\
$n=$ 16 & 448 & 560 & 0.800000 & 567 & 3045 & 2.97931 & $\;n=$ 16 \\
$n=$ 17 & 680 & 680 & 1.00000 & 815 & 3860 & 3.58938 & $\;n=$ 17 \\
$n=$ 18 & 594 & 816 & 0.727941 & 746 & 4828 & 2.78128 & $\;n=$ 18 \\
$n=$ 19 & 969 & 969 & 1.00000 & 1139 & 5967 & 3.62678 & $\;n=$ 19 \\
$n=$ 20 & 880 & 1140 & 0.771930 & 1069 & 7296 & 2.93037 & $\;n=$ 20 \\
$n=$ 21 & 1218 & 1330 & 0.915789 & 1427 & 8835 & 3.39185 & $\;n=$ 21 \\
$n=$ 22 & 1210 & 1540 & 0.785714 & 1440 & 10605 & 2.98727 & $\;n=$ 22 \\
$n=$ 23 & 1771 & 1771 & 1.00000 & 2023 & 12628 & 3.68459 & $\;n=$ 23 \\
$n=$ 24 & 1440 & 2024 & 0.711462 & 1715 & 14927 & 2.75742 & $\;n=$ 24 \\
$n=$ 25 & 2250 & 2300 & 0.978261 & 2549 & 17526 & 3.63603 & $\;n=$ 25 \\
$n=$ 26 & 2028 & 2600 & 0.780000 & 2352 & 20450 & 2.99032 & $\;n=$ 26 \\
$n=$ 27 & 2673 & 2925 & 0.913846 & 3023 & 23725 & 3.44029 & $\;n=$ 27 \\
$n=$ 28 & 2520 & 3276 & 0.769231 & 2897 & 27378 & 2.96282 & $\;n=$ 28 \\
$n=$ 29 & 3654 & 3654 & 1.00000 & 4059 & 31437 & 3.74435 & $\;n=$ 29 \\
$n=$ 30 & 2760 & 4060 & 0.679803 & 3194 & 35931 & 2.66678 & $\;n=$ 30 \\
$n=$ 31 & 4495 & 4495 & 1.00000 & 4959 & 40890 & 3.75957 & $\;n=$ 31 \\
$n=$ 32 & 3840 & 4960 & 0.774194 & 4335 & 46345 & 2.99320 & $\;n=$ 32 \\
$n=$ 33 & 4950 & 5456 & 0.907258 & 5477 & 52328 & 3.45400 & $\;n=$ 33 \\
$n=$ 34 & 4624 & 5984 & 0.772727 & 5184 & 58872 & 2.99389 & $\;n=$ 34 \\
$n=$ 35 & 6300 & 6545 & 0.962567 & 6894 & 66011 & 3.65530 & $\;n=$ 35 \\
$n=$ 36 & 4968 & 7140 & 0.695798 & 5597 & 73780 & 2.73098 & $\;n=$ 36 \\
$n=$ 37 & 7770 & 7770 & 1.00000 & 8435 & 82215 & 3.79608 & $\;n=$ 37 \\
$n=$ 38 & 6498 & 8436 & 0.770270 & 7200 & 91353 & 2.99498 & $\;n=$ 38 \\
$n=$ 39 & 8268 & 9139 & 0.904694 & 9008 & 101232 & 3.47037 & $\;n=$ 39 \\
$n=$ 40 & 7360 & 9880 & 0.744939 & 8139 & 111891 & 2.90962 & $\;n=$ 40 \\
$n=$ 41 & 10660 & 10660 & 1.00000 & 11479 & 123370 & 3.81486 & $\;n=$ 41 \\
$n=$ 42 & 7812 & 11480 & 0.680488 & 8672 & 135710 & 2.68384 & $\;n=$ 42 \\
$n=$ 43 & 12341 & 12341 & 1.00000 & 13243 & 148953 & 3.82301 & $\;n=$ 43 \\
$n=$ 44 & 10120 & 13244 & 0.764120 & 11065 & 163142 & 2.98427 & $\;n=$ 44 \\
$n=$ 45 & 12420 & 14190 & 0.875264 & 13409 & 178321 & 3.38381 & $\;n=$ 45 \\
$n=$ 46 & 11638 & 15180 & 0.766667 & 12672 & 194535 & 2.99644 & $\;n=$ 46 \\
$n=$ 47 & 16215 & 16215 & 1.00000 & 17295 & 211830 & 3.83735 & $\;n=$ 47 \\
$n=$ 48 & 11904 & 17296 & 0.688252 & 13031 & 230253 & 2.71652 & $\;n=$ 48 \\
$n=$ 49 & 18179 & 18424 & 0.986702 & 19354 & 249852 & 3.79563 & $\;n=$ 49 \\
$n=$ 50 & 14500 & 19600 & 0.739796 & 15724 & 270676 & 2.90458 & $\;n=$ 50 \\
$n=$ 51 & 18768 & 20825 & 0.901224 & 20042 & 292775 & 3.49122 & $\;n=$ 51 \\
$n=$ 52 & 16848 & 22100 & 0.762353 & 18173 & 316200 & 2.98860 & $\;n=$ 52 \\
$n=$ 53 & 23426 & 23426 & 1.00000 & 24803 & 341003 & 3.85498 & $\;n=$ 53 \\
$n=$ 54 & 17010 & 24804 & 0.685776 & 18440 & 367237 & 2.71149 & $\;n=$ 54 \\
$n=$ 55 & 25300 & 26235 & 0.964361 & 26784 & 394956 & 3.72983 & $\;n=$ 55 \\
$n=$ 56 & 20832 & 27720 & 0.751515 & 22371 & 424215 & 2.95316 & $\;n=$ 56 \\
$n=$ 57 & 26334 & 29260 & 0.900000 & 27929 & 455070 & 3.49826 & $\;n=$ 57 \\
$n=$ 58 & 23548 & 30856 & 0.763158 & 25200 & 487578 & 2.99767 & $\;n=$ 58 \\
$n=$ 59 & 32509 & 32509 & 1.00000 & 34219 & 521797 & 3.86917 & $\;n=$ 59 \\
$n=$ 60 & 22560 & 34220 & 0.659264 & 24329 & 557786 & 2.61703 & $\;n=$ 60 \\
$n=$ 61 & 35990 & 35990 & 1.00000 & 37819 & 595605 & 3.87330 & $\;n=$ 61 \\
$n=$ 62 & 28830 & 37820 & 0.762295 & 30720 & 635315 & 2.99795 & $\;n=$ 62 \\
$n=$ 63 & 35154 & 39711 & 0.885246 & 37106 & 676978 & 3.45311 & $\;n=$ 63 \\
$n=$ 64 & 31744 & 41664 & 0.761905 & 33759 & 720657 & 2.99806 & $\;n=$ 64 \\
$n=$ 65 & 42120 & 43680 & 0.964286 & 44199 & 766416 & 3.74853 & $\;n=$ 65 \\
$n=$ 66 & 31020 & 45760 & 0.677885 & 33164 & 814320 & 2.68792 & $\;n=$ 66 \\
$n=$ 67 & 47905 & 47905 & 1.00000 & 50115 & 864435 & 3.88428 & $\;n=$ 67 \\
$n=$ 68 & 38080 & 50116 & 0.759837 & 40357 & 916828 & 2.99323 & $\;n=$ 68 \\
$n=$ 69 & 47058 & 52394 & 0.898156 & 49403 & 971567 & 3.50857 & $\;n=$ 69 \\
$n=$ 70 & 39480 & 54740 & 0.721228 & 41894 & 1028721 & 2.85070 & $\;n=$ 70 \\
$n=$ 71 & 57155 & 57155 & 1.00000 & 59639 & 1088360 & 3.89060 & $\;n=$ 71 \\
$n=$ 72 & 40608 & 59640 & 0.680885 & 43163 & 1150555 & 2.70108 & $\;n=$ 72 \\
$n=$ 73 & 62196 & 62196 & 1.00000 & 64823 & 1215378 & 3.89350 & $\;n=$ 73 \\
$n=$ 74 & 49284 & 64824 & 0.760274 & 51984 & 1282902 & 2.99853 & $\;n=$ 74 \\
$n=$ 75 & 58500 & 67525 & 0.866346 & 61274 & 1353201 & 3.39606 & $\;n=$ 75 \\
$n=$ 76 & 53352 & 70300 & 0.758919 & 56201 & 1426350 & 2.99455 & $\;n=$ 76 \\
$n=$ 77 & 71610 & 73150 & 0.978947 & 74535 & 1502425 & 3.81995 & $\;n=$ 77 \\
$n=$ 78 & 51480 & 76076 & 0.676692 & 54482 & 1581503 & 2.68706 & $\;n=$ 78 \\
$n=$ 79 & 79079 & 79079 & 1.00000 & 82159 & 1663662 & 3.90137 & $\;n=$ 79 \\
$n=$ 80 & 60160 & 82160 & 0.732230 & 63319 & 1748981 & 2.89627 & $\;n=$ 80 \\
$n=$ 81 & 76545 & 85320 & 0.897152 & 79784 & 1837540 & 3.51693 & $\;n=$ 81 \\
$n=$ 82 & 67240 & 88560 & 0.759259 & 70560 & 1929420 & 2.99879 & $\;n=$ 82 \\
$n=$ 83 & 91881 & 91881 & 1.00000 & 95283 & 2024703 & 3.90600 & $\;n=$ 83 \\
$n=$ 84 & 63504 & 95284 & 0.666471 & 66989 & 2123472 & 2.64994 & $\;n=$ 84 \\
$n=$ 85 & 95200 & 98770 & 0.963855 & 98769 & 2225811 & 3.77182 & $\;n=$ 85 \\
$n=$ 86 & 77658 & 102340 & 0.758824 & 81312 & 2331805 & 2.99889 & $\;n=$ 86 \\
$n=$ 87 & 95004 & 105995 & 0.896306 & 98744 & 2441540 & 3.51857 & $\;n=$ 87 \\
$n=$ 88 & 82720 & 109736 & 0.753809 & 86547 & 2555103 & 2.98076 & $\;n=$ 88 \\
$n=$ 89 & 113564 & 113564 & 1.00000 & 117479 & 2672582 & 3.91218 & $\;n=$ 89 \\
$n=$ 90 & 76680 & 117480 & 0.652707 & 80684 & 2794066 & 2.59892 & $\;n=$ 90 \\
$n=$ 91 & 119028 & 121485 & 0.979775 & 123122 & 2919645 & 3.83749 & $\;n=$ 91 \\
$n=$ 92 & 95128 & 125580 & 0.757509 & 99313 & 3049410 & 2.99625 & $\;n=$ 92 \\
$n=$ 93 & 116250 & 129766 & 0.895843 & 120527 & 3183453 & 3.52102 & $\;n=$ 93 \\
$n=$ 94 & 101614 & 134044 & 0.758065 & 105984 & 3321867 & 2.99907 & $\;n=$ 94 \\
$n=$ 95 & 133380 & 138415 & 0.963624 & 137844 & 3464746 & 3.77955 & $\;n=$ 95 \\
$n=$ 96 & 96768 & 142880 & 0.677268 & 101327 & 3612185 & 2.69294 & $\;n=$ 96 \\
$n=$ 97 & 147440 & 147440 & 1.00000 & 152095 & 3764280 & 3.91927 & $\;n=$ 97 \\
$n=$ 98 & 113190 & 152096 & 0.744201 & 117942 & 3921128 & 2.94770 & $\;n=$ 98 \\
$n=$ 99 & 139590 & 156849 & 0.889964 & 144440 & 4082827 & 3.50237 & $\;n=$ 99 \\
$n=$ 100 & 118000 & 161700 & 0.729746 & 122949 & 4249476 & 2.89327 & $\;n=$ 100 \\
$n=$ 101 & 166650 & 166650 & 1.00000 & 171699 & 4421175 & 3.92240 & $\;n=$ 101 \\
$n=$ 102 & 115872 & 171700 & 0.674851 & 121022 & 4598025 & 2.68468 & $\;n=$ 102 \\
$n=$ 103 & 176851 & 176851 & 1.00000 & 182103 & 4780128 & 3.92387 & $\;n=$ 103 \\
$n=$ 104 & 137280 & 182104 & 0.753855 & 142635 & 4967587 & 2.98617 & $\;n=$ 104 \\
$n=$ 105 & 158760 & 187460 & 0.846901 & 164219 & 5160506 & 3.34134 & $\;n=$ 105 \\
$n=$ 106 & 146068 & 192920 & 0.757143 & 151632 & 5358990 & 2.99926 & $\;n=$ 106 \\
$n=$ 107 & 198485 & 198485 & 1.00000 & 204155 & 5563145 & 3.92666 & $\;n=$ 107 \\
$n=$ 108 & 138024 & 204156 & 0.676071 & 143801 & 5773078 & 2.69016 & $\;n=$ 108 \\
$n=$ 109 & 209934 & 209934 & 1.00000 & 215819 & 5988897 & 3.92798 & $\;n=$ 109 \\
$n=$ 110 & 156200 & 215820 & 0.723751 & 162194 & 6210711 & 2.87267 & $\;n=$ 110 \\
$n=$ 111 & 198468 & 221815 & 0.894746 & 204572 & 6438630 & 3.52676 & $\;n=$ 111 \\
$n=$ 112 & 169344 & 227920 & 0.742998 & 175559 & 6672765 & 2.94670 & $\;n=$ 112 \\
$n=$ 113 & 234136 & 234136 & 1.00000 & 240463 & 6913228 & 3.93048 & $\;n=$ 113 \\
$n=$ 114 & 162108 & 240464 & 0.674147 & 168548 & 7160132 & 2.68354 & $\;n=$ 114 \\
$n=$ 115 & 237820 & 246905 & 0.963204 & 244374 & 7413591 & 3.79074 & $\;n=$ 115 \\
$n=$ 116 & 191632 & 253460 & 0.756064 & 198301 & 7673720 & 2.99762 & $\;n=$ 116 \\
$n=$ 117 & 231660 & 260130 & 0.890555 & 238445 & 7940635 & 3.51333 & $\;n=$ 117 \\
$n=$ 118 & 201898 & 266916 & 0.756410 & 208800 & 8214453 & 2.99940 & $\;n=$ 118 \\
$n=$ 119 & 268464 & 273819 & 0.980443 & 275484 & 8495292 & 3.85891 & $\;n=$ 119 \\
$n=$ 120 & 182400 & 280840 & 0.649480 & 189539 & 8783271 & 2.58955 & $\;n=$ 120 \\
$n=$ 121 & 286165 & 287980 & 0.993697 & 293424 & 9078510 & 3.91081 & $\;n=$ 121 \\
$n=$ 122 & 223260 & 295240 & 0.756198 & 230640 & 9381130 & 2.99943 & $\;n=$ 122 \\
$n=$ 123 & 270600 & 302621 & 0.894188 & 278102 & 9691253 & 3.52963 & $\;n=$ 123 \\
$n=$ 124 & 234360 & 310124 & 0.755698 & 241985 & 10009002 & 2.99792 & $\;n=$ 124 \\
$n=$ 125 & 306250 & 317750 & 0.963808 & 313999 & 10334501 & 3.79795 & $\;n=$ 125 \\
$n=$ 126 & 215460 & 325500 & 0.661936 & 223334 & 10667875 & 2.63783 & $\;n=$ 126 \\
$n=$ 127 & 333375 & 333375 & 1.00000 & 341375 & 11009250 & 3.93802 & $\;n=$ 127 \\
$n=$ 128 & 258048 & 341376 & 0.755906 & 266175 & 11358753 & 2.99948 & $\;n=$ 128 \\
$n=$ 129 & 312438 & 349504 & 0.893947 & 320693 & 11716512 & 3.53086 & $\;n=$ 129 \\
$n=$ 130 & 258960 & 357760 & 0.723837 & 267344 & 12082656 & 2.87641 & $\;n=$ 130 \\
$n=$ 131 & 366145 & 366145 & 1.00000 & 374659 & 12457315 & 3.93988 & $\;n=$ 131 \\
$n=$ 132 & 250800 & 374660 & 0.669407 & 259445 & 12840620 & 2.66706 & $\;n=$ 132 \\
$n=$ 133 & 375858 & 383306 & 0.980569 & 384635 & 13232703 & 3.86591 & $\;n=$ 133 \\
$n=$ 134 & 296274 & 392084 & 0.755639 & 305184 & 13633697 & 2.99953 & $\;n=$ 134 \\
$n=$ 135 & 345060 & 400995 & 0.860509 & 354104 & 14043736 & 3.40394 & $\;n=$ 135 \\
$n=$ 136 & 308992 & 410040 & 0.753565 & 318171 & 14462955 & 2.99187 & $\;n=$ 136 \\
$n=$ 137 & 419220 & 419220 & 1.00000 & 428535 & 14891490 & 3.94247 & $\;n=$ 137 \\
$n=$ 138 & 288420 & 428536 & 0.673036 & 297872 & 15329478 & 2.68152 & $\;n=$ 138 \\
$n=$ 139 & 437989 & 437989 & 1.00000 & 447579 & 15777057 & 3.94329 & $\;n=$ 139 \\
$n=$ 140 & 319200 & 447580 & 0.713169 & 328929 & 16234366 & 2.83658 & $\;n=$ 140 \\
$n=$ 141 & 408618 & 457310 & 0.893525 & 418487 & 16701545 & 3.53301 & $\;n=$ 141 \\
$n=$ 142 & 352870 & 467180 & 0.755319 & 362880 & 17178735 & 2.99958 & $\;n=$ 142 \\
$n=$ 143 & 471900 & 477191 & 0.988912 & 482052 & 17666078 & 3.90202 & $\;n=$ 143 \\
$n=$ 144 & 328320 & 487344 & 0.673693 & 338615 & 18163717 & 2.68450 & $\;n=$ 144 \\
$n=$ 145 & 479080 & 497640 & 0.962704 & 489519 & 18671796 & 3.80147 & $\;n=$ 145 \\
$n=$ 146 & 383688 & 508080 & 0.755172 & 394272 & 19190460 & 2.99960 & $\;n=$ 146 \\
$n=$ 147 & 454818 & 518665 & 0.876901 & 465548 & 19719855 & 3.47039 & $\;n=$ 147 \\
$n=$ 148 & 399600 & 529396 & 0.754822 & 410477 & 20260128 & 2.99853 & $\;n=$ 148 \\
$n=$ 149 & 540274 & 540274 & 1.00000 & 551299 & 20811427 & 3.94704 & $\;n=$ 149 \\
$n=$ 150 & 357000 & 551300 & 0.647560 & 368174 & 21373901 & 2.58381 & $\;n=$ 150 \\
$n=$ 151 & 562475 & 562475 & 1.00000 & 573799 & 21947700 & 3.94773 & $\;n=$ 151 \\
$n=$ 152 & 432288 & 573800 & 0.753377 & 443763 & 22532975 & 2.99348 & $\;n=$ 152 \\
$n=$ 153 & 521424 & 585276 & 0.890903 & 533051 & 23129878 & 3.52604 & $\;n=$ 153 \\
$n=$ 154 & 438900 & 596904 & 0.735294 & 450680 & 23738562 & 2.92371 & $\;n=$ 154 \\
$n=$ 155 & 585900 & 608685 & 0.962567 & 597834 & 24359181 & 3.80408 & $\;n=$ 155 \\
$n=$ 156 & 415584 & 620620 & 0.669627 & 427673 & 24991890 & 2.66955 & $\;n=$ 156 \\
$n=$ 157 & 632710 & 632710 & 1.00000 & 644955 & 25636845 & 3.94970 & $\;n=$ 157 \\
$n=$ 158 & 486798 & 644956 & 0.754777 & 499200 & 26294203 & 2.99966 & $\;n=$ 158 \\
$n=$ 159 & 587028 & 657359 & 0.893010 & 599588 & 26964122 & 3.53561 & $\;n=$ 159 \\
$n=$ 160 & 486400 & 669920 & 0.726057 & 499119 & 27646761 & 2.88855 & $\;n=$ 160 \\
$n=$ 161 & 669438 & 682640 & 0.980660 & 682317 & 28342280 & 3.87594 & $\;n=$ 161 \\
$n=$ 162 & 468018 & 695520 & 0.672904 & 481058 & 29050840 & 2.68259 & $\;n=$ 162 \\
$n=$ 163 & 708561 & 708561 & 1.00000 & 721763 & 29772603 & 3.95153 & $\;n=$ 163 \\
$n=$ 164 & 544480 & 721764 & 0.754374 & 557845 & 30507732 & 2.99880 & $\;n=$ 164 \\
$n=$ 165 & 627000 & 735130 & 0.852910 & 640529 & 31256391 & 3.38130 & $\;n=$ 165 \\
$n=$ 166 & 564898 & 748660 & 0.754545 & 578592 & 32018745 & 2.99969 & $\;n=$ 166 \\
$n=$ 167 & 762355 & 762355 & 1.00000 & 776215 & 32794960 & 3.95268 & $\;n=$ 167 \\
$n=$ 168 & 512064 & 776216 & 0.659693 & 526091 & 33585203 & 2.63161 & $\;n=$ 168 \\
$n=$ 169 & 786526 & 790244 & 0.995295 & 800721 & 34389642 & 3.93496 & $\;n=$ 169 \\
$n=$ 170 & 582080 & 804440 & 0.723584 & 596444 & 35208446 & 2.87986 & $\;n=$ 170 \\
$n=$ 171 & 729486 & 818805 & 0.890915 & 744020 & 36041785 & 3.53000 & $\;n=$ 171 \\
$n=$ 172 & 628488 & 833340 & 0.754180 & 643193 & 36889830 & 2.99891 & $\;n=$ 172 \\
$n=$ 173 & 848046 & 848046 & 1.00000 & 862923 & 37752753 & 3.95430 & $\;n=$ 173 \\
$n=$ 174 & 579768 & 862924 & 0.671865 & 594818 & 38630727 & 2.67917 & $\;n=$ 174 \\
$n=$ 175 & 829500 & 877975 & 0.944788 & 844724 & 39523926 & 3.74018 & $\;n=$ 175 \\
$n=$ 176 & 668800 & 893200 & 0.748768 & 684199 & 40432525 & 2.97827 & $\;n=$ 176 \\
$n=$ 177 & 811014 & 908600 & 0.892597 & 826589 & 41356700 & 3.53767 & $\;n=$ 177 \\
$n=$ 178 & 697048 & 924176 & 0.754237 & 712800 & 42296628 & 2.99973 & $\;n=$ 178 \\
$n=$ 179 & 939929 & 939929 & 1.00000 & 955859 & 43252487 & 3.95581 & $\;n=$ 179 \\
$n=$ 180 & 617760 & 955860 & 0.646287 & 633869 & 44224456 & 2.57994 & $\;n=$ 180 \\
$n=$ 181 & 971970 & 971970 & 1.00000 & 988259 & 45212715 & 3.95630 & $\;n=$ 181 \\
$n=$ 182 & 727272 & 988260 & 0.735912 & 743742 & 46217445 & 2.92879 & $\;n=$ 182 \\
$n=$ 183 & 896700 & 1004731 & 0.892478 & 913352 & 47238828 & 3.53826 & $\;n=$ 183 \\
$n=$ 184 & 769120 & 1021384 & 0.753017 & 785955 & 48277047 & 2.99554 & $\;n=$ 184 \\
$n=$ 185 & 999000 & 1038220 & 0.962224 & 1016019 & 49332286 & 3.81015 & $\;n=$ 185 \\
$n=$ 186 & 708660 & 1055240 & 0.671563 & 725864 & 50404730 & 2.67853 & $\;n=$ 186 \\
$n=$ 187 & 1062160 & 1072445 & 0.990410 & 1079550 & 51494565 & 3.92033 & $\;n=$ 187 \\
$n=$ 188 & 821560 & 1089836 & 0.753838 & 839137 & 52601978 & 2.99908 & $\;n=$ 188 \\
$n=$ 189 & 969570 & 1107414 & 0.875526 & 987335 & 53727157 & 3.47322 & $\;n=$ 189 \\
$n=$ 190 & 813960 & 1125180 & 0.723404 & 831914 & 54870291 & 2.88068 & $\;n=$ 190 \\
$n=$ 191 & 1143135 & 1143135 & 1.00000 & 1161279 & 56031570 & 3.95856 & $\;n=$ 191 \\
$n=$ 192 & 780288 & 1161280 & 0.671921 & 798623 & 57211185 & 2.68017 & $\;n=$ 192 \\
$n=$ 193 & 1179616 & 1179616 & 1.00000 & 1198143 & 58409328 & 3.95898 & $\;n=$ 193 \\
$n=$ 194 & 903264 & 1198144 & 0.753886 & 921984 & 59626192 & 2.99977 & $\;n=$ 194 \\
$n=$ 195 & 1038960 & 1216865 & 0.853801 & 1057874 & 60861971 & 3.38940 & $\;n=$ 195 \\
$n=$ 196 & 913752 & 1235780 & 0.739413 & 932861 & 62116860 & 2.94350 & $\;n=$ 196 \\
$n=$ 197 & 1254890 & 1254890 & 1.00000 & 1274195 & 63391055 & 3.95981 & $\;n=$ 197 \\
$n=$ 198 & 849420 & 1274196 & 0.666632 & 868922 & 64684753 & 2.65977 & $\;n=$ 198 \\
$n=$ 199 & 1293699 & 1293699 & 1.00000 & 1313399 & 65998152 & 3.96021 & $\;n=$ 199 \\
$n=$ 200 & 952000 & 1313400 & 0.724836 & 971899 & 67331451 & 2.88691 & $\;n=$ 200 \\
$n=$ 201 & 1189518 & 1333300 & 0.892161 & 1209617 & 68684850 & 3.53983 & $\;n=$ 201 \\
$n=$ 202 & 1020100 & 1353400 & 0.753731 & 1040400 & 70058550 & 2.99979 & $\;n=$ 202 \\
$n=$ 203 & 1347108 & 1373701 & 0.980641 & 1367610 & 71452753 & 3.88543 & $\;n=$ 203 \\
$n=$ 204 & 933504 & 1394204 & 0.669561 & 954209 & 72867662 & 2.67140 & $\;n=$ 204 \\
$n=$ 205 & 1361200 & 1414910 & 0.962040 & 1382109 & 74303481 & 3.81318 & $\;n=$ 205 \\
$n=$ 206 & 1082118 & 1435820 & 0.753659 & 1103232 & 75760415 & 2.99980 & $\;n=$ 206 \\
$n=$ 207 & 1297890 & 1456935 & 0.890836 & 1319210 & 77238670 & 3.53549 & $\;n=$ 207 \\
$n=$ 208 & 1108224 & 1478256 & 0.749683 & 1129751 & 78738453 & 2.98441 & $\;n=$ 208 \\
$n=$ 209 & 1485990 & 1499784 & 0.990803 & 1507725 & 80259972 & 3.92617 & $\;n=$ 209 \\
$n=$ 210 & 962640 & 1521520 & 0.632683 & 984584 & 81803436 & 2.52755 & $\;n=$ 210 \\
$n=$ 211 & 1543465 & 1543465 & 1.00000 & 1565619 & 83369055 & 3.96245 & $\;n=$ 211 \\
$n=$ 212 & 1179568 & 1565620 & 0.753419 & 1201933 & 84957040 & 2.99928 & $\;n=$ 212 \\
$n=$ 213 & 1416450 & 1587986 & 0.891979 & 1439027 & 86567603 & 3.54073 & $\;n=$ 213 \\
$n=$ 214 & 1213594 & 1610564 & 0.753521 & 1236384 & 88200957 & 2.99981 & $\;n=$ 214 \\
$n=$ 215 & 1571220 & 1633355 & 0.961959 & 1594224 & 89857316 & 3.81447 & $\;n=$ 215 \\
$n=$ 216 & 1111968 & 1656360 & 0.671332 & 1135187 & 91536895 & 2.67871 & $\;n=$ 216 \\
$n=$ 217 & 1647030 & 1679580 & 0.980620 & 1670465 & 93239910 & 3.88772 & $\;n=$ 217 \\
$n=$ 218 & 1283148 & 1703016 & 0.753456 & 1306800 & 94966578 & 2.99982 & $\;n=$ 218 \\
$n=$ 219 & 1540008 & 1726669 & 0.891895 & 1563878 & 96717117 & 3.54114 & $\;n=$ 219 \\
$n=$ 220 & 1258400 & 1750540 & 0.718864 & 1282489 & 98491746 & 2.86468 & $\;n=$ 220 \\
$n=$ 221 & 1760928 & 1774630 & 0.992279 & 1785237 & 100290685 & 3.93394 & $\;n=$ 221 \\
$n=$ 222 & 1206792 & 1798940 & 0.670835 & 1231322 & 102114155 & 2.67694 & $\;n=$ 222 \\
$n=$ 223 & 1823471 & 1823471 & 1.00000 & 1848223 & 103962378 & 3.96445 & $\;n=$ 223 \\
$n=$ 224 & 1365504 & 1848224 & 0.738820 & 1390479 & 105835577 & 2.94294 & $\;n=$ 224 \\
$n=$ 225 & 1606500 & 1873200 & 0.857623 & 1631699 & 107733976 & 3.40777 & $\;n=$ 225 \\
$n=$ 226 & 1430128 & 1898400 & 0.753333 & 1455552 & 109657800 & 2.99983 & $\;n=$ 226 \\
$n=$ 227 & 1923825 & 1923825 & 1.00000 & 1949475 & 111607275 & 3.96507 & $\;n=$ 227 \\
$n=$ 228 & 1305072 & 1949476 & 0.669448 & 1330949 & 113582628 & 2.67168 & $\;n=$ 228 \\
$n=$ 229 & 1975354 & 1975354 & 1.00000 & 2001459 & 115584087 & 3.96537 & $\;n=$ 229 \\
$n=$ 230 & 1447160 & 2001460 & 0.723052 & 1473494 & 117611881 & 2.88154 & $\;n=$ 230 \\
$n=$ 231 & 1760220 & 2027795 & 0.868046 & 1786784 & 119666240 & 3.44915 & $\;n=$ 231 \\
$n=$ 232 & 1546048 & 2054360 & 0.752569 & 1572843 & 121747395 & 2.99719 & $\;n=$ 232 \\
$n=$ 233 & 2081156 & 2081156 & 1.00000 & 2108183 & 123855578 & 3.96596 & $\;n=$ 233 \\
$n=$ 234 & 1406808 & 2108184 & 0.667308 & 1434068 & 125991022 & 2.66346 & $\;n=$ 234 \\
$n=$ 235 & 2053900 & 2135445 & 0.961814 & 2081394 & 128153961 & 3.81672 & $\;n=$ 235 \\
$n=$ 236 & 1628872 & 2162940 & 0.753082 & 1656601 & 130344630 & 2.99942 & $\;n=$ 236 \\
$n=$ 237 & 1953354 & 2190670 & 0.891670 & 1981319 & 132563265 & 3.54225 & $\;n=$ 237 \\
$n=$ 238 & 1633632 & 2218636 & 0.736323 & 1661834 & 134810103 & 2.93388 & $\;n=$ 238 \\
$n=$ 239 & 2246839 & 2246839 & 1.00000 & 2275279 & 137085382 & 3.96681 & $\;n=$ 239 \\
$n=$ 240 & 1466880 & 2275280 & 0.644703 & 1495559 & 139389341 & 2.57505 & $\;n=$ 240 \\
$n=$ 241 & 2303960 & 2303960 & 1.00000 & 2332879 & 141722220 & 3.96708 & $\;n=$ 241 \\
$n=$ 242 & 1743610 & 2332880 & 0.747407 & 1772770 & 144084260 & 2.97750 & $\;n=$ 242 \\
$n=$ 243 & 2106081 & 2362041 & 0.891636 & 2135483 & 146475703 & 3.54272 & $\;n=$ 243 \\
$n=$ 244 & 1800720 & 2391444 & 0.752984 & 1830365 & 148896792 & 2.99945 & $\;n=$ 244 \\
$n=$ 245 & 2284380 & 2421090 & 0.943534 & 2314269 & 151347771 & 3.74631 & $\;n=$ 245 \\
$n=$ 246 & 1643280 & 2450980 & 0.670458 & 1673414 & 153828885 & 2.67609 & $\;n=$ 246 \\
$n=$ 247 & 2463084 & 2481115 & 0.992733 & 2493464 & 156340380 & 3.93939 & $\;n=$ 247 \\
$n=$ 248 & 1889760 & 2511496 & 0.752444 & 1920387 & 158882503 & 2.99754 & $\;n=$ 248 \\
$n=$ 249 & 2266398 & 2542124 & 0.891537 & 2297273 & 161455502 & 3.54290 & $\;n=$ 249 \\
$n=$ 250 & 1862500 & 2573000 & 0.723863 & 1893624 & 164059626 & 2.88557 & $\;n=$ 250 \\
$n=$ 251 & 2604125 & 2604125 & 1.00000 & 2635499 & 166695125 & 3.96838 & $\;n=$ 251 \\
$n=$ 252 & 1732752 & 2635500 & 0.657466 & 1764377 & 169362250 & 2.62528 & $\;n=$ 252 \\
$n=$ 253 & 2643850 & 2667126 & 0.991273 & 2675727 & 172061253 & 3.93441 & $\;n=$ 253 \\
$n=$ 254 & 2032254 & 2699004 & 0.752964 & 2064384 & 174792387 & 2.99986 & $\;n=$ 254 \\
$n=$ 255 & 2333760 & 2731135 & 0.854502 & 2366144 & 177555906 & 3.39818 & $\;n=$ 255 \\
$n=$ 256 & 2080768 & 2763520 & 0.752941 & 2113407 & 180352065 & 2.99987 & $\;n=$ 256
\end{longtable}

\end{footnotesize}

\newpage
\addcontentsline{toc}{section}{References}
\bibliographystyle{siam}

\begin{thebibliography}{} 
\bibitem{apostol}
Tom~M. Apostol, \emph{Introduction to analytic number theory}, Undergraduate
  Texts in Mathematics, Springer-Verlag, 1976.

\bibitem{dixon1}
John~D. Dixon, \emph{The probability of generating the symmetric group},
  Mathematische Zeitschrift \textbf{110} (1969), 199--205.

\bibitem{eskin-masur-schmoll}
Alex Eskin, Howard Masur, and Martin Schmoll, \emph{Billiards in rectangles
  with barriers}, Duke Mathematical Journal \textbf{118} (2003), no.~3,
  427--463.

\bibitem{eskin-okounkov}
Alex Eskin and Andrei Okounkov, \emph{Asymptotics of numbers of branched
  coverings of a torus and volumes of moduli spaces of holomorphic
  differentials}, Inventiones Mathematicae \textbf{145} (2001), 59--103.

\bibitem{hardy-ramanudjan}
Godfrey~H. Hardy and Srinivasa Ramanujan, \emph{Asymptotic formulae in
  combinatory analysis}, Proceedings of the London Mathematical Society
  \textbf{s2-17} (1918), no.~1, 75--115.

\bibitem{hardy-wright79}
Godfrey~H. Hardy and Edward~M. Wright, \emph{An introduction to the theory of
  numbers}, Oxford Press, USA, 1979.

\bibitem{isaacs-z}
I.~M. Isaacs and Thilo Zieschang, \emph{Generating symmetric groups}, The
  American Mathematical Monthly \textbf{102} (1995), no.~8, 734--739.

\bibitem{jones}
Gareth~A.\ Jones, \emph{Enumeration of homomorphisms and surface-coverings},
  Quarterly Journal of Mathematics \textbf{46} (1995), no.~4, 485--507.

\bibitem{jordan73}
Camille Jordan, \emph{Sur la limite de transitivit\'{e} des groupes non
  altern\'{e}s}, Bulletin de la Soci\'{e}t\'{e} Math\'{e}matique de France
  \textbf{1} (1873), 40--71.

\bibitem{knopp}
Konrad Knopp, \emph{Theory and application of infinite series}, Blackie \& Son
  Limited, London and Glasgow, 1954.

\bibitem{lelievre-royer}
Samuel Leli\`{e}vre and Emmanuel Royer, \emph{Orbitwise countings in
  $\mathcal{H}(2)$ and quasimodular forms}, International Mathematics Research
  Notices \textbf{2006} (2006), 30 pp., doi:10.1155/IMRN/2006/42151.

\bibitem{netto}
Eugen Netto, \emph{The theory of substitutions and its applications to
  algebra}, Ann Arbor, Michigan, 1892.

\bibitem{ramanujan}
Srinivasa Ramanujan, \emph{On certain arithmetical functions}, Collected papers
  of Srinivasa Ramanujan (Godfrey~H. Hardy, P.~V.~Seshu Aiyar, and Bertram~M.
  Wilson, eds.), The University Press, Cambridge, 1927, pp.~136--162.

\bibitem{thurston}
William~P. Thurston, \emph{On the geometry and dynamics of diffeomorphisms of
  surfaces}, Bulletin of the American Mathematical Society \textbf{19} (1988),
  417--431, (this paper was actually written around 1976).

\bibitem{veech1}
William~A. Veech, \emph{Teichm\"{u}ller curves in moduli space, {E}isenstein
  series and an application to triangular billiards}, Inventiones Mathematicae
  \textbf{97} (1989), no.~3, 553--583.

\bibitem{wielandt}
Helmut Wielandt, \emph{Finite permutation groups}, Academic Press, New York,
  1964.

\bibitem{zagier}
Don Zagier, \emph{Applications of the representation theory of finite groups},
  Appendix to S. Lando and A. Zvonkin, Graphs on Surfaces and Their
  Applications, Encyclopaedia of Mathematical Sciences 141, Springer-Verlag,
  Berlin Heidelberg, 2004, pp.~399--427.

\bibitem{zmiaikou}
David Zmiaikou, \emph{Origamis and permutation groups}, Ph.D. thesis,
  University Paris-Sud, France, 2011.

\bibitem{zorich}
Anton Zorich, \emph{Flat surfaces}, Frontiers in Number Theory, Physics, and
  Geometry I: On Random Matrices, Zeta Functions, and Dynamical Systems
  (P.~Cartier, B.~Julia, P.~Moussa, and P.~Vanhove, eds.), vol.~I,
  Springer-Verlag, 2nd ed., 2006, pp.~437--583.

\end{thebibliography}

\end{document}